\documentclass[11pt,reqno]{amsart}
\usepackage{amsthm,amssymb,amsmath}
\usepackage{textcomp}
\usepackage{xcolor}
\usepackage{enumitem}
\usepackage[colorlinks=true, pdfstartview=FitV, linkcolor=blue, citecolor=blue, urlcolor=blue,pagebackref=false]{hyperref}
\usepackage{ulem}
\usepackage{esint}

\usepackage[T1]{fontenc}
\usepackage{mlmodern}

\newcounter{ascan} 

\usepackage{fix-cm}

\newtheorem{theorem}{Theorem}
\newtheorem*{theorem*}{Theorem}
\newtheorem{lemma}{Lemma}[section]
\newtheorem{corollary}{Corollary}
\newtheorem{proposition}{Proposition}
\newtheorem*{proposition3prime}{Proposition $\ref{prop1}'$}
\newtheorem*{proposition4prime}{Proposition $\ref{prop2}'$}
\newtheorem{claim}{\it Claim}

\newtheorem{definition}{\bf Definition}
\newtheorem*{definition*}{\bf Definition}

\theoremstyle{definition}

\newtheorem{example}{\bf Example}
\newtheorem*{examples}{\bf Examples}
\newtheorem{remark}{Remark}

\newtheorem*{remark*}{Remark}

\newtheorem*{example*}{\bf Example}

\newcommand{\loc}{{\rm loc}}

\newcommand{\const}{{\rm const}}

\newcommand{\sprt}{{\rm sprt\,}}

\numberwithin{equation}{section}

\begin{document}

\title[Strong solutions and discontinuous diffusion coefficients]{Strong solutions of SDEs with critical discontinuities in diffusion coefficients}

\author{S.\,E.\,Boutiah}

\address{Universit\'{e} Laval, D\'{e}partement de math\'{e}matiques et de statistique, Qu\'{e}bec, Canada and Laboratoire de Math\'{e}matiques Appliqu\'{e}es, Universit\'{e} Ferhat Abbas, S\'{e}tif 1, Campus El Bez,  S\'{e}tif, Algeria}

\email{sallah-eddine.boutiah.1@ulaval.ca}

\author{D.\,Kinzebulatov}

\address{Universit\'{e} Laval, D\'{e}partement de math\'{e}matiques et de statistique, Qu\'{e}bec, Canada}

\email{damir.kinzebulatov@mat.ulaval.ca}

\thanks{The research of D.K. is supported by Discovery grant RGPIN-2024-04236 of the Natural Sciences and Engineering Research Council of Canada}

\begin{abstract}
We prove strong existence for It\^{o} SDEs with diffusion coefficients that can introduce strong attraction  to a submanifold. We extend and strengthen the R\"{o}ckner-Zhao approach, which uses Malliavin calculus to establish compactness of the approximating solutions in Wiener-Sobolev space. At least when the diffusion coefficients are sufficiently regular in time, this provides an alternative to Krylov's recent proof of strong existence via analysis of the It\^{o}-Duhamel series.
\end{abstract}

\subjclass[2020]{60H10 (primary), 60H07, 60J60, 35R05 (secondary)}

\keywords{Stochastic differential equations, strong solutions, Malliavin calculus, discontinuous diffusion coefficients, singular drifts}

\maketitle

\fontsize{10.4pt}{4.5mm}\selectfont

\section{Introduction and main result}

\subsection{Introduction}
This paper, which continues \cite{KiM_strong}, concerns strong existence for the It\^{o} SDE
\begin{equation}
\label{sde1}
dX_t=\sigma(t,X_{t})\, dW_t, \quad X_0=x \in \mathbb R^d, 
\end{equation}
where $\{W_s\}_{s \geq 0}$ is a $d$-dimensional Brownian motion ($d \geq 3$) 
with diffusion coefficients that are non-degenerate, i.e.\,$\sigma$ is bounded on $\mathbb R^{1+d}$ and there exists $\kappa>0$ such that, for every $t \in \mathbb R$,
 \begin{equation}
 \label{sigma_1}
\sigma(t,\cdot) \sigma(t,\cdot)^{\top} \geq \kappa I \quad \text{a.e.\,on }\mathbb R^d,
 \end{equation}
but can have critical discontinuities.
By a critical discontinuity of $\sigma$ we mean a discontinuity that can ``trap'' the solution, for example, 
\begin{equation}
\label{hardy_sigma}
\sigma(x) = \frac{1}{\sqrt{1+c}} \left[ I+\left(\sqrt{1+c}-1\right) \frac{x\otimes x}{|x|^2} \right], \quad 0 \neq x \in \mathbb R^d,  \quad c>-1.
\end{equation}
Here $I$ is the $d \times d$ identity matrix, $\sigma(0)=0$.
This is a well known-example of a diffusion coefficient that can make $X_{t}$ arrive at the origin and stay there a.s.\,if $c>d-2$, see \cite{D} and \cite[V.3, Sect.3]{B}. This attraction phenomenon can be seen at the level of the two-sided estimates on the transition density $p(t,x,y)$ of \eqref{sde1} with $\sigma$ given by \eqref{hardy_sigma}, which are no longer purely Gaussian even for $c<d-2$ but involve an extra polynomial factor $|y|^{-\theta} \vee 1$ with power $\theta=\theta(d,c)$ that quantifies the strength of the attraction to the origin introduced by $\sigma$ \cite{MNS}. 
Such attraction mechanisms occur in some physical models, for instance, in some models of turbulent transport the  multiplicative diffusion may cause clustering, collision, coalescence, or sticky behaviour, see \cite{GC,GV}.

See \eqref{nabla_sigma} for the definition of our class of $\sigma$ that includes \eqref{hardy_sigma}.

The study of related problem of weak and strong well-posedness for SDEs with \textit{additive noise and singular drift} has undergone dramatic progress in the past decade; more on this below. 
However, less is known about critical discontinuities of the diffusion coefficients. Compared to singular drifts, the former present a different set of difficulties. In particular, while the problem of weak existence for \eqref{sde1}  is straightforward and is settled using a simple tightness estimate obtained from the BDG inequalities, the problem of strong existence is a different matter, as illustrated by the classical example of Krylov-Zvonkin (Example \ref{kr_zv_example}).
Recently, Krylov introduced an approach to constructing a conditionally unique strong solution of It\^{o} SDE \eqref{sde1} -- that can also contain a singular drift as in \cite{KiM_strong} -- as the sum of the It\^{o}-Duhamel series. Briefly, he defines $X_t$ via the \textit{explicit formula}
\begin{align*}
f(X_t) = P^{0,t}f(x) + \sum_{n=1}^{\infty} \int_{\Delta_n (t)} P^{0, t_n}Q_{k_n}^{t_{n-1},t_n} \cdots Q_{k_1}^{0,t_1}f(x) \, dW_{t_1}^{k_1} \cdots dW_{t_n}^{k_n},
\end{align*}
obtained by iterating It\^{o}'s formula
\begin{equation*}
f(X_t) = P^{0,t}f(x) + \int_0^t Q_k^{r,t}f(X_r) \, dW_r,
\end{equation*}
where $Q_k^{r,t}f(x) := \sigma^{ik}(t,x)\nabla_i P^{r,t}f(x)$ and $P^{r,t}$ is the propagator for SDE \eqref{sde1}. Krylov's proof that this formal series actually converges to a strong solution to SDE \eqref{sde1} uses his much earlier result with Veretennikov \cite{KrV} and his new functional-analytic arguments, see \cite{Kr1} and \cite{Kr2,Kr3} regarding the time-homogeneous case $\sigma=\sigma(x)$ and the time-inhomogeneous case $\sigma=\sigma(t,x)$,  respectively. 

The principal objective of the present work is to provide an alternative to Krylov's proof of strong existence for \eqref{sde1}  based on the approach to constructing strong solutions via compactness on the Wiener-Sobolev space due to \cite{RZ} and its strengthening in \cite{KiM_strong}.
Namely, in \cite{RZ}, Michael R\"{o}ckner and Guohuan Zhao  considered the SDE with  additive noise
\begin{equation}
\label{sde2}
X_{t}=x + \int_r^t b(s,X_s)\, ds+ W_t-W_r, \quad t \in [r,T],
\end{equation}
with drift $b$ in the Ladyzhenskaya-Prodi-Serrin class
\begin{equation}
\label{lps}
|b| \in L^q(\mathbb R,L^p(\mathbb R^d)), \quad \frac{d}{p}+\frac{2}{q} \leq 1, \quad d \leq p \leq \infty,\;\; 2 \leq q \leq \infty.
\end{equation}
Building on earlier works of Bailly-Saussereau \cite{BS}, Meyer-Brandis--Proske \cite{MP}, Mohammed-Nilsen-Proske \cite{MNP} and Rezakhanlou \cite{R}, they constructed a strong solution to \eqref{sde2}, \eqref{lps} by showing that the sequence of strong solutions of the SDEs with regularized drifts is compact in the Wiener-Sobolev space. They furthermore showed that thus constructed strong solution is conditionally pathwise unique -- here and below, conditional on a reasonable occupation estimate -- using Cherny's ``dual to the Yamada-Watanabe principle'' \cite{C}, exploiting the fact that \eqref{sde2} has constant diffusion coefficients. Shortly after, \cite{KiM_strong} strengthened the PDE part of \cite{RZ}, which allowed the authors to establish strong existence and conditional pathwise uniqueness for \eqref{sde2} for form-bounded drifts that, unlike \eqref{lps}, also cover critical-order singularities that can introduce e.g.\,strong attraction phenomena (see Definition \ref{fbd_def} and examples thereafter).

The present paper thus continues \cite{RZ, KiM_strong} and replaces the functional-analytic arguments of Krylov by algebraic calculations involving Malliavin derivatives of solutions of the approximating SDEs, in order to keep track of 
how strong solutions $X^{x,n}_{t}$ of the approximating SDEs 
$$
dX_{t}^{n,x}=\sigma_n(t,X^{n,x}_{t})\, dW_t, \quad X^{n,x}_{0}=x \in \mathbb R^d, 
$$
with smooth $\sigma_n$, driven the same Browninan motion, depend on Brownian paths.
The heavy lifting is still done by PDE estimates (Propositions \ref{prop_grad} and \ref{prop1}), but only as a priori estimates, i.e.\,the PDEs have smooth coefficients, but the constants in the estimates do not depend on the smoothness. Up to the use of the Malliavin calculus, the proofs below arguably use only elementary probabilistic arguments.

\begin{enumerate}

\item[--] At the PDE level we strengthen \cite{RZ}:

\medskip

\begin{enumerate}\item[(a)] Proposition \ref{prop1}, the crucial simplex estimate in the approach of R\"{o}ckner-Zhao needed to verify compactness on the Wiener-Sobolev space, is proved along the lines of \cite{KiM_strong} -- that was the authors' main observation in that paper -- rather than \cite{RZ}. 

\medskip

\item[(b) ] Proposition \ref{prop_grad} provides gradient bounds for solutions of the backward Kolmogorov equation. These bounds are nontrivial because of the possible blow-up phenomena and are proved using the approach of \cite{KS, Ki_Osaka}. 

We also use these  gradient bounds to run Moser iterations in order to construct the Feller propagator.

\medskip

\end{enumerate}

\item[--] At the probabilistic level we extend \cite{RZ}:

\medskip

\begin{enumerate}

\item[(c)] In the case of multiplicative noise, the calculations involving the Malliavin derivatives of solutions of the approximating SDEs are  more involved; it is at this step that we use the additional time-regularity assumption ($\mathbf{H}$). Let us also add that iterating the estimate of Proposition \ref{prop1} in order to prove Proposition \ref{prop2} (i.e.\,the actual verification of compactness on the Wiener-Sobolev space), as was done in \cite{RZ} and \cite{KiM_strong} for singular drifts, is rather difficult since now one needs to iterate It\^{o}'s integrals, and simply using the BDG inequality does not allow to control the convergence of the geometric series (in the one-dimensional case there is, however, a result of Carlen-Kr\'{e}e \cite{CK} that provides the sought control on the iterated It\^{o} integrals). Instead, to obtain Proposition \ref{prop2} from Proposition \ref{prop1}, we appeal to a stochastic Gronwall-type estimate.

\end{enumerate}
\end{enumerate}

\medskip

Our main interest in this work is the spatial singularities of the diffusion coefficients $\sigma$; although our $\sigma$ can depend on time $t$, Krylov in \cite{Kr2} can handle stronger singularities in $t$ (naturally, at the moments when the drift is more regular in space). In the time-homogeneous case, our class of diffusion coefficients is somewhat larger than the class considered by Krylov -- the difference is e.g.\,the Chang-Wilson-Wolff class, see examples below -- at least comparing it with how Krylov's result and proof are written down. Let us also add that Krylov also proves other results related to SDE \eqref{sde1}, such as conditional pathwise uniqueness, that we do not attempt to discuss in the present paper.

The class of time-homogeneous matrix fields $a=\sigma\sigma^{\top}$, where $\sigma=\sigma(x)$ is in our class \eqref{nabla_sigma}, was treated earlier in \cite{KiS_Osaka} in the context of the weak existence/unique Feller semigroup theory of the SDE
$$
dX_t=b(X_t)dt + \sqrt{a(X_t)}dW_t, \quad X_0=x \in \mathbb R^d
$$
having additionally form-bounded drift $b \in \mathbf{F}_\delta$ (Definition \ref{fbd_def}). The present paper is of course quite different, but it uses gradient bounds (Proposition \ref{prop_grad}) and employs the construction of the Feller semigroup via Moser-type iterations (Theorem \ref{thm1}(\textit{ii})) similar to those in \cite{KiS_Osaka}.

Commenting further on the weak solutions (for simplicity, still for time-homogeneous $a$), we mention the detailed weak solution theory of 
$$
dX_t=\sqrt{a(X_t)}dW_t, \quad X_0=x, \quad d \geq 2,
$$
due to Krylov, with bounded, symmetric $a \geq \kappa I$ in the VMO class, that is, satisfying
$$
\omega_a(r) \rightarrow 0 \quad \text{ as } r \downarrow 0,
$$
where
$$
\omega_a(r):=\sup_{x \in \mathbb R^d, 0<\rho \leq r} \frac{1}{|B_\rho|}\int_{B_\rho(x)}\big|a(y)-a_{B_{\rho}(x)}\big|dy, \quad a_B:=\frac{1}{|B|}\int_B a.
$$
The VMO is a broad condition that, unlike the class \eqref{nabla_sigma}, does not impose any conditions on the derivatives of $a$. Although, on the other hand, there are matrices that satisfy \eqref{nabla_sigma} but are not in VMO, see examples before Theorem \ref{thm1}. 
Under these condition on $a$, Krylov proved, among many other results, the uniqueness in law. Recently, he also included a Morrey class singular drift in this SDE, see details in the end of Section \ref{sing_drift_sect}.

In what follows, $\sigma_j:\mathbb R^{1+d} \rightarrow \mathbb R^d$ denotes the $j$-th column of the diffusion matrix $\sigma$.

\subsection{Main result}

In Theorem \ref{thm1} we consider the class of diffusion matrices $\sigma$ having form-bounded derivatives:
\begin{equation}
\label{nabla_sigma}
|\nabla \sigma|:=\sqrt{\sum_{i,j=1}^d |\nabla_i \sigma_j|^2} \in \mathbf{F}_{\nu},
\end{equation}
i.e.\,$|\nabla \sigma| \in L^2_{\loc}(\mathbb R^{1+d})$ and for a.e.\,$t \in \mathbb R$
\begin{equation*}
\||\nabla \sigma(t,\cdot)|\varphi\|_2^2   \leq \nu \|\nabla \varphi\|_2^2+ c_\nu\|\varphi\|_2^2 \quad \forall\,\varphi \in C_c^\infty(\mathbb R^d)
\end{equation*}
for some constant $c_\nu$ -- the value of this constant is not so important for us in this paper, unlike the value of the constant $\nu$ (``form-bound of $|\nabla \sigma|$''), which measures the magnitude of the singularities of $|\nabla \sigma|$. 

\begin{examples} 1.~We have $$|\nabla \sigma| \in L^d(\mathbb R^d) \quad \Rightarrow \quad |\nabla \sigma| \in \mathbf{F}_\nu$$ with $\nu$ that can be chosen arbitrarily small at the expense of increasing $c_\nu$, as can be easily seen by applying H\"{o}lder inequality the Sobolev embedding theorem; we postpone the details until the examples in the next section.

2.~More generally, if $|\nabla \sigma| \in L^{d,\infty}(\mathbb R^d)$ (weak $L^d$ space), then  $|\nabla \sigma| \in \mathbf{F}_\nu$ with $\nu \approx \|\nabla \sigma\|_{L^{d,\infty}(\mathbb R^d)}$. 

3.~A still broader example is provided by matrices $\sigma$ with finite Morrey norm $\|\nabla \sigma\|_{M_{2+\varepsilon}}$ or finite Chang-Wilson-Wolff norm, see \eqref{morrey}, \eqref{cww}.

4.~We now switch to examples of concrete diffusion matrices. Consider the attracting diffusion matrix from the introduction:
$$
\sigma(x)=\frac{1}{\sqrt{1+c}} \left[ I+\left(\sqrt{1+c}-1\right) \frac{x\otimes x}{|x|^2} \right], \quad x \in \mathbb R^d.
$$
A simple calculation using Hardy's inequality $$\langle |x|^{-2},\varphi^2\rangle \leq \frac{4}{(d-2)^2}\langle|\nabla \varphi|^2\rangle, \quad \varphi \in C_c^\infty(\mathbb R^d),$$ shows that $$|\nabla \sigma| \in \mathbf{F}_\nu \quad \text{ with } \nu=8\frac{(d-1)}{(d-2)^2}\biggl(\frac{\sqrt{1+c}-1}{\sqrt{1+c}} \biggr)^2.$$

5.~An example of a rapidly oscillating diffusion matrix from \cite{KiS_Osaka}:
$$\sigma(x)={\rm diag}(1+c\sin (\alpha\log |x|),1,\dots,1)$$
for fixed $0<c<1$. Indeed, the uniform ellipticity condition $\sigma \sigma^{\top} \geq (1-c)^2I$ is satisfied, and
$$
|\nabla \sigma|^2=\frac{c^2 \alpha^2 \cos^2 (\alpha \log|x|)}{|x|^2} \leq \frac{c^2\alpha^2}{|x|^2},
$$
so the Hardy inequality implies the form-boundedness of $|\nabla \sigma|^2$. This example shows that the smallness of the form-bound (take $\alpha$ small) does not imply that the diffusion matrix is close to a constant matrix, as long as one keeps the amplitude of the oscillations $c$ close to $1$. Let us also note that this matrix $\sigma$ is not in the ${\rm VMO}$ class of diffusion coefficients. The class has received a considerable attention  in the literature on SDEs and Kolmogorov equations, more on this in the end of this section.

6.~Rotating anisotropy. Define the rotation the matrix in the first two coordinates: $$
R_\theta:=
\biggl(\begin{array}{cc} \cos\theta & -\sin\theta 
\\ 
\sin\theta & \cos\theta 
\end{array} \biggr) \oplus I_{d-2},
$$
where $I_{d-2}$ is the identity matrix of dimension $d-2$.
Fix a constant diagonal matrix $D:={\rm diag\,}(\lambda_1,\dots,\lambda_d)$, $\lambda_i>0$, and define 
$$
\sigma(x):=R_{\varepsilon\log |x|}D.
$$
It follows that $\sigma\sigma^{\top} \geq (\min_i \lambda^2_i) I$ and $$|\nabla \sigma|^2=\frac{\varepsilon^2(\lambda_1^2+\lambda_2^2)}{|x|^2},$$ i.e.\,by Hardy's inequality, $|\nabla \sigma|$ is form-bounded. 

We could also write $\mathbb R^d=\mathbb R^m \oplus \mathbb R^{d-m}$, $m \geq 3$, and make the logarithm dependent on the first $m$ coordinates of $x$. In this case, to show the form-boundedness of $|\nabla \sigma|$, we can refer to the Hardy inequality in the first $m$ variables. Alternatively, one can verify the Morrey class condition (example 3), albeit at the expense of having much less explicit expression for the form-bound.

See also examples after Theorem \ref{thm2}.
\end{examples}

More generally, we can consider a series of matrices from the previous examples -- properly normalized to ensure the convergence of the series to a uniformly elliptic matrix -- with their points of discontinuity constituting a dense subset of $\mathbb R^d$.

\medskip

We now state the main result of the paper. Set 
\begin{align*}
\rho(x)  \equiv \rho_{\theta}(x) 
:=(1+\theta |x|^2)^{-\frac{d}{2}-1}, \quad x \in \mathbb R^d,
\end{align*}
where $\theta>0$.
We have
\begin{equation}
\label{two_est}
|\nabla \rho| \leq C\sqrt{\theta}\rho, \qquad \frac{|\nabla \rho|^2}{\rho} \leq C\theta\rho.
\end{equation}
We will be applying \eqref{two_est} to $\rho$ with $\theta$ chosen sufficiently small, in order to control the occurrences of the derivatives of $\rho$. 

Let $L^p_\rho(\mathbb R^d)$ denote the $L^p$ space with respect to measure $\rho(x) dx$. Denote by $\rho_y$, $y \in \mathbb R^d$, the translation $\rho_y(x):=\rho(x-y)$.

Let $x \in \mathbb R^d$. We consider SDE
\begin{equation} 
\label{sde1_}
X_{r,t}^x=x + \int_r^t \sigma(s,X^x_{r,s})\, dW_s, \quad r \leq t \leq T
\end{equation}
on a fixed complete probability space  $(\Omega,\{\mathcal F^W_{t}\}_{0 \leq t \leq T},\mathbf{P})$ that supports a Brownian motion $\{W_t\}_{t \geq 0}$, where $\{\mathcal F^W_{t}\}_{0 \leq t \leq T}$ is the filtration of $W$ on $[0,T]$ after completion.

\begin{definition}
\label{strong_def} 
A family $\{X_{r,t}^x \mid 0 \leq r \leq t \leq T\}$ is a family of strong solutions if, for every $r \in [0,T[$  and $x \in \mathbb R^d$, the process $t \mapsto X_{r,t}^x$, $r \in [r,T]$, is continuous and adapted to $\mathcal F^W_{r,t}:=\sigma(W_s-W_r \mid r \leq s \leq t)$ (after completion), and
$$
\mathbf{P}\bigg(X_{r,t}^x=x + \int_r^t \sigma(s,X^x_{r,s})\, dW_s \text{ for every } t \in [r,T]\bigg)=1
$$
\end{definition}

\begin{theorem}
\label{thm1}
Let $d \geq 3$. Assume that the diffusion coefficients $\sigma \in [L^{\infty}(\mathbb R^{1+d})]^{d \times d}$ satisfy the uniform ellipticity condition: for every $t \in \mathbb R$
$$
\sigma(t,\cdot) \sigma(t,\cdot)^{\top} \geq \kappa I \quad \text{a.e.\,on $\mathbb R^{d}$ for some fixed $\kappa>0$}
$$
and have \textit{form-bounded} derivatives, i.e.\,condition \eqref{nabla_sigma} holds.

We also impose the following auxiliary condition on the time-dependence of $\sigma$:

\begin{itemize}

\item[{\rm ($\mathbf{H}$)}] 
$$
\sup_{y \in \mathbb R^d}\|\sigma(t,\cdot)-\sigma(s,\cdot)\|_{L_{\rho_y}^2(\mathbb R^d)} \leq C|t-s|^\gamma \quad t,s \in \mathbb R, |t-s| \leq 1
$$
with the H\"{o}lder continuity exponent $\gamma \in ]0,1]$.
\end{itemize}

Let  $$\sigma_n \in \big[L^\infty(\mathbb R, C_b^\infty(\mathbb R^d))\big]^{d \times d} \cap [C_{\text{unif},\loc}^\gamma(\mathbb R,L^2_\rho(\mathbb R^d))]^{d \times d}$$ be a sequence of bounded matrix fields having bounded  derivatives in the spatial variables, $$\|\sigma_n\|_{L^{\infty}(\mathbb R^{1+d})} \leq \|\sigma\|_{L^{\infty}(\mathbb R^{1+d})}, \quad \sigma_n\sigma_n^{\top} \geq \frac{\kappa}{4}I$$ such that 
$$
\sigma_n \rightarrow \sigma \quad \text{ in $[L^\infty_{\loc}(\mathbb R,L^2_{\rho}(\mathbb R^d))]^{d \times d}$},
$$
and that satisfy \eqref{nabla_sigma} and ${\rm (\mathbf{H})}$ uniformly in $n$, i.e.\,with the same parameters $\nu$, $c_\nu$ and $C$, $\gamma$. It is not difficult to see that the spatial mollification of $\sigma$ produces such $\{\sigma_n\}$, see Appendix \ref{approx_app}. By the classical theory, for each $n$, there exists a pathwise unique strong solution $X_{r,t}^{x,n}$ to the approximating equation
\begin{equation} 
\label{sde1_n}
X_{r,t}^{x,n}=x + \int_r^t \sigma_n(s,X^{x,n}_{r,s})\, dW_s,
\end{equation}
considered on the same (fixed) complete probability space  $(\Omega,\{\mathcal F^W_{t}\}_{0 \leq t \leq T},\mathcal F^W,\mathbf{P})$.

Then, if the form-bound $\nu$ of $|\nabla \sigma|$ in \eqref{nabla_sigma} is smaller than a certain constant that only depends on the dimension $d$, the ellipticity constant $\kappa$ and $\|\sigma\|_{L^\infty(\mathbb R^{1+d})}$, the following are true:

\begin{itemize}
    \item[(i)] There exists a subsequence $\{\sigma_{n_k}\}$  and a continuous random field $X$ such that, for every $N \geq 1$,
$$
\sup_{0\leq r\leq t\le T, |x|\le N} |X_{r,t}^{x,n_k}-X_{r,t}^x| \rightarrow 0 \quad \mathbf{P}-a.s.,
$$
where $X^x_{r,t}$ is a strong solution of SDE \eqref{sde1_}
on $(\Omega,\{\mathcal F^W_t\}_{0 \leq t \leq T},\mathcal F^W,\mathbf{P})$.

\medskip

    \item[(ii)] The operators
    \[
    P_\sigma^{r,t}f(x) := \mathbf{E}[f(X^x_{r,t})], \quad f \in C_\infty, \quad r,t \in [0,T], r<t,
    \]
    constitute a backward Feller propagator on $C_\infty$, 
and, furthermore, if $\{\sigma_n\}$ additionally satisfies
$$
\nabla a_n \rightarrow \nabla a \quad \text{ in } L_{\loc}^2(\mathbb R,L^2_\rho(\mathbb R^d)),
$$
where $a=\sigma\sigma^{\top}$, $a_n=\sigma_n(\sigma_n)^{\top}$ (once again, mollifying $\sigma$ in the spatial variables produces this convergence, see Appendix \ref{approx_app}), then
$$
P_\sigma^{r,t}=s\mbox{-}C_\infty\mbox{-}\lim_n P_{\sigma_n}^{r,t} \quad \text{loc.\,uniformly in $r,t \in [0,T]$, $r<t$.}
$$
There is no need to pass to a subsequence of $\{\sigma_n\}$. In this sense, the law of $X^x_{r,t}$ does not depend on the choice of the subsequence in {\rm (\textit{i})}.
\end{itemize}
\end{theorem}

\begin{example} 
\label{kr_zv_example}

The celebrated Krylov-Zvonkin's counterexample to strong existence shows that the form-bound $\nu$ of $|\nabla \sigma|$ in \eqref{nabla_sigma} cannot be greater than $12$ in dimension $d=4$; this demonstrates that the smallness of $\nu$ assumed in Theorem \ref{thm1} is necessary for strong existence (or rather than $\nu$ cannot be too large). Namely, in $\mathbb R^4$, consider SDE 
\begin{equation}
\label{sde_ctr}
dX_t=\sigma(X_t)dW_t, \quad X_0=0,
\end{equation}
with diffusion matrix
$$
\sigma(x)=\frac{1}{|x|}\left(
\begin{array}{cccc}
x_1&-x_2&-x_3&-x_4\\
x_2&x_1&-x_4&x_3\\
x_3&x_4&x_1&-x_2\\
x_4&-x_3&x_2&x_1
\end{array}
\right),
$$
define $\sigma(0)$ arbitrarily.
Then $\sigma\sigma^{\top}=I$, i.e.\,in this case the ellipticity constant $\kappa=1$. On the other hand, one shows, using L\'{e}vy's characterization of Brownian motion, that $$B_t:=\int_0^t \sigma^{\top}(W_s)dW_s$$ is a Brownian motion. We have $W_t=\int_0^t \sigma(W_s)dB_s$, hence $t \mapsto (B_t,W_t)$ is a weak solution to SDE \eqref{sde_ctr}. At the same time, $(-B_t,W_t)$ is also a weak solution to the same SDE. Therefore, pathwise uniqueness fails. One can prove that every solution to this SDE is Brownian, so there is uniqueness in law. If we had strong existence, then the uniqueness in law would imply pathwise uniqueness, a contradiction. See details in  \cite{KZ}.

We can evaluate the form-bound $\nu$ of $|\nabla \sigma|$. We have
$$
|\nabla \sigma|^2=\sum_{i,j,k=1}^4 |\nabla_k \sigma_{ij}|^2=\frac{12}{|x|^2} 
$$
since each entry of the form $\frac{x_l}{|x|}$ in $\sigma$ occurs exactly four times, and $\sum_{l=1}^4 |\nabla \frac{x_l}{|x|}|^2=\frac{3}{|x|^2}$. Now, invoking the Hardy inequality on $\mathbb R^4$ with sharp constant, i.e.\,$\||x|^{-1}\varphi\|_2^2 \leq \|\nabla \varphi\|_2^2$, $\varphi \in C_c^{\infty}(\mathbb R^d)$, we obtain that the form-bound $\nu$ of $|\nabla \sigma|$ is exactly $12$. 
\end{example}

\subsection{About the proof of Theorem \ref{thm1} and its possible extension} 
\label{proof_struct_sect}
1.~In the next Theorem \ref{thm2} we consider the SDE with additive noise  and singular (form-bounded) drift. The task of combining such drift and the diffusion coefficients as in Theorem \ref{thm1} now seems to be achievable with some additional work, given that the principal estimates -- Proposition \ref{prop1} and \cite[Prop.\,1]{KiM_strong} -- are proved within a compatible framework. We plan to write down the details elsewhere.

2. The proof of Theorem \ref{thm1}(i) follows the compactness approach of R\"{o}ckner and Zhao. Proposition \ref{prop1}, together with the stochastic Gronwall-type estimate, yields Proposition \ref{prop2}, which provides uniform estimates for the spatial and Malliavin derivatives of the approximating solutions. These estimates verify the hypotheses of the Wiener-Sobolev compactness criterion in Lemma \ref{lem_compact}. A diagonal argument then gives, after passing to a subsequence, convergence on a fixed countable dense subset of $\Delta_2(T)\times\mathbb R^d$.

The increment estimates derived from Proposition \ref{prop2}, together with
the Kolmogorov-Chentsov theorem, yield continuity of the limiting
random field and stochastic equicontinuity of the approximating
random fields. A finite-net argument and the Borel-Cantelli
lemma then upgrades the convergence on the dense set to a.s.
local uniform convergence on $\Delta_2(T)\times\mathbb R^d$.

It remains to pass to the limit in the approximating SDEs.
Proposition \ref{prop_grad} yields the occupation estimate used in Lemma \ref{lem_conv}. The construction of the limiting occupation measure then allows us
to prove convergence of the stochastic integrals and to show that the limit found at the previous step is a strong solution.

At this last step, we could instead appeal to Krylov's classical
occupation-time estimate, which applies to bounded uniformly
nondegenerate diffusion coefficients, i.e.\,without requiring the
form-boundedness condition on their derivatives. Since our eventual
goal is to add a singular drift to SDE \eqref{sde1}, we use the longer
argument based on Proposition \ref{prop_grad} which is more amenable for extensions in our context.

The proof of Theorem \ref{thm1}(\textit{ii}) uses PDE arguments only. The gradient bounds of
Proposition \ref{prop2} allow us to run Moser iterations for the differences
of the approximating Kolmogorov equations. This yields strong
convergence in $C_\infty$ of the corresponding Feller propagators,
locally uniformly in the time parameters, and gives the backward
Feller propagator.

\subsection*{Acknowledgements} We are deeply grateful to Kodjo Rapha\"{e}l Madou for valuable discussions and for his participation in the early stages of this work.

\bigskip

\section{Remarks on additive noise and singular drift} 
\label{sing_drift_sect}

Let $d \geq 3$.
Consider SDE 
\begin{equation}
\label{sde_b}
dX_t=b(t,X_t)dt + dW_t, \quad X_0=x \in \mathbb R^d,
\end{equation}
having form-bounded drift $b$:

\begin{definition}
\label{fbd_def}
A Borel measurable vector field $b  :\mathbb R^{1+d} \rightarrow \mathbb R^d$ with $|b| \in L^2_{\loc}(\mathbb R^{1+d})$ is said to be form-bounded  if 
there exist constant $\delta>0$ such that the following quadratic form inequality holds for a.e.\,$t \in \mathbb R$:
\begin{align}
\label{fbb_inhom}
\|b(t,\cdot)\varphi\|_2^2   \leq \delta \|\nabla \varphi\|_2^2 + g_\delta(t)\|\varphi\|_2^2 \qquad \forall\,\varphi \in C_c^\infty(\mathbb R^d)
\end{align}
for some function $g_\delta \in L_{\loc}^1(\mathbb R)$. 

This is written as $b\in \mathbf{F}_{\delta,g}$.
\end{definition}

The function $g_\delta$ is required to handle drifts that can be locally unbounded in time, cf.\,examples below.

\begin{examples} We mention the following examples/sub-classes of form-bounded drifts defined in elementary terms; there are some repetitions with the examples before Theorem \ref{thm1}.

1.~The ``spatial endpoint'' of the Ladyzhenskaya-Prodi-Serrin class \eqref{lps}:
$$
b \in L^\infty([0,\infty[,L^d(\mathbb R^d)+L^\infty(\mathbb R^d)) \quad \Rightarrow \quad b \in \mathbf{F}_{\delta,g_\delta}
$$
for appropriate $\delta$ and bounded $g_\delta$; let us omit $\mathbb R^d$ in $L^p(\mathbb R^d)$ from now on.
Indeed, writing $b=b_1+b_2$, where $b_1 \in L^\infty(\mathbb R,L^d)$, $b_2 \in L^\infty(\mathbb R,L^\infty)$, we have for a.e.\,$t \in \mathbb R$,
\begin{align*}
\|b(t,\cdot)\varphi\|_2^2 
& \leq (1+\varepsilon)\|b_1(t,\cdot)\|_{L^d}^2 \|\varphi\|_{L^\frac{2d}{d-2}}^2 + (1+\varepsilon^{-1})\|b_2(t,\cdot)\|_{L^\infty}^2 \|\varphi\|_{L^2}^2 \qquad (\varepsilon>0)\\
& (\text{apply the Sobolev embedding}) \\
& \leq C_S (1+\varepsilon) \|b_1(t,\cdot)\|_{L^d}^2 \|\nabla \varphi\|_{L^2}^2 + (1+\varepsilon^{-1})\|b_2(t,\cdot)\|_{L^\infty}^2 \|\varphi\|_{L^2}^2,
\end{align*}
so $b \in \mathbf{F}_{\delta,g_\delta}$ with $\delta:=\sup_{t \in \mathbb R}C_S (1+\varepsilon) \|b_1(t,\cdot)\|_{L^d}^2$ and $g_\delta(t)=(1+\varepsilon^{-1})\|b_2(t,\cdot)\|_{L^\infty}^2$.

Let us note that if $b$ is time-homogeneous, then, by considering cutoff of $b$, we can find, for every $\varepsilon>0$, vector fields $b_1$ and $b_2$ as above such that $\|b_1\|_{L^d}<\varepsilon$, and so the form-bound $\delta$ of $b$ is, in fact, arbitrarily small; this comes at the expense of having $\|b_2\|_{L^\infty}$ large. 

2.~Recall that a Borel measurable function $h:\mathbb R^d \rightarrow \mathbb R$ is in $L^{d,\infty}$ if $\|h\|_{L^{d,\infty}}:=\sup_{s>0}s|\{x \in \mathbb R^d: |h(x)|>s\}|^{1/d}<\infty$. By the Strichartz inequality with sharp constants \cite[Prop.~2.5, 2.6, Cor.~2.9]{KPS}, if, say, $b=b(x)$ with $|b| \in L^{d,\infty}$, then 
\begin{align*}
b \in \mathbf{F}_{\delta_1,0} \quad \text{ with } \sqrt{\delta_1}&=\||b| (\lambda - \Delta)^{-\frac{1}{2}} \|_{2 \rightarrow 2} \\ & \leq
\|b\|_{d,\infty} \Omega_d^{-\frac{1}{d}} \||x|^{-1} (\lambda - \Delta)^{-\frac{1}{2}} \|_{2 \rightarrow 2} \\ & \leq \|b\|_{d,\infty} \Omega_d^{-\frac{1}{d}}2^{-1}\frac{\Gamma\bigl(\frac{d-2}{4} \bigr)}{\Gamma\bigl(\frac{d+2}{4} \bigr)}=\|b\|_{d,\infty} \Omega_d^{-\frac{1}{d}} \frac{2}{d-2},
\end{align*}
where $\Omega_d=\pi^{\frac{d}{2}}\Gamma(\frac{d}{2}+1)$ is the volume of the unit ball in $\mathbb R^d$.

3. Extending the previous example, we obtain that
$$
|b| \in L^\infty(\mathbb R,L^{d,\infty}(\mathbb R^d)) \quad \Rightarrow \quad b \in \mathbf{F}_{\delta,0},
$$
with $\delta \approx \sup_{t \in \mathbb R}\|b(t,\cdot)\|^2_{L^{d,\infty}(\mathbb R^d)}$.
This class contains the class $|b| \in L^\infty(\mathbb R,L^{d}(\mathbb R^d)) $ from the first example. The set-theoretic difference between the latter and $L^\infty(\mathbb R,L^{d,\infty}(\mathbb R^d))$ contains e.g.\,the attracting drift
\begin{equation}
\label{hardy_drift}
b(x):=-\frac{d-2}{2}\sqrt{\delta}\frac{x}{|x|^2},
\end{equation}
that, by the Hardy inequality or by the calculation in 2, is in $\mathbf{F}_{\delta,0}$. This drift provides a counterexample to the weak existence if the ``strength of attraction to the origin'' $\delta$ is too large, see Section \ref{weak_sol_sect} for details. 

The counterexample to weak existence discussed in that section already shows that there is an important distinction between the class $|b| \in L^\infty(\mathbb R,L^d(\mathbb R^d))$ and the class of form-bounded drifts $b \in \mathbf{F}_\delta$. Namely, the former does not cover such strong attraction/blow-up phenomena, despite being critical under the diffusive scaling.

4.~The weighted Hardy inequality of Hoffmann-Ostenhof--Laptev \cite{HL}. Fix $$0 \leq \Phi \in L^s(S^{d-1})\quad \text{ for some } s \geq \frac{2(d-2)^2}{2(d-1)}+1,$$ where $S^{d-1}$ is the unit sphere in $\mathbb R^d$. If
$$
|b(x)|^2 \leq \delta \frac{(d-2)^2}{4} c\frac{\Phi(x/|x|)}{|x|^2}, \qquad \text{where $c:=\frac{|S^{d-1}|^{\frac{1}{s}}}{\|\Phi\|_{L^s(S^{d-1})}}$},
$$
then $b \in \mathbf{F}_\delta$ with $g_\delta=0$.

5.\,Let $K \in \mathbf{F}_\kappa(\mathbb R^d)$. Then the many-particle drift $b=(b_1,\dots,b_N):\mathbb R^{dN} \rightarrow \mathbb R^{dN}$ with components defined by
\begin{equation*}
b_i(x^1,\dots,x^N):=\frac{1}{N}\sum_{j=1, j \neq i}^N K(x^i-x^j), \quad x^i,x^j \in \mathbb R^d
\end{equation*}
is in $\mathbf{F}_\delta(\mathbb R^{dN})$ with $\delta=\frac{(N-1)^2}{N^2}\kappa.$

6.~Fix a small $\varepsilon>0$. If $\sup_{t \in \mathbb R}\|b(t,\cdot)\|_{M_{2+\varepsilon}}<\infty$, then $b \in \mathbf{F}_\delta$ with $\delta$ proportional to the supremum, and $g_\delta=0$. Here
\begin{equation}
\label{morrey}
\|f\|_{M_{2+\varepsilon}}:=\sup_{r>0,x \in \mathbb R^d} r \biggl(\frac{1}{r^d}\int_{B_r(x)}|f(x)|^{2+\varepsilon} dx \biggr)^{\frac{1}{2+\varepsilon}}
\end{equation}
is the Morrey norm. The inclusion is a consequence of the Adams inequality \cite[Theorem 7.3]{A}. 

We note that, for each $0<\varepsilon<\varepsilon_0$, there exist vector fields $b=b(x)$ in the Morrey class $M_{2+\varepsilon}$ such that $|b| \not \in L_{\loc}^{2+\varepsilon_0}(\mathbb R^{d}).$

7.  Chang-Wilson-Wolff \cite{CWW} identified a larger than $\cup_{\varepsilon>0}M_{2+\varepsilon}$ sub-class of $\mathbf{F}_\delta \cap \{b=b(x)\}$. That is, they demonstrated that if $|b| \in L^2_{\loc}(\mathbb R^d)$ and
\begin{equation}
\label{cww}
\sup_{r>0, x \in \mathbb R^d} \frac{1}{|B_r|}\int_{B_r(x)} |b(y)|^2\, r^2 \xi\big(|b(y)|^2\,r^2 \big) dy<\infty,
\end{equation}
where
$\xi:\mathbb R_+ \rightarrow [1,\infty[$ is a fixed increasing function such that
$$
\int_1^\infty \frac{ds}{s\xi(s)}<\infty,
$$
then $b \in \mathbf{F}_\delta$ with appropriate $\delta$.
For instance, one can take $\xi(s)=1+(\log^+s)^{1+\epsilon}$ or $\xi(s)=1+\log^+s (\log\log^+s)^{1+\epsilon}$ for some $\epsilon>0$ (but not $\xi(s)=1+\log^+s$).

8.~In Theorem \ref{thm2} we impose a more restrictive condition on $g_\delta$, namely, it must be in $L^{1+\varepsilon}(\mathbb R)$ for some arbitrarily small but fixed $\varepsilon>0$. For instance,
$$
|b(t,x)| \leq |t-t_0|^{-\frac{1}{2}+\varepsilon_1} \quad \text{ for some $\varepsilon_1>0$, for $t_0$ fixed,}
$$
satisfies this requirement. Note that, due to this $\varepsilon>0$, the Ladyzhenskaya-Prodi-Serrin class of drifts \eqref{lps} handled by \cite{RZ} is not entirely contained in the class of form-bounded drifts of Theorem \ref{thm2} unless $q=\infty$ in \eqref{lps}, i.e.\,one focuses on the maximal admissible spatial singularities of the drift.
\end{examples}

The following theorem strengthens the result in \cite{RZ} and yields strong existence/conditional uniqueness for drifts that can introduce strong attraction phenomena, e.g.\,as in example 3.

The authors of \cite{RZ}, in turn, strengthened Krylov-R\"{o}ckner \cite{KR}, i.e.\,replaced the sub-critical Ladyzhenskaya-Prodi-Serrin class ($<1$) with the critical one ($=1$), see \eqref{lps}, with the exception of the pathwise uniqueness since, in the setting of \cite{KR}, one can prove unconditional pathwise uniqueness, see Zhang \cite{Z}. 

Let us also add that the strong solutions to SDE \eqref{sde_b} with $b$ in the Ladyzhenskaya-Prodi-Serrin class \eqref{lps}, but only for a.e.\,initial point $x \in \mathbb R^d$ -- possibly excluding the singular set of the drift -- were constructed earlier in \cite{BFGM} via the regularity theory of the stochastic transport/continuity equations.\footnote{We note that the authors of \cite{KSS} extended  the relevant parts of \cite{BFGM} to time-homogeneous form-bounded drifts, so, strictly speaking, the extension of \cite{RZ} to form-bounded drifts in \cite{KiM_strong} was consistent with what was done earlier.}

In Theorem \ref{thm2} we work on a finite time interval, so, without loss of generality, we impose from now on the global $L^1$ condition on $g_\delta$.

\begin{theorem}[{\cite{KiM_strong}}, but with a few technical conditions on the drift removed, see remark below]
\label{thm2}
Let $d \ge 3$, $T>0$. Assume that $b \in \mathbf{F}_{\delta, g}$ and that there exists $\varepsilon>0$ such that the function $g_\delta$ from \eqref{fbb_inhom} additionally satisfies
$$
g_\delta \in L^{1+\varepsilon}(\mathbb{R}).
$$
Let $\{b_n\} \subset (C^\infty \cap L^\infty)(\mathbb R^{1+d},\mathbb R^d)$ be a sequence of vector fields such that
$$
b_n \rightarrow b \quad \text{ in } [L_{\loc}^2(\mathbb R,L^2_{\rho}(\mathbb R^d))]^d,
$$
are form-bounded with the same form-bound $\delta$ having $\sup_n\|g_{\delta,n}\|_{L^{1+\varepsilon}(\mathbb R)}<\infty$. (For example, the standard mollifier in $\mathbb R^{1+d}$ applied to $b$ yields $b_n$, see Appendix \ref{Conv_b_n} for details.) 

If the form-bound $\delta$ of drift $b$ is sufficiently small, then 
 the following are true:

\begin{itemize}
    \item[(i)] There exists a subsequence $\{b_{n_k}\}$ such that for any initial time $r \in [0,T[$ and every initial point $x \in \mathbb{R}^d$ the approximating strong solutions $\{X^{x,n_k}_{r,t}\}_{k=1}^\infty$ to
\begin{equation*}
X_{r,t}^{x,n}=x+\int_r^t b_n(s,X_{r,s}^{x,n})ds + W_t-W_r, 
\end{equation*}
considered on a fixed complete probability space $(\Omega,\{\mathcal F^W_t\}_{0 \leq t \leq T},\mathcal F^W,\mathbf{P})$,
	converge for every $t \in [r,T]$ $\mathbf{P}$-a.s.\,to a strong solution $X^x_{r,t}$ of SDE 
	\begin{equation*}
X_{r,t}^{x}=x+\int_r^t b(s,X_{r,s}^{x})ds + W_t-W_r, \quad t \in [r,T]. 
\end{equation*}

\item[(ii)] 
The constructed strong solution $X^x_{r,t}$ satisfies the following Krylov-type bound. For every $q \in(d,\delta^{-1/2})$ and any vector field $f \in \mathbf{F}_{\beta}$, $\beta < \infty$, there exists a constant $C$, depending only on $d$, $q$, $\delta$, $g_\delta$, $\beta$, $g_\beta$ and $T$, such that
    \begin{equation} \label{kr1}
        \mathbf{E}\int_r^T
        \big|fh\left(\tau,X_{r,\tau}^x\right) \big| d\tau  \le C\left\| f |h|^{q/2} \right\|_{L^2([r,T],L^2_{\rho}(\mathbb{R}^d))}^{2/q} \quad  \text{for all }   h \in C_c \left([r,T]\times\mathbb{R}^d\right).
\end{equation}

\item[(iii)] 
The constructed solution $X_{r,t}^x$ is pathwise unique among strong solutions to \eqref{sde_b} that satisfy \eqref{kr1} for some
$q \in (d,\delta^{-1/2})$, with $|f|=1$ and with $|f|=|b|$.

    \item[(iv)] The operators
    \begin{equation}
\label{T_b}
    T_b^{r,t}f(x) := \mathbf{E}[f(X^x_{r,t})], \quad f \in C_\infty,
    \end{equation}
    constitute a backward Feller propagator on $C_\infty$, 
and, furthermore, we have
$$
T_b^{r,t}=s\mbox{-}C_\infty\mbox{-}\lim_n T_{b_n}^{r,t} \quad \text{loc.\,uniformly in $r,t \in [0,T]$}
$$
regardless of the choice of the sequence $\{b_n\}$ as long as it satisfies the preceeding uniform form-boundedness assumptions. In this sense, the law of $X^x_{r,t}$ does not depend on the choice of the subsequence in {\rm (\textit{i})}.

\end{itemize}

\end{theorem}

\begin{remark}
Assertion (\textit{iv}) was proved in \cite{Ki_Osaka}. Its proof uses the analogue of the gradient bounds of Proposition \ref{prop_grad} on the solution to Cauchy problem for the Kolmogorov equation $(\partial_t - \frac{1}{2}\Delta + b \cdot \nabla u)=F$, $u|_{t=0}=f$.

The proofs of the Krylov-type bound (\textit{ii})  and the conditional pathwise uniqueness (\textit{iii}) remain unchanged. The Krylov-type bound was proved in \cite{KM_JDE}. The proof of the conditional pathwise uniqueness follows the idea of R\"{o}ckner and Zhao and uses:

-- Chernyi's theorem \cite{C}, which can be viewed as a dual of the Yamada-Watanabe principle; its application here exploits the fact that the diffusion coefficients in \eqref{sde_b} are constant; 

-- conditional weak uniqueness for \eqref{sde_b} established in \cite{KM_JDE}. 

The proof of the strong existence (\textit{i}) repeats the proof in \cite{KiM_strong} with one essential addition. The simplex estimate \cite[Prop.\,1]{KiM_strong} ($=$ the drift analogue of Proposition \ref{prop1} in the proof of Theorem \ref{thm1}) imposes the compact support assumption on the drift $b$ (i.e.\,function $f_i$ in the estimate below) and assumes that the function $g_\delta$ in the form-boundedness condition on $b$ is identically zero\footnote{These simplifying assumptions do not affect the class of admissible local singularities of $b$, which was the authors' main interest in \cite{KiM_strong}.}. In Section \ref{thm2_proof_sect} we state and prove \cite[Prop.\,1]{KiM_strong} in complete generality, as is needed to obtain assertion (\textit{i}).
\end{remark}

\subsection{Weak solutions}
\label{weak_sol_sect}

 If one is interested in the weak solution theory of SDE \eqref{sde_b}, then one can go beyond the setting of Theorem \ref{thm2} in several directions. Let us focus for simplicity on $b=b(x)$, although the results cited below are, in general, valid for time-inhomogeneous drifts as well; we also ignore the behaviour of the drift at infinity and focus on admissible local singularities.

1.~Theorem \ref{thm2} requires the form-bound of $b$ to be smaller than a certain dimension-dependent constant. One can, however, prove weak existence/strong Markov property (in fact, one obtains a strongly continuous Feller semigroup), plus some conditional weak uniqueness results, for $b \in \mathbf{F}_\delta$, $\delta<4$ \cite{KiS_sharp, KiS_JDE,KiV}. 
One thus arrives at completely dimension-independent conditions on the drift, which is important for applications to particle systems \cite{Ki_multi}. 
The value $\delta=4$ corresponds to the first critical threshold in the following table, that is, until the PDE methods work. 
Specifically, consider SDE
$$
dX_t=-\frac{d-2}{2}\sqrt{\delta}\frac{X_t}{|X_t|^2}dt + \sqrt{2}W_t, \quad X_0=0,
$$
with drift set to zero at the origin.
We have the following tetrachotomy:

\begin{table}[h]
\caption{Weak solvability for $b(x)=-\frac{d-2}{2}\sqrt{\delta}\frac{x}{|x|^2}$}
\label{fig1}
\centering
\begin{tabular}{|c|c|}
\hline
Values of $\delta$ & Proof of weak existence \\
\hline
$[0,4[$ & Via PDEs with \eqref{hardy_drift} viewed as a form-bounded drift \cite{KiS_sharp,KiS_JDE}\\
\hline
$[0,4\big(\frac{d-1}{d-2}\big)^2[$ & Tightness via SDE exploiting special form of \eqref{hardy_drift} \cite{FJ} \\
\hline
$[4\big(\frac{d-1}{d-2}\big)^2,4\big(\frac{d}{d-2}\big)^2[$ & Via $X_t \mapsto |X_t|^2X_t$ or via distributional drift, using special form of \eqref{hardy_drift} \cite{FJ, T} \\
\hline
 $[4\big(\frac{d}{d-2}\big)^2,\infty[$ & No weak solution \\
\hline
\end{tabular}
\end{table}

The nonexistence of a weak solution is due to the strong attraction to the origin, see the detailed discussion of the counterexample for $\delta \geq 4\big(\frac{d}{d-2}\big)^2$ in \cite{BFGM,FJ}; in fact, even started away from the origin, if $\delta$ surpasses the threshold $4\big(\frac{d}{d-2}\big)^2$, the solution arrives at the origin with positive probability and stays there; the proof is an application of the theory of Bessel processes.

2.~There is weak solution theory to SDE 
$$
dX_t=b(X_t)dt + \sqrt{2}W_t, \quad X_0=x \in \mathbb R^d, \quad d \geq 2,
$$
with the same level of detail as in Theorem \ref{thm2} for a larger class of weakly form-bounded drifts \cite{KiS_brownian, Ki_NoDEA, KY}: $|b| \in L^1_{\loc}(\mathbb R^d)$ and
$$
\langle |b|,\varphi^2\rangle \leq \delta \langle |(\lambda-\Delta)^{\frac{1}{4}}\varphi|^2\rangle \text{ for all } \varphi \in C_c^\infty(\mathbb R^d),
$$
for some sufficiently small $\delta>0$, for some $\lambda>0$,
or for the broad subclass of drifts having sufficiently small Morrey norm
$$
\|b\|_{M_{1+\varepsilon}}=\sup_{x \in \mathbb R^d, r>0} r\biggl(\frac{1}{|B_r(x)|}\int_{B_r(x)}|b|^{1+\varepsilon} \biggr)^{\frac{1}{1+\varepsilon}}
$$
for some fixed $0<\varepsilon \leq 1$ (the smaller it is, the broader is the class). For example, the many-particle drift in the finite-particle approximation of the celebrated Keller-Segel model of chemotaxis, i.e.\,$b=(b^1,\dots,b^N):\mathbb R^{2N} \rightarrow \mathbb R^{2N}$,
$$
b^i(x)=\frac{\kappa}{N}\sum_{j=1, j \neq i}^N \frac{x^i-x^j}{|x^i-x^j|^2}, \quad x=(x^1,\dots,x^N) \in \mathbb R^{2N},
$$ 
belongs to these classes, but does not belong to $\mathbf{F}_\delta$ since its square is not locally summable, or, one could also say, because there is no non-trivial Hardy inequality in $\mathbb R^2$.  In this setting, additionally having discontinuous diffusion coefficients seems to be impossible, at least within the existing approach of \cite{KiS_brownian, Ki_NoDEA, KY}.

3.~At the other end of the range (or, perhaps, convex body) of possible results, there are recent results of Krylov on detailed weak solution theory of 
$$
dX_t=b(X_t)dt+\sqrt{a(X_t)}dW_t, \quad X_0=x, \quad d \geq 2,
$$
with bounded, symmetric $a \geq \kappa I$ in the VMO class (see the introduction) and drift satisfying, for some $q_b>d_0(d,\kappa)>\frac{d}{2}$, 
$$
b \in M_{q_b} \text{ with sufficiently small Morrey norm}.
$$
Under these assumptions, Krylov proves existence and conditional weak uniqueness in Krylov class ($=$reasonable occupation time estimate). The conditions on the drift $b$ are more restrictive than in the results described in the previous paragraph. Indeed, to deal with the VMO coefficients (in presence of a singular drift) one needs to establish regularity of the second derivatives of solutions of the corresponding Kolmogorov backward equation; see  \cite{BKRS, Do, Kr3} and the references therein for the relevant regularity theory. The classes of form-bounded or weakly form-bounded drifts, on the other hand, destroy any non-trivial estimate on the second derivatives, but these are not needed when the diffusion coefficients in the SDE are constant or satisfy \eqref{nabla_sigma}; in small dimensions the situation is better, see more detailed discussion in \cite{Ki_survey}.

\bigskip

\section{Notations and auxiliary results}

\label{notations_sect}

\subsection{Notations}
\begin{itemize}
\item Let $\mathcal B(X,Y)$ denote the space of bounded linear operators between Banach spaces $X \rightarrow Y$,  endowed with the operator norm $||\cdot ||_{X\to Y}$, and set $\mathcal B(X):=\mathcal B(X,X)$.  
\item We write $T=s\mbox{-} Y \mbox{-}\lim_n T_n$, for $T, T_n \in \mathcal B(X,Y)$ if 
$$
\lim_n\|Tf- T_nf\|_Y=0 \quad \text{ for every $f \in X$}.
$$ 

\item Let $[X]^d$ and $[X]^{d \times m}$ denote the spaces of $d$-vectors and $d \times m$-matrices with entries in $X$.
  
\item Put
$$
\langle f, g \rangle = \langle f g \rangle := \int_{\mathbb{R}^d} f(x) g(x) dx.
$$
Given a weight $\rho$ on $\mathbb R^d$, we write
$$
\langle f,g\rangle_\rho=\langle fg\rangle_\rho=\int_{\mathbb R^d}f(x)g(x)\rho(x)dx.
$$

\item Given a matrix $A=(a_{ij}) \in \mathbb R^{d \times d}$, we denote by $|A|$ its Euclidean norm, i.e.\,$|A|^2:=\sum_{i,j=1}^d a_{ij}^2$.

\item Put $a \wedge b:=\min\{a,b\}$, $a \vee b:=\max\{a,b\}$.

\item Set $ \nabla_i := \partial_{x_i}$.
\item Let $L^p = L^p(\mathbb{R}^d, dx)$, $W^{1,p} = W^{1,p}(\mathbb{R}^d, dx)$ denote the usual Lebesgue and Sobolev spaces, respectively.

\item $C_\infty:=\{f \in C_b(\mathbb R^d) \mid \lim_{|x| \rightarrow \infty f(x)=0}\}$ endowed with the $\sup$-norm. 

\item For any $p>1$, we use $p^\prime$ to denote its conjugate $p/(p- 1)$; if $p=1$, then $p'=\infty$.

\item For a given matrix field $a=(a_{ij})_{i,j=1}^d:\mathbb R^d \rightarrow \mathbb R^{d \times d}$, we define its row divergence vector field $\nabla a:\mathbb R^d \rightarrow \mathbb R^d$ by the formula
$$
(\nabla a)_j:=\sum_{i=1}^d\nabla_i a_{ij}.
$$
\item Given $T_0<T_1$, define the temporal simplex
    $$
    \Delta_n(T_0,T_1) :=   \left\{ (t_1,\ldots,t_n)\in\mathbb{R}^n: T_0 \leq t_1\leq\cdots\leq t_n\leq T_1  \right\},
    \qquad
    \Delta_n(T):=\Delta_n(0,T).
    $$
\item Denote by $\mathbf{E}_{\mathcal{F}_t}$ the conditional expectation with respect to the $\sigma$-algebra $\mathcal{F}_t$.

\item Recall that a family of operators  $\{P^{r,t}\}_{0 \leq r \leq t \leq T} \subset \mathcal B(C_\infty)$ is a backward  Feller propagator if

$1^\circ$)  $P^{r,s}P^{s,t}=P^{r,t}$ for every $s \in [r,t]$ (reproduction property).

$2^\circ$) $P^{r,r}={\rm Id}$ (identity operator),  

$3^\circ$) $\|P^{r,t}f\|_\infty \leq \|f\|_\infty$, $P^{r,t}[C_\infty^+] \subset C_\infty^+,$

$4^\circ$) for every $f \in C_\infty$, the function $(r,t) \mapsto P^{r,t}f$, $0 \leq r \leq t \leq T$, is continuous in $C_\infty$. 

The last property also follows from the separate continuity: for each fixed $r$ the function $[r,T] \ni t \mapsto P^{r,t}f$ is continuous, and for each fixed $t$ the function $[0,t] \ni r \mapsto P^{r,t}f$ is continuous \cite[Theorem 2.1]{GvC}.

\end{itemize}

\medskip

\subsection{Compactness for $L^2$ random vector fields}

\begin{lemma}[{\cite[Lemma 3.1]{RZ}}]
\label{lem_compact}
Let $U \subset \mathbb{R}^d$ be a bounded smooth domain. Let $\{F_n=F_n(x,\omega)\}$ be a sequence of random fields in $L^2(U \times \Omega, dx \times \mathbf{P})$. 
Assume that
\begin{align}
\label{a.1} 
& \quad \qquad \mathbf{E}\int_U \big(|F_n(x)|^2+|\nabla_x F_n(x)|^2 \big)dx   \leq K \\
\label{a.2} 
&\qquad \qquad \mathbf{E} \int_{U} \int_0^T |D_s F_n(x)|^2 \, ds \, dx  \leq K, \\
\label{a.3} 
&\mathbf{E} \int_{U} \int_0^T \int_0^T \frac{|D_s F_n(x) - D_{s'} F_n(x)|^2}{|s - s'|^{1 + 2\beta}} \, ds \, ds' \, dx  \leq K
\end{align}
for some $\beta>0$ and $K<\infty$ independent of $n$.
Then the sequence $\{F_n\}$ is relatively compact in $L^2(U \times \Omega)$.
\end{lemma}

We refer to \cite{RZ} for the discussion of the history of this type of results. Regarding the definition of the Malliavin derivative and its properties, see \cite{H,N}.

\bigskip

\section{PDE estimates} In this section, the entries of $\sigma$ and of the other functions have bounded derivatives in the spatial variables, uniformly in $t$, in order to justify the manipulations with the equations; however, the constants in the estimates will not depend on the smoothness of $\sigma(t,\cdot)$. In the next section we will apply our PDE estimates to matrices $\sigma_n$ from the statement of Theorem \ref{thm1}, and so the constants in the resulting estimates will be independent of $n$.

\subsection{Contractivity and gradient bounds for the Kolmogorov equation}
\label{contr_grad_sect}

The following simple observation makes the proofs of the results of this section -- Propositions \ref{contr_prop} and \ref{prop_grad} -- more transparent.
Since the matrix of diffusion coefficients $\sigma$ is bounded, the conditions of Theorem \ref{thm1} imply that
\begin{equation}
\label{nabla_a}
(\nabla_r a_{il})_{i=1}^{d} \in  \mathbf{F}_{\gamma_{rl}} \quad \text{ for all }r,l, \qquad \nabla a \in  \mathbf{F}_{\delta_{a}} 
\end{equation}
with $\gamma_{rl} \leq C\nu$, $\delta_a \leq C \nu$. It therefore suffices to show that, if $\gamma_{rl}$, $\delta_a$ are sufficiently small, then the assertions of Propositions \ref{contr_prop}, \ref{prop_grad} hold.

\begin{proposition}
\label{contr_prop}
Under the conditions of Theorem \ref{thm1}, for every $p>1$, provided that the form-bound $\nu$ of the derivatives of $\sigma$ is sufficiently small, and $\theta$ is fixed sufficiently small, the solution to the terminal-value problem
\begin{equation}
\label{Cauchy2}
\left\{
\begin{array}{l}
(\partial_t + \frac{1}{2} a\cdot \nabla^2)u=0 \quad \text{ in } ]0,T[ \times \mathbb R^d, \\
u|_{t=T}=f \in C_c^\infty(\mathbb R^d), 
\end{array}
\right.
\end{equation}
where $a=\sigma\sigma^{\top}$, satisfies
$$
\|u(t)\|_{L^p_\rho} \leq e^{\omega (T-t)}\|f\|_{L^p_\rho}, \quad t \in [0,T],
$$
where $\omega=\omega(c_\delta,\theta)$.  The same estimate holds in the non-weighted case $\rho = 1$, i.e.\,$\theta=0$.
\end{proposition}

\begin{proof} We will only need $\nabla a \in \mathbf{F}_{\delta_a}$, see the remark in the beginning of Section \ref{contr_grad_sect}. The proof given below is standard, but we nevertheless include the details for the reader's convenience.

It will be convenient to re-denote $u(T-t,\cdot)$, $a(T-t,\cdot)$ by $u(t,\cdot)$, $a(t,\cdot)$, respectively, and 
carry out the proof for the initial-value problem 
\begin{equation}\label{Cprob3}
\left\{
\begin{array}{l}
\partial_t u- \frac{1}{2} a \cdot \nabla^2 u = 0, \\
 u|_{t=0}=f.
 \end{array}
 \right.
\end{equation}
Let us rewrite the equation in the divergence form
$$
\partial_t u- \frac{1}{2} \nabla \cdot a \cdot \nabla u + \frac{1}{2} (\nabla a) \cdot \nabla u = 0.
$$
We now multiply this equation  by $\rho |u|^{p-2}u$ and, after integrating by parts, obtain, 
\begin{align*}
\frac{1}{p}\frac{d}{dt}\langle |u|^p \rho \rangle & + \frac{2(p-1)}{p^2} \langle a \cdot \nabla |u|^{p/2},\nabla  |u|^{p/2} \rho \rangle \\
& = - \frac{1}{p}\langle \nabla a \cdot \nabla |u|^{\frac{p}{2}}, |u|^{\frac{p}{2}}\rho \rangle  -\frac{1}{p}
\langle a \cdot \nabla |u|^{\frac{p}{2}}, |u|^{\frac{p}{2}}  \nabla \rho \rangle.
\end{align*}
Multiplying by $p$ and using the ellipticity of $a$, we obtain
\begin{align}
\label{ineq_contr}
\frac{d}{dt}\langle |u|^p \rho\rangle & + \frac{2(p-1)}{p}\kappa \langle |\nabla |u|^{p/2}|^2 \rho \rangle \\
& \leq |\langle \nabla a \cdot \nabla |u|^{\frac{p}{2}}, |u|^{\frac{p}{2}}\rho \rangle|  + 
|\langle a \cdot \nabla |u|^{\frac{p}{2}}, |u|^{\frac{p}{2}}  \nabla \rho \rangle|=:I_1+I_2. \notag
\end{align}
Let us handle $I_2$ first. By the Cauchy-Schwarz inequality,
\begin{align*}
I_2 & \leq \alpha \langle a \cdot \nabla |u|^{\frac{p}{2}}, (\nabla |u|^{\frac{p}{2}})\rho\rangle + \frac{1}{4\alpha} \langle a \cdot \nabla \rho, \frac{\nabla \rho}{\rho} |u|^p  \rangle \\
& \leq \alpha \|a\|_\infty \langle |\nabla |u|^{\frac{p}{2}}|^2 \rho\rangle + \frac{\|a\|_\infty }{4\alpha} \langle |\nabla \rho|^2, |u|^p  \rangle \\
& (\text{apply \eqref{two_est}}) \\
& \leq \alpha \|a\|_\infty  \langle |\nabla |u|^{\frac{p}{2}}|^2\rangle + \frac{\|a\|_\infty }{4\alpha}C\theta \langle |u|^p \rho \rangle,
\end{align*}
where $\|a\|_\infty  \leq \|\sigma\|_\infty^2$.
Next,
\begin{align*}
I_1 & \leq \beta \langle |\nabla a|^2,|u|^p \rho\rangle + \frac{1}{4\beta} \langle |\nabla |u|^{\frac{p}{2}}|^2 \rho\rangle \\
& (\text{use $\nabla a \in \mathbf{F}_{\delta_a}$}) \\
& \leq \beta\left(\delta_a \langle |\nabla (|u|^{\frac{p}{2}}\sqrt{\rho})|^2 \rangle + c_{\delta_a} \langle |u|^p \rho\rangle \right) + \frac{1}{4\beta} \langle |\nabla |u|^{\frac{p}{2}}|^2 \rho\rangle.
\end{align*}
The term $\langle |\nabla (|u|^{\frac{p}{2}}\sqrt{\rho})|^2 \rangle$ is bounded from above by the sum of $\langle |\nabla |u|^{\frac{p}{2}}|^2 \rho\rangle$ and, via \eqref{two_est}, $\langle |u|^p\rho\rangle$. Applying these estimates in \eqref{ineq_contr} and optimizing in $\alpha$ and $\beta$, we arrive at the inequality
$$
\frac{d}{dt}\langle |u(t)|^p \rho \rangle + \biggl[\frac{2(p-1)}{p}\kappa - \text{contribution from $I_1,I_2$} \biggr]\langle |\nabla |u|^{p/2}|^2 \rho \rangle  \leq \omega \langle |u|^p \rho\rangle.
$$
Provided that $\delta_a$ is sufficiently small, we can select $\theta$ sufficiently small so that the expression in the square brackets, i.e.\,the multiple of the energy term, is positive. Discarding the positive energy term, we arrive at the sought bound.
\end{proof}

\begin{proposition}[A priori gradient bounds]
\label{prop_grad}
Under the conditions of Theorem \ref{thm1}, for every $q>d$, provided that the form-bound $\nu$ of the derivatives of $\sigma$ is sufficiently small, and $\theta$ is fixed sufficiently small, the solution to 
\begin{equation}
\label{Cauchy3}
\left\{
\begin{array}{l}
(\partial_t +\frac{1}{2} a\cdot \nabla^2)v=-|F|\quad \text{ in } ]0,T[ \times \mathbb R^d, \\[2mm]
v|_{t=T}=f \in C_c^\infty(\mathbb R^d), 
\end{array}
\right.
\end{equation}
where $F \in \mathbf{F}_\mu$ for some $\mu<\infty$ (here, additionally, bounded and smooth), satisfies
\begin{align*}
\|v\|_{L^\infty([0,T],L_\rho^q)}^q  + \|\nabla v\|^q_{L^\infty([0,T],L_\rho^q)} & + \|\nabla|\nabla v|^{\frac{q}{2}}\|_{L^2([0,T],L_\rho^2)}^2 \\
& \leq C \bigl(\|F\|^2_{L^2([0,T],L_\rho^2)} + \|\nabla f\|^q_{L_\rho^q} + \|f\|^q_{L^q_\rho}\big).
\end{align*}
with constant $C=C(T,\theta,d,q,\kappa,\|\sigma\|_\infty,\nu,c_\nu,\mu,c_\mu)$ independent of the smoothness of $\sigma$, $f$ or $F$. 
In particular, by the Sobolev embedding theorem, for every $z \in \mathbb Z^d$
\begin{equation}
\label{v_est_unif}
\|v\|^q_{L^\infty([0,T],B_{10}(z))} \leq C_z  \bigl(\|F\|^2_{L^2([0,T],L_\rho^2)} + \|\nabla f\|^q_{L_\rho^q} + \|f\|^q_{L^q_\rho}\big).
\end{equation}
The same estimates hold in the non-weighted case $\rho = 1$, in which case, in the last estimate, one has simply $\|v\|_{L^\infty([0,T],\mathbb R^d)}$.
\end{proposition}

We prove Proposition \ref{prop_grad} in Section \ref{prop_grad_proof}; see Section \ref{proof_struct_sect} for the discussion of the role of this proposition in the proof of Theorem \ref{thm1}.

\subsection{Simplex estimate} Since $\sigma$ is assumed to be smooth in this section, by classical theory there exists a unique strong solution $X_t^x$ to SDE
\begin{equation}
\label{sde0}
X_t^x=x+ \int_0^t \sigma(r,X_r^x) \, dW_r.
\end{equation}
Let $0 \leq T_0 \leq T_1 \leq T$. 

At the risk of annoying the reader, we will carry out the proof of the next proposition in a greater generality that is needed for the proof of Theorem \ref{thm1}. Namely, we will assume that the functions $g_\nu$ appearing in the form-boundedness condition on $\nabla \sigma(t,\cdot)$  can be locally unbounded. 
 For the proof of Theorem \ref{thm1} we only need $g_\nu$ to be bounded, and thus can replace them with constant $c_\nu$. However, keeping in mind possible extensions of Theorem \ref{thm1} and the importance of the next proposition, we opt for greater generality. (We will also need this generality when we discuss Theorem \ref{thm2} dealing with singular drift.)

\begin{proposition}
\label{prop1}
Let $f_i \in \mathbf{F}_\nu$ ($i \geq 1$) with the same
$$
g_\nu \in L^{1+\varepsilon}(\mathbb R)
$$ 
for some fixed $\varepsilon>0$.
Assuming that the constant $\theta$ in the definition of weight $\rho$ is fixed sufficiently small, there exist positive constants $C_0$, $K$ (independent of the smoothness of $\sigma$ and $f_i$) such that, for every $n \geq 2$, for all $T_0<T_1$ satisfying $0 \leq T_0 \leq T_1 \leq T$,
\begin{equation}
\label{estimate_g0}
\left\|\mathbf{E} \int_{\Delta_{n}(T_0,T_1)} \prod_{i=1}^n |f_i(t_i, X_{t_i}^x) |^{2}dt_1 \dots dt_n \right\|_{L^{2}_\rho(\mathbb R^d)}^{2} \leq C_0K^{n-2}(T_1-T_0)^{\frac{\varepsilon}{1+\varepsilon}}.
\end{equation}
Moreover, $K$ can be made as small as needed by assuming that the form-bound $\nu$ of $f_i$ is sufficiently small.

If $g_\nu \in L^{\infty}(\mathbb R)$, then 
\begin{equation}
\label{estimate_g}
\left\|\mathbf{E} \int_{\Delta_{n}(T_0,T_1)} \prod_{i=1}^n |f_i(t_i, X_{t_i}^x) |^{2}dt_1 \dots dt_n \right\|_{L^{2}_\rho(\mathbb R^d)}^{2} \leq C_0K^{n-2}(T_1-T_0).
\end{equation}
\end{proposition}

\begin{corollary} 
\label{cor1}
For every integer $n \geq 2$,
$$
\int_{\mathbb R^d} \biggl[\mathbf{E} \biggl(\int_0^t |\nabla \sigma(r,X_r^x)|^2 ds \biggr)^n \biggr]^2 \rho(x) dx \leq C_0(n!)^2
K^{n-2}t^{\frac{\varepsilon}{1+\varepsilon}}.$$
If, moreover, $g_\nu \in L^{\infty}(\mathbb R)$, then $t^{\frac{\varepsilon}{1+\varepsilon}}$ is replaced with $t$.
\end{corollary}
\begin{proof}[Proof of Corollary \ref{cor1}] This follows from
\begin{align*}
\int_{\mathbb R^d} & \biggl[\mathbf{E} \biggl(\int_0^t |\nabla \sigma(r,X_r^x)|^2 ds \biggr)^n \biggr]^2  \rho(x) dx\\
& = \int_{\mathbb R^d} \biggl[n!\mathbf{E}  \int_{\Delta_n(0,t)} \prod_{i=1}^n|\nabla \sigma(r_i,X_{r_i}^x)|^2 dr_1,\cdots,dr_n  \biggr]^2 \rho(x)dx.
\end{align*}
\end{proof}

\begin{remark}
\label{trans_rem}
The assertions of Proposition \ref{prop1} and Corollary \ref{cor1} remain valid if we replace weight $\rho$ by its translation $\rho_y$, for any $y \in \mathbb R^d$ fixed, defined by $\rho_y(x):=\rho(x-y)$. In particular, the constants in these estimates do not depend on $y$ since the form-boundedness condition on $f_i$ and $b$ does not involve any special points.
\end{remark}

\begin{proof}[Proof of Proposition \ref{prop1}]
Put $u_{n+1}=1$ and define consecutively $$h_k=f_k^2 u_{k+1}, \quad k=1,\dots,n,$$
where $u_k$ solves the terminal-value problem on $[T_0,T_1]$
\begin{equation}
\label{eq_uk}
\partial_t u_k + \frac{1}{2}a \cdot \nabla^2 u_k +h_k=0, \quad 
u_k(T_1,\cdot)=0.
\end{equation}
Then, repeating the proof of \cite[Lemma 4.2]{RZ}, we obtain
$$
\mathbf{E}_{\mathcal F^W_{T_0}} \int_{\Delta_{n}(T_0,T_1)} \prod_{i=1}^n \biggl |f_i(t_i, X_{t_i}^x) \biggr |^{2} dt_1 \dots dt_n=u_1(T_0,X_{T_0}^x).
$$
Therefore, 
\begin{align*}
\int_{\mathbb R^d }\left|\mathbf{E} \int_{\Delta_{n}(T_0,T_1)} \prod_{i=1}^n |f_i(t_i, X_{t_i}^x) |^{2} dt_1 \dots dt_n \right|^2 dx =\int_{\mathbb R^d } |\mathbf{E}u_1(T_0,X_{T_0}^x)|^2 dx.
\end{align*}
The goal is to control $u_1$ in terms of $u_2$ (Step 1), $u_2$ in terms of $u_3$ (Step 2), and so on, until $u_n$ (Step $n$). Note that $u_n$ can be estimated explicitly because $u_{n+1}=1$. 

1.~Let us first consider the special case: 
$$
\text{$g_{\delta_a}=g_{\nu}=0$ in conditions $\nabla a \in \mathbf{F}_{\delta_a}$, $f_i \in \mathbf{F}_{\nu}$, and $\rho \equiv 1$.}
$$
The latter forces us to impose an additional \textit{temporary} condition on the $f_i$:
\begin{equation}
\label{A}
\tag{A}
\text{all $f_i$ have supports in $[0,T] \times B_R$ for some fixed $R$ independent of $n$}.
\end{equation}
Below we will exclude this condition using non-constant (i.e.\,polynomially vanishing at infinity) weight $\rho$. 

\begin{remark} This proposition plays a crucial role in the proof of Theorem \ref{thm1}. Specifically, we will be applying it to $f_i=\nabla_k \sigma_{il}$. If we put condition \eqref{A} in the assumptions of Proposition \ref{prop1}, this will require us to assume that the matrix of diffusion coefficients has form $\sigma=\sigma_{\const} + \sigma_c$, where $\sigma_{\const}$ is a constant matrix and $\sigma_c$ has compact support. Since we are interested in describing admissible local discontinuities of $\sigma$, this would not be a deal breaking extra condition on $\sigma$, but with a few additional efforts we will avoid it by means of the weight $\rho$.
\end{remark}

What we will obtain for $u_2,\dots, u_n$ are the following two inequalities
\begin{equation}
\label{u2_un}
\int_{T_0}^{T_1}\langle |\nabla u_{2}|^2 \rangle  \leq K^{n-2} \int_{T_0}^{T_1} \langle |\nabla u_n|^2\rangle  
\end{equation}
where $0<K<1$ can be made as small as needed by assuming that the form-bounds $\delta_{a}$ and $\nu$ are sufficiently small,
and 
\begin{equation}
\label{un}
\int_{T_0}^{T_1}\langle |\nabla u_n(s)|^2 \rangle ds \leq  C(T_1-T_0)
\end{equation}
whose proof is basically the energy inequality for \eqref{eq_uk}. The relationship between $u_1$ and $u_2$ is different: we need to prove inequality
\begin{equation}
\label{prelim_est}
\int_{\mathbb R^d } |\mathbf{E}u_1(T_0,X_{T_0}^x)|^2 dx \leq \int_{T_0}^{T_1}\langle |\nabla u_{2}|^2 \rangle.
\end{equation}
Proposition \ref{prop1} for $g_{\delta_a}=g_\nu=0$ and $\rho \equiv 1$ follows right away from \eqref{u2_un}-\eqref{prelim_est}.

Let us prove \eqref{u2_un}-\eqref{prelim_est}. Recall that the action of the row-divergence operator $\nabla$ on $a=\sigma\sigma^{\scriptscriptstyle \top}$ is defined as the vector field $\nabla a:\mathbb R^{1+d} \rightarrow \mathbb R^d$ with components $(\nabla a )_{k} :=  \sum_{i=1}^{d} ( \nabla_{i} a_{ik})$. The boundedness of $\sigma$ and the hypothesis of form-boundedness of $|\nabla \sigma|$ of Theorem \ref{thm1} ensure that
\begin{equation}
\label{a1}
\nabla a \in \mathbf{F}_{\delta_a}\text{ for }\delta_a=C\nu.
\end{equation}

\medskip

\noindent Step 1 (proof of \eqref{prelim_est}). Set $v(t,x):=\mathbf{E}_{X_t=x}u_1(T_0,X_{T_0}^x)$. The function $v$ satisfies
the terminal-value problem on $[0,T_0]$
$$
\partial_t v +  \frac{1}{2} a\cdot\nabla^2 v=0, \quad v(T_0,\cdot)=u_1(T_0,\cdot).
$$
We need to bound $\langle |v(0,\cdot)|^2\rangle$, where $v(0,\cdot)=\mathbf{E}_{X_0=x}u_1(T_0,X_{T_0}^x) \equiv \mathbf{E}u_1(T_0,X_{T_0}^x)$. So, we define
$$
\tilde{u}(t,\cdot):=\left\{
\begin{array}{ll}
v(t,\cdot) & t \in [0,T_0] \\
u_1(t,\cdot) & t \in [T_0,T_1].
\end{array}
\right.
$$
Then $\tilde{u}$ satisfies on the whole interval $[0,T_1]$ the terminal-value problem 
\begin{equation}
\label{tilde_u}
\partial_t \tilde{u} +  \frac{1}{2} a \cdot \nabla^2 \tilde{u}+\tilde{h}=0, \quad \tilde{u}(T_1,\cdot)=0,
\end{equation}
where $$
\tilde{h}(t,\cdot)=\left\{
\begin{array}{ll}
0 & t \in [0,T_0], \\
h_1(t,\cdot) & t \in ]T_0,T_1].
\end{array}
\right.
$$
So, what we need to bound is $\langle |\tilde{u}(0,\cdot)|^2\rangle$. To this end, we multiply the equation in \eqref{tilde_u} by $\tilde{u}$ and integrate over $[0,T_1] \times \mathbb R^d$:
\begin{equation}
\label{tilde_en}
\frac{1}{2}\langle |\tilde{u}(0,\cdot)|^2\rangle - \frac{1}{2}\int_0^{T_1} \langle a \cdot \nabla^2 \tilde{u},\tilde{u}\rangle = \int_0^{T_1}\langle \tilde{h},\tilde{u}\rangle.
\end{equation}
We now use our hypothesis on $a$:
\begin{align}
-\int_0^{T_1} \langle a \cdot \nabla^2 \tilde{u},\tilde{u}\rangle & = - \int_0^{T_1} \langle \nabla \cdot a \cdot \nabla \tilde{u}, \tilde{u}\rangle + \int_0^{T_1} \langle (\nabla a) \cdot \nabla \tilde{u},\tilde{u}\rangle \notag \\
& (\text{integrate by parts and use $a \geq \kappa I$  in the first term,} \notag \\
& \text{apply Cauchy-Schwarz in the second term}) \notag \\
& \geq \kappa \int_0^{T_1} \langle |\nabla \tilde{u}|^2\rangle - \biggl(\frac{1}{2\sqrt{\delta_a}}\int_0^{T_1}\langle |\nabla a|^2,\tilde{u}^2\rangle + \frac{\sqrt{\delta_a}}{2}\int_0^{T_1}\langle |\nabla \tilde{u}|^2\rangle\biggr) \notag\\
& (\text{use \eqref{a1}, where, recall, for now $g_{\delta_a}=0$}) \notag \\
& \geq (\kappa-\sqrt{\delta_a}) \int_0^{T_1} \langle |\nabla \tilde{u}|^2\rangle. \label{a_calc}
\end{align}
Next, we estimate the RHS of \eqref{tilde_en}:
\begin{align}
\int_0^{T_1}\langle \tilde{h},\tilde{u}\rangle & = \int_{T_0}^{T_1}\langle h_1,u_1\rangle = \int_{T_0}^{T_1}\langle f_1^2 u_2,u_1\rangle \notag \\
& \leq \beta \int_{T_0}^{T_1} \langle f_1^2, u_1^2\rangle + \frac{1}{4\beta}\int_{T_0}^{T_1}\langle f_1^2,u_2^2\rangle\notag \\
& \text{(use $f_i \in \mathbf{F}_\nu$)} \notag \\
& \leq \beta \nu \int_{T_0}^{T_1}\langle |\nabla u_1|^2\rangle + \frac{\nu}{4\beta} \int_{T_0}^{T_1}\langle |\nabla u_2|^2\rangle \notag \\
& \leq \beta \nu \int_{0}^{T_1}\langle |\nabla \tilde{u}|^2\rangle+ \frac{\nu}{4\beta} \int_{T_0}^{T_1}\langle |\nabla u_2|^2\rangle. \label{h_calc}
\end{align}
Applying these estimates in \eqref{tilde_en}, we arrive at
\begin{equation}
\label{U_est}
\frac{1}{2}\langle |\tilde{u}(0,\cdot)|^2\rangle + C_1 \int_{0}^{T_1}\langle |\nabla \tilde{u}|^2\rangle \leq C_2 \int_{T_0}^{T_1}\langle |\nabla u_2|^2\rangle,
\end{equation}
where $C_1=\frac{1}{2}(\kappa-\sqrt{\delta_a}) - \beta \nu$, $C_2=\frac{\nu}{4\beta}$.
So, $0<\frac{C_2}{C_1}<1$ can be made arbitrarily small by assuming that the form-bound $\delta_a$ and $\nu$ are sufficiently small. 
So, we get
\begin{equation*}
\int_{\mathbb R^d}\mathbf{E}|u_1(T_0,X_{T_0}^x)|^2 dx \leq 2C_2 \int_{T_0}^{T_1}\langle |\nabla u_2|^2\rangle,
\end{equation*}
i.e.\,we have proved \eqref{prelim_est}. (At this step we actually did not need the term with $\nabla \tilde{u}$ in \eqref{U_est}, but since we will be repeating this calculation below, we kept it.)

\medskip

\noindent Step 2 (proof of \eqref{u2_un}). Now, we repeat  the same procedure for $u_2$ in place of $\tilde{u}$. That is, we multiply equation \eqref{eq_uk} (for $k=2$) by $u_2$ and integrate over $[T_0,T_1] \times \mathbb R^d$ to obtain
$$
\frac{1}{2}\langle u_2^2(T_0)\rangle + \frac{1}{2}(\kappa-\sqrt{\delta_a})\int_{T_0}^{T_1}\langle  |\nabla u_2|^2 \rangle = \int_{T_0}^{T_1} \langle h_2,u_2\rangle,
$$
where
\begin{align*}
\int_{T_0}^{T_1} \langle h_2,u_2\rangle ds & =\int_{T_0}^{T_1} \langle f_2^2 u_3,u_2 \rangle  \\
&\leq \beta\int_{T_0}^{T_1} \langle f_2^2, u^2_3\rangle + \frac{1}{4\beta} \int_{T_0}^{T_1} \langle f_2^2, u_2^2\rangle \\
& (\text{we are using $f_i \in\mathbf{F}_\nu$}) \\
& \leq \beta \nu  \int_{T_0}^{T_1} \langle |\nabla  u_2(s)|^2\rangle  +  \frac{\nu}{4\beta}\int_{T_0}^{T_1} \langle |\nabla  u_3(s)|^2 \rangle.
\end{align*}
Thus, we arrive at
$$
\int_{T_0}^{T_1}\langle |\nabla u_2|^2 \rangle \leq \frac{C_2}{C_1} \int_{T_0}^{T_1} \langle |\nabla u_3|^2\rangle.
$$
If $n > 3$, we repeat this $n-3$ more times:
$$
\int_{T_0}^{T_1}\langle |\nabla u_{2}|^2 \rangle  \leq \biggl(\frac{C_2}{C_1}\biggr)^{n-2} \int_{T_0}^{T_1} \langle |\nabla u_n|^2\rangle,
$$
i.e.\,we have \eqref{u2_un}.

\medskip

\noindent Step 3 (proof of \eqref{un}). Finally, we estimate $\int_{T_0}^{T_1}\langle |\nabla u_n|^2\rangle  $. Arguing as above,
we have (recall that $u_{n+1}=1$)
\begin{align*}
\int_{T_0}^{T_1}\langle |\nabla u_n|^2 \rangle & \leq C_3\int_{T_0}^{T_1} \langle  
f_n^2,  u_n\rangle \\
& (\text{we apply the Cauchy-Schwarz inequality}) \\
& \leq C_4\int_{T_0}^{T_1} \langle f_n^2 \rangle  + \frac{1}{2}\int_{T_0}^{T_1} \langle f_{n}^{2},  u_n^2 \rangle  \\
& \leq C_4\int_{T_0}^{T_1} \langle f_n^2 \rangle + \frac{\nu}{2}\int_{T_0}^{T_1} \langle |\nabla u_n|^2  \rangle \\
& (\text{using assumption ($A$),} \\
& \text{we apply $f_n \in \mathbf{F}_\nu$ in $\int_{T_0}^{T_1} \langle f_n^2 \varphi^2\rangle  \geq \int_{T_0}^{T_1} \langle f_n^2 \rangle $ for a smooth $\varphi \geq \mathbf{1}_{B_R(0)}$}) \\
& \leq C_5(T_1-T_0) + \frac{\nu}{2}\int_{T_0}^{T_1} \langle |\nabla u_n|^2  \rangle.
\end{align*}
Thus, $(1- \frac{\nu}{2})\int_{T_0}^{T_1}\langle |\nabla u_n|^2 \rangle  \leq  C_5(T_1-T_0)$, i.e.\,since $\nu$ is small, we obtain \eqref{un}.

\medskip

2.~We now explain what needs to be modified in the previous calculations to handle the case $g_{\delta_a}, g_{\nu} \in L^{1+\varepsilon}(\mathbb R)$, but still assuming for now $\rho \equiv 1$.

Every time we were applying the form-boundedness condition on $\nabla a$ or $f_i$, we were getting terms with bad sign that needed to be absorbed by the dispersion term in the equation, cf.\,\eqref{a_calc}. Now, with $g_{\delta_a}, g_{\nu} \neq 0$, we will be getting additional bad terms $\int_0^{T_1}g_{\delta_a}\langle \tilde{u}^2\rangle$ and $\int_{T_0}^{T_1}g_{\nu}\langle \tilde{u}^2\rangle$. These terms, however, can be dealt with simply by introducing in the equation a time-dependent exponential factor. In detail, the following are the extensions of Steps 1-3 from above:

\medskip

\text{Step 1':} Put $$F(t):=\lambda \int_0^t \big[g_{\delta_a}(s)+g_\nu(s)\big]ds,$$ where $\lambda$ is to be fixed sufficiently large (depending on the values of $\delta_a$ and $\nu$). 
We multiply \eqref{tilde_u} by $e^{F}$:
$$
\partial_t (e^{F}\tilde{u}) - F'e^{F}\tilde{u} + \frac{1}{2}a \cdot \nabla^2 e^{F}\tilde{u}+e^{F}\tilde{h}=0, \quad e^{F(T_1)}\tilde{u}(T_1,\cdot)=0.
$$
After multiplying the previous equation by $\tilde{u}$ and integrating over $[0,T_1] \times \mathbb R^d$, one sees that the second term from the left
$$
\int_0^{T_1} \langle F'e^{F}\tilde{u}^2 \rangle= \lambda \int_0^{T_1} (g_{\delta_a}+g_\nu)e^{F}\langle \tilde{u}^2 \rangle 
$$
will
absorb the new terms $\frac{1}{2\sqrt{\delta_a}}\int_0^{T_1}\langle g_{\delta_a} e^{F} \tilde{u}^2\rangle$ and $\beta\int_{0}^{T_1}\langle g_\nu e^{F}\tilde{u}^2 \rangle$ that now appear in \eqref{a_calc} and \eqref{h_calc} as long as we fix $\lambda \geq \frac{1}{2\sqrt{\delta_a}} + \beta$. This will give us, instead of \eqref{U_est}, the following estimate (note that $e^{F(0)}=1$):
$$
\frac{1}{2}\langle |\tilde{u}(0,\cdot)|^2\rangle + C_1 \int_{0}^{T_1} e^{F}\langle|\nabla \tilde{u}|^2\rangle \leq \frac{\nu}{4\beta} \int_{T_0}^{T_1}e^F \langle |\nabla u_2|^2\rangle + \frac{1}{4\beta}\int_{T_0}^{T_1}e^{F}g_\nu \langle   u_2^2 \rangle,
$$
where $C_1=\frac{1}{2}(\kappa-\sqrt{\delta_a}) - \beta \nu$ is as before.
Hence
\begin{equation}
\label{u1_g_ineq}
\mathbf{E}u_1(T_0,X_{T_0}^x) \leq C\biggl(\int_{T_0}^{T_1}e^F \langle |\nabla u_2|^2\rangle + \int_{T_0}^{T_1}e^{F}g_\nu \langle   u_2^2 \rangle\biggr).
\end{equation}

\smallskip

\text{Step 2':} Set 
$$
G(t):=\lambda \int_{T_0}^t \big(g_{\delta_a}(s)+g_\nu(s)\big)ds.
$$ We multiply \eqref{eq_uk} for $k=2$ by $e^{G}$: 
$$
\partial_t (e^G u_2) - G'e^G u_2 + \frac{1}{2}a \cdot \nabla^2 e^G u_2 +e^G h_2=0, \quad 
u_2(T_1,\cdot)=0.
$$
Next, we multiply it by $u_2$ and integrate over $[T_0,T_1] \times \mathbb R^d$, obtaining (note that $e^{G(T_0)}=1$)
\begin{align*}
\frac{1}{2}\langle u_2^2(T_0)\rangle + (\frac{\lambda}{2}-\beta)\int_{T_0}^{T_1} g_\nu e^G \langle u_2^2\rangle & + (\kappa-\sqrt{\delta_a}-\beta \nu)\int_{T_0}^{T_1} e^G\langle |\nabla u_2|^2\rangle \\
& \leq \frac{1}{4\beta}\biggl(\nu\int_{T_0}^{T_1} e^G \langle |\nabla u_3|^2\rangle + \int_{T_0}^{T_1} e^Gg_\nu \langle u_3^2\rangle \biggr). 
\end{align*}
Since we are going to iterate this inequality ($u_3$ and $u_4$, then $u_4$ and $u_5$, etc), we need to be more precise when specifying the constants, including $\lambda$. Discarding the first (positive) term and assuming that $\nu \leq 1$, we get
\begin{align*}
(\frac{\lambda}{2}-\beta)\int_{T_0}^{T_1} g_\nu e^G \langle u_2^2\rangle & + (\kappa-\sqrt{\delta_a}-\beta \nu)\int_{T_0}^{T_1} e^G\langle |\nabla u_2|^2\rangle \\
& \leq \frac{1}{4\beta}\biggl(\int_{T_0}^{T_1} e^G \langle |\nabla u_3|^2\rangle + \int_{T_0}^{T_1} e^Gg_\nu \langle u_3^2\rangle \biggr). 
\end{align*}
So, we take $C_2:=\frac{1}{4\beta}$, $C_1:=\kappa-\sqrt{\delta_a}-\beta \nu$ and fix $\frac{\lambda}{2}>\beta$ so that $\frac{\lambda}{2}-\beta=C_1$. Note that by selecting $\beta$ sufficiently large we can make $C_2$ arbitrarily small. Accordingly, we will have to select $\nu$ small to have $C_1$, say, $C_1 \geq \frac{\kappa}{2}$ and have $0<\frac{C_2}{C_1}<1$.
This gives us inequality
\begin{align*}
 \int_{T_0}^{T_1} e^G\langle |\nabla u_2|^2\rangle  & +  \int_{T_0}^T g_\nu e^G \langle u_2^2\rangle\\
& \leq \frac{C_2}{C_1}\biggl(\int_{T_0}^{T_1} e^G \langle |\nabla u_3|^2\rangle + \int_{T_0}^{T_1} e^Gg_\nu \langle u_3^2\rangle \biggr)
\end{align*}
which we can iterate (since the selection of the parameters $\beta$, $\lambda$, $\delta_a$, $\nu$ does not depend on $n$), obtaining
\begin{align}
 \int_{T_0}^{T_1} e^G\langle |\nabla u_2|^2\rangle & +  \int_{T_0}^{T_1} g_\nu e^G \langle u_2^2\rangle \notag \\
& \leq \biggl(\frac{C_2}{C_1}\biggr)^{n-2}\biggl(\int_{T_0}^{T_1} e^G \langle |\nabla u_n|^2\rangle + \int_{T_0}^{T_1} e^Gg_\nu \langle u_n^2\rangle \biggr). \label{step2_est}
\end{align}

\medskip

\text{Step 3':}~Finally, we estimate $\int_{T_0}^{T_1} e^G \langle |\nabla u_n|^2\rangle + \int_{T_0}^{T_1} e^Gg_\nu \langle u_n^2\rangle$. Arguing as above,
we obtain (recall that $u_{n+1}=1$)
\begin{align*}
\int_{T_0}^{T_1} e^G \langle |\nabla u_n|^2\rangle + \int_{T_0}^{T_1} e^Gg_\nu \langle u_n^2\rangle & \leq C_3\int_{T_0}^{T_1} \langle  
e^G f_n^2,  u_n\rangle \\
& (\text{we are applying Cauchy-Schwarz inequality}) \\
& \leq C_4\int_{T_0}^{T_1} \langle e^G f_n^2 \rangle  + \frac{1}{2}\int_{T_0}^{T_1} \langle f_{n}^{2}, e^G u_n^2 \rangle  \\
& \leq C_4\int_{T_0}^{T_1} \langle e^G f_n^2 \rangle + \frac{\nu}{2}\int_{T_0}^{T_1} \langle e^G |\nabla u_n|^2  \rangle + \frac{1}{2}\int_{T_0}^{T_1} g_\nu e^G \langle u_n^2 \rangle \\
& \leq C_4\int_{T_0}^{T_1} \langle e^G f_n^2 \rangle  + \frac{1}{2}\biggl(\int_{T_0}^{T_1} \langle e^G |\nabla u_n|^2  \rangle + \int_{T_0}^{T_1} g_\nu e^G \langle u_n^2 \rangle\biggr).
\end{align*}
Thus, 
$$
\int_{T_0}^{T_1} e^G \langle |\nabla u_n|^2\rangle + \int_{T_0}^{T_1} e^Gg_\nu \langle u_n^2\rangle  \leq 2 C_4\int_{T_0}^{T_1} \langle e^G f_n^2 \rangle.
$$
We are now using the compact support assumption \eqref{A}, i.e.\,we apply $f_n \in \mathbf{F}_\nu$ in $\int_{T_0}^{T_1} \langle f_n^2 \varphi^2\rangle  \geq \int_{T_0}^{T_1} \langle f_n^2 \rangle $ for a smooth $\varphi \geq \mathbf{1}_{B_R(0)}$, obtaining
\begin{align}
\int_{T_0}^{T_1} e^G \langle |\nabla u_n|^2\rangle + \int_{T_0}^{T_1} e^Gg_\nu \langle u_n^2\rangle & \leq  C_5\int_{T_0}^{T_1}(1+g_\nu) \notag \\
& \leq C_6\bigl[T_1-T_0 + (T_1-T_0)^{\frac{\varepsilon}{1+\varepsilon}} \bigr],\label{step3_est}
\end{align}
where at the last step we have applied $g_\nu \in L_{\loc}^{1+\varepsilon}(\mathbb R)$.

Now, combining \eqref{u1_g_ineq}, \eqref{step2_est} and \eqref{step3_est}, and noting that both $F$, $G$ are non-negative continuous functions that are, thus, bounded  on $[T_0,T_1]$, we obtain the assertion of Proposition \ref{prop1} in the case $\rho \equiv 1$.

\medskip

3. We finally pass to the weighted space $L^2_\rho(\mathbb R^d)$ and
remove assumption \eqref{A}. We write
\[
 \langle h\rangle_\rho:=\int_{\mathbb R^d}h(x)\rho(x)\,dx.
\]
We first note that the weight $\rho$ works very well together with the form-boundedness
condition. Namely, for every fixed $\eta>0$, applying the form-boundedness
of $f_i$ to test function $\sqrt{\rho}\,u$ yields
\begin{align*}
 \langle f_i^2,u^2\rangle_\rho
 &\leq
 \nu\langle|\nabla(\sqrt{\rho}\,u)|^2\rangle
 +g_\nu\langle u^2\rangle_\rho \\
 &\leq
 (1+\eta)\nu\langle|\nabla u|^2\rangle_\rho
 +(g_\nu+C_\eta\nu\theta)\langle u^2\rangle_\rho.
\end{align*}
Similarly,
\begin{equation*}
 \langle|\nabla a|^2,u^2\rangle_\rho
 \leq
 (1+\eta)\delta_a\langle|\nabla u|^2\rangle_\rho
 +(g_{\delta_a}+C_\eta\delta_a\theta)
       \langle u^2\rangle_\rho.
\end{equation*}
Also,
\begin{align*}
 -\langle a\cdot\nabla^2u,u\rangle_\rho
 &=
 \langle a\nabla u,\nabla u\rangle_\rho
 +\langle(\nabla a)\cdot\nabla u,u\rangle_\rho
 +\int_{\mathbb R^d}
       a\nabla u\cdot\nabla\rho\,u\,dx.
\end{align*}
Since the matrix field $a$ is bounded and
$
 \frac{|\nabla\rho|^2}{\rho^2}\leq C\theta,
$
the last term satisfies
\begin{equation*}
 \left|
 \int_{\mathbb R^d}a\nabla u\cdot\nabla\rho\,u\,dx
 \right|
 \leq
 \eta\langle|\nabla u|^2\rangle_\rho
 +C_\eta\theta\langle u^2\rangle_\rho.
\end{equation*}
Thus, after first fixing $\eta>0$ sufficiently small, the terms
containing $\nabla\rho$ only produce an arbitrarily small perturbation
of the energy term $\langle|\nabla u|^2\rangle_\rho$, as well as an additional term of the form $C_\eta\theta\langle u^2\rangle_\rho.$ Now, we define
\begin{equation*}
 F(t):= \lambda\int_0^t \bigl[g_{\delta_a}(s)+g_\nu(s)+C_\eta\theta\bigr]\,ds, \quad G(t):=
 \lambda\int_{T_0}^t \bigl[g_{\delta_a}(s)+g_\nu(s)+C_\eta\theta\bigr]\,ds.
\end{equation*}
Multiplying the equations by
$e^F\widetilde u\,\rho$ and $e^Gu_k\,\rho$, respectively, repeat Steps $1'$-$3'$. The previous
estimates show that, provided $\eta$ and then $\theta$ are chosen
sufficiently small, the constants $C_1$ and $C_2$ are changed only by
arbitrarily small values. In particular, they can still be chosen so
that  $0<\frac{C_2}{C_1}<1$,
and $C_2/C_1$ can still be made arbitrarily small by requiring
the form-bounds $\delta_a$ and $\nu$ to be sufficiently small. We
therefore obtain
\begin{equation*}
 \langle|\tilde u(0)|^2\rangle_\rho \leq
 C\left(\int_{T_0}^{T_1}e^G\langle|\nabla u_2|^2\rangle_\rho dt + \int_{T_0}^{T_1}g_\nu e^G\langle u_2^2\rangle_\rho \right)
\end{equation*}
and
\begin{align*}
&\int_{T_0}^{T_1}e^G\langle|\nabla u_2|^2\rangle_\rho\,dt + \int_{T_0}^{T_1}g_\nu e^G\langle u_2^2\rangle_\rho dt \\
&\qquad\leq \left(\frac{C_2}{C_1}\right)^{n-2} \left[ \int_{T_0}^{T_1}e^G\langle|\nabla u_n|^2\rangle_\rho dt + \int_{T_0}^{T_1}g_\nu e^G\langle u_n^2\rangle_\rho\,dt \right].
\end{align*}
It remains to estimate the last expression. The weighted version of
Step $3'$ gives
$$
\int_{T_0}^{T_1}e^G\langle|\nabla u_n|^2\rangle_\rho dt + \int_{T_0}^{T_1}g_\nu e^G\langle u_n^2\rangle_\rho dt \leq C\int_{T_0}^{T_1}e^G\langle f_n^2\rangle_\rho\,dt.
$$
Here we no longer need the compact support assumption of $f_i$. Namely, applying the
form-boundedness inequality for $f_n$ to $\sqrt{\rho}$, we obtain
\begin{align*}
\langle f_n^2(t)\rangle_\rho
&\leq \nu\langle|\nabla\sqrt{\rho}|^2\rangle+g_\nu(t)\langle\rho\rangle \\
&\leq C\nu\theta+Cg_\nu(t),
\end{align*}
where we used $\langle\rho\rangle<\infty$ and $
|\nabla\sqrt{\rho}|^2 =\frac{|\nabla\rho|^2}{4\rho} \leq C\theta\rho$.
Since $G$ is non-negative and bounded on $[T_0,T_1]$, H\"older's
inequality now yields
\begin{align*}
\int_{T_0}^{T_1}e^G\langle f_n^2\rangle_\rho\,dt
&\leq C\left[\nu\theta(T_1-T_0)+\int_{T_0}^{T_1}g_\nu(t) dt \right] \\ 
&\leq C(T_1-T_0)^{\frac{\varepsilon}{1+\varepsilon}},
\end{align*}
where $T$ is fixed and $g_\nu\in L_{\loc}^{1+\varepsilon}(\mathbb R)$.

Combining the previous estimates, we obtain
\begin{equation*}
\left\|
\mathbf E\int_{\Delta_n(T_0,T_1)}\prod_{i=1}^n|f_i(t_i,X_{t_i}^{\,\cdot})|^2 dt_1\dots dt_n 
\right\|_{L^2_\rho(\mathbb R^d)}^2 \leq C_0K^{n-2} (T_1-T_0)^{\frac{\varepsilon}{1+\varepsilon}},
\end{equation*}
with $K=C_2/C_1$. The proof of Proposition \ref{prop1} is completed.
\end{proof}

\bigskip

\section{Proof of Theorem \ref{thm1}(\textit{i}): strong existence}
\label{section_3}

\begin{proposition}
\label{prop2}
Let $X_t^x=X_{0,t}^x$ be the strong solution to \eqref{sde0}. Then, 
for every integer $p \geq 2$, there exist constants $K_1$, $K_2$ independent of the smoothness of $\sigma$ such that

\begin{enumerate}[label=\rm(\textit{\roman*})]
\item
\begin{equation}
\label{compact_crit1}
 \sup_{y \in \mathbb R^d}\|\nabla X_t^x - I\|_{L^{2p}(B_1(y), \, L^p(\Omega))} \leq K_1 t^{\frac{1}{4p}}
\end{equation}
for all $0 \leq t \leq T$;
 
\item
\begin{equation}
\label{compact_crit2}
 \sup_{y \in \mathbb R^d} \|D_sX_t^x - \sigma(s, X_s^x)\|_{L^{2p}(B_1(y), \, L^p(\Omega))} \leq K_1|t-s|^{\frac{1}{4p}}
\end{equation}
for a.e. $s \in [0,T]$ and $0 \leq s \leq t \leq T$;

\item
\begin{equation}
\label{compact_crit3}
 \sup_{y \in \mathbb R^d} \|D_s X_t^x - D_{s'}X_t^x\|_{L^{2p}(B_1(y), \, L^p(\Omega))} \leq K_2 \big(|s'-s|^{\frac{1}{8p}} +  |s'-s|^{\frac{\gamma}{2p}}\big)
\end{equation}
for a.e. $s,s' \in [0,T]$ and $0 \leq s,s' \leq t \leq T$.
\end{enumerate}
\end{proposition}

\begin{proof}[Proof of Proposition \ref{prop2}] We will be applying Corollary \ref{cor1}. Without loss of generality, $y=0$. (In the general case $y \in \mathbb R^d$ we would have to apply Corollary \ref{cor1} with the translated weight $\rho_y$, see Remark \ref{trans_rem}.)

(\textit{i}) We will consider initial time $0 \leq s<T$, i.e.\,$X_{s,s}^x=x$; this will allow us to reuse the proof of (\textit{i}) in the proof of (\textit{ii}). Since $\sigma:[0,T]\times\mathbb{R}^d\to \mathbb{R}^{d\times d}$ is smooth with bounded derivatives, we can differentiate \eqref{sde0} in $x$:
\begin{equation}
\label{nabla_sde}
\nabla X_{s,t}^x-I =  \int_s^t \nabla \sigma (r,X_r^x)\cdot \nabla X_{s,r}^x dW_r, 
\end{equation}
where $\nabla \sigma(r,x) \equiv \nabla_x\sigma(r,x) = \left( \partial_j\sigma_{ik}(r,x) \right)_{1\leq i,j,k\leq d}
$
 and, in coordinates, 
$$
(\nabla_x\sigma)(r,X_r^x)\cdot\nabla_xX_{s,r}^x := \left( \sum_{l=1}^{d} \partial_l\sigma_{ik}(r,X_r^x)\, \partial_jX_{s,r}^{x,l} \right)_{1\leq i,j,k\leq d}.
$$
So, with the Brownian motion, we have
$$
(\nabla_x\sigma)(r,X_r^x) \cdot\nabla_xX_{s,r}^x\,dW_r = \left( \sum_{k=1}^{d}\sum_{l=1}^{d} \partial_l\sigma_{ik}(r,X_r^x)\,
\partial_jX_{s,r}^{x,l} dW_r^k \right)_{1\leq i,j\leq d}.
$$
Put for brevity
$$
J_{s,t}^x:=\nabla_x X_{s,t}^x, \quad J_{s,s}^x=I, 
$$
We estimate
\begin{align*}
\|J_{s,t}^x - I\|^{2p}_{L^{2p}((B_1(0),L^p(\Omega))} & = \int_{B_1(0)}\bigl[\mathbf{E}|J_{s,t}^x - I|^p \bigr]^2dx \\
& (\text{we use \eqref{nabla_sde} and then apply the BDG  inequality}) \\
& \leq C\int_{B_1(0)} \biggl[ \mathbf{E}\biggl(\int_s^t |\nabla \sigma(r,X_{s,r}^x)|^2 |J_{s,r}^x|^2 dr \biggr)^{\frac{p}{2}}\biggr]^2 dx \\
& \leq C \int_{B_1(0)} \biggl[ \mathbf{E}\sup_{s \leq r \leq t}|J^x_{s,r}|^p \biggl(\int_s^t |\nabla \sigma(r,X_{s,r}^x)|^2  dr \biggr)^{\frac{p}{2}}\biggr]^2 dx \\
& (\text{use Cauchy-Schwarz twice}) \\
& \leq C \int_{B_1(0)} \mathbf{E}\sup_{s \leq r \leq t}|J^x_{s,r}|^{2p}\, \mathbf{E}\biggl(\int_s^t |\nabla \sigma(r,X_{s,r}^x)|^2  dr \biggr)^p  dx \\
& \leq C \biggl(\int_{B_1(0)} \biggl[\mathbf{E}\sup_{s \leq r \leq t}|J^x_{s,r}|^{2p}\biggr]^2 dx\biggr)^{\frac{1}{2}} \biggl(\int_{B_1(0)} \biggl[ \mathbf{E}\biggl(\int_s^t |\nabla \sigma(r,X_{s,r}^x)|^2  dr \biggr)^p \biggr]^2 dx \biggr)^{\frac{1}{2}} \\
& (\text{apply Corollary \ref{cor1} the the second factor}) \\
& \leq C \biggl(\int_{B_1(0)} \biggl[\mathbf{E}\sup_{s \leq r \leq t}|J^x_{s,r}|^{2p}\biggr]^2 dx\biggr)^{\frac{1}{2}} C_1 t^{\frac{1}{2}},
\end{align*}
where $C_1:=C_0^{\frac{1}{2}}p!  K^{\frac{p-2}{2}}$ (since $p$ is fixed, the growth of $C_1$ in $p$ plays no role here). This yields assertion (\textit{i}) once we prove that the first factor is bounded from above by a constant independent of the smoothness of $\sigma$. To this end, we will use the following stochastic Gronwall-type estimate:

\begin{claim}
\label{claim_gronwall}

$$
\biggl[\mathbf{E}\sup_{s \leq r \leq t}|J^x_{s,r}|^{2p}\biggr]^2 \leq C_p \mathbf E e^{c_p\int_s^t |\nabla \sigma(r,X^x_{s,r})|^2 dr }.
$$

\end{claim}
\begin{proof}[Proof of Claim \ref{claim_gronwall}] Put for brevity $J_t:=J_{s,t}^x$ and  $V_t:=|J_t|^2$, where $x \in B_1(0)$ is fixed. Our goal is to estimate
$\mathbf{E}\sup_{s \leq r \leq t}|V_r|^{p}$. Identity \eqref{nabla_sde} becomes
$dJ_t=\sum_{l=1}^d A^l_t J_t dW^l_t$, where $A^l_t=(\nabla_x \sigma_{\cdot l})(t,X_{s,t}^x)$ (a $d \times d$ matrix field). By It\^{o}'s formula, 
\begin{align}
\label{Ito_V}
dV_t =  2\sum_{l=1}^d J_t\cdot\left(A_t^l J_t\right) dW_t^l + \sum_{l=1}^d |A_t^l J_t|^2 dt, 
\end{align}
where $\cdot$ is the Euclidean inner product in $\mathbb{R}^{d\times d}$, i.e.\,in coordinates,
\begin{align*}
J_t\cdot\left(A_t^l J_t\right)dW_t^l & = \sum_{i,j,q=1}^d (J_t)_{ij} (A_t^l)_{iq} (J_t)_{qj}\,dW_t^l \\
& = \sum_{i,j,q=1}^d \partial_jX_{s,t}^{x,i} \partial_q\sigma_{il}(t,X_{s,t}^x) \partial_jX_{s,t}^{x,q}dW_t^l.
\end{align*}
Since $\sigma$ is smooth, $J_t$ is invertible (for the proof, if needed, see Appendix \ref{J_app})
 and so $V_t>0$. Therefore, the identity \eqref{Ito_V} can be rewritten as
 $$
dV_t=q_t V_t dt + V_t dM_t,
$$
where
\begin{equation}
\label{q_t}
q_t:=\frac{\sum_{l=1}^d |A_t^l J_t|^2}{|J_t|^2} \leq |\nabla \sigma(t,X_{s,t}^x)|^2, 
\end{equation}
and
\begin{equation}
d M_t:=c_t dW_t, \quad c_t=(c_t^1,\dots,c_t^d) \quad \text{ with } c_t^l:=2\frac{J_t \cdot A^l_tJ_t}{|J_t|^2}, \quad |c_t|^2 \leq 4|\nabla \sigma(t,X_{s,t}^x)|^2. \label{M_t}
\end{equation}
Since $\nabla_x\sigma$ is bounded, cf.\,the beginning of this section, and thus so is $|c|$, the It\^{o} integral $ M_t=\int_s^t c_r\cdot dW_r =\sum_{l=1}^d\int_s^t c_r^l\,dW_r^l $, $t\leq T$, is  a continuous square-integrable martingale. Applying It\^{o}'s formula to $\log V_t$, 
\begin{align*}
d\log V_t &= \frac{1}{V_t} dV_t -\frac{1}{2V_t^2} d[ V ]_t \\
&= \frac{1}{V_t} \left(q_tV_t\,dt+V_t dM_t\right) -\frac{1}{2V_t^2} V_t^2 d[ M]_t\\
&= q_t dt+dM_t-\frac12\,d[ M]_t. 
\end{align*}
Then, we obtain the following representation for $V_t$:
$$
V_t=V_s e^{M_t-M_s-\frac{1}{2}([M]_t-[M]_s) + \int_s^t q_r dr},
$$
Hence
$$
V_t^p=V_s^p e^{p(M_t-M_s)-\frac{p}{2}([M]_t-[M]_s) + p\int_s^t q_r dr},
$$
which we rewrite, for a $\lambda>0$ to be chosen, as
$$
V_t^p=V_s^p (Z_t^\lambda)^{\frac{p}{\lambda}} e^{\frac{1}{2}(p\lambda-p)([M]_t-[M]_s) + p\int_s^t q_rdr},$$
where 
$
Z_t^\lambda=e^{\lambda(M_t-M_s)-\frac{\lambda^2}{2}([M]_t - [M]_s)}. 
$
 The process $ Z^\lambda$ is a positive continuous  martingale, $Z_s^\lambda = 1$ -- Novikov's condition is trivially satisfied for $M_t$ due to $d[ M ]_t=|c_t|^2dt$, the estimates in \eqref{q_t}, \eqref{M_t} and the boundedness $|\nabla \sigma|$. Also by \eqref{q_t}, \eqref{M_t}, 
$$
V_t^p \leq V_s^p (Z_t^\lambda)^{\frac{p}{\lambda}} e^{c_p\int_s^t |\nabla \sigma(r,X^x_{s,r})|^2 dr}.
$$ 
It easily follows that 
$$
\sup_{s \leq r \leq t} V_r^p \leq V_s^p \left(\sup_{s \leq r \leq t }Z_r^\lambda\right)^{\frac{p}{\lambda}} e^{c_p\int_s^t |\nabla \sigma(r,X^x_{s,r})|^2 dr}.
$$
Recalling that $J_s=I$, that $V_s$ is a dimension-dependent constant, and applying the Cauchy-Schwarz inequality, we obtain
\begin{equation}
\label{Vp_est}
\mathbf{E} \sup_{s \leq r \leq t} V_r^p \leq C \bigg(\mathbf{E}\bigg(\sup_{s \leq r \leq t }Z_r^\lambda\bigg)^{\frac{2p}{\lambda}}\bigg)^{\frac{1}{2}} \bigg( \mathbf{E} e^{2c_p\int_s^t |\nabla \sigma(r,X^x_{s,r})|^2 dr}\bigg)^{\frac{1}{2}}.
\end{equation}
Since $Z^\lambda$ is a positive continuous martingale with $Z_s^\lambda=1$, Doob's weak maximal inequality yields
$$
\mathbf{P}\left( \sup_{s\leq r\leq t}Z_r^\lambda\geq a \right) \leq \frac{\mathbf{E}Z_s^\lambda}{a} =
\frac{1}{a}, \qquad a>0.
$$
Furthermore, clearly, $ \mathbf{P}\left( \sup_{s\leq r\leq t}Z_r^\lambda\geq a \right)
\leq 1 \wedge \frac{1}{a}$. Now, applying the Cavalieri principle to $\sup_{s\leq r\leq t}Z_r^\lambda$ gives  
\begin{align*}
\mathbf{E}\bigg(\sup_{s \leq r \leq t }Z_r^\lambda\bigg)^{\frac{2p}{\lambda}}   &=
\frac{2p}{\lambda} \int_0^\infty a^{\frac{2p}{\lambda}-1} \mathbf{P}\left(\sup_{s\leq r\leq t}Z_r^\lambda\geq a \right) da\\
&\leq \frac{2p}{\lambda} \int_0^1a^{\frac{2p}{\lambda}-1} da + \frac{2p}{\lambda} \int_1^\infty a^{\frac{2p}{\lambda}-2}\,da\\
&=
1+\frac{\frac{2p}{\lambda}}{1-\frac{2p}{\lambda}}\\
&=
\frac1{1-\frac{2p}{\lambda}}.
\end{align*}
This, and the previous estimate \eqref{Vp_est}, yield Claim \ref{claim_gronwall}, provided that $\lambda$ is fixed sufficiently large.
\end{proof}

Armed with Claim \ref{claim_gronwall}, we can now show that $\mathbf{E}\sup_{s \leq r \leq t}|J^x_{s,r}|^{2p}$ is bounded by a constant independent of smoothness of $\sigma$. Namely, we bound the right-hand side of the estimate in Claim \ref{claim_gronwall} as follows. Write
\begin{align*}
\mathbf E e^{c_p\int_s^t |\nabla \sigma(r,X_{s,r})|^2 dr }=\sum_{n=0}^\infty \frac{c_p^n}{n!}\mathbf E \biggl(\int_s^t |\nabla \sigma(r,X_{s,r})|^2 dr  \biggr)^n.
\end{align*}
Then, using Cauchy-Schwarz,
\begin{align*}
\int_{B_1(0)} \mathbf E e^{c_p\int_s^t |\nabla \sigma(r,X_{s,r})|^2 dr } dx & \leq \sum_{n=0}^\infty \frac{c_p^n}{n!} |B_1(0)|^{\frac{1}{2}} \biggl(\int_{B_1(0)}\biggl[\mathbf E \biggl(\int_s^t |\nabla \sigma(r,X_{s,r})|^2 dr  \biggr)^n\biggr]^2 dx
\biggr)^{\frac{1}{2}} \\
& (\text{apply Corollary \ref{cor1} to the $n \geq 2$ terms}) \\
& \leq (\text{term $n=0$}) + (\text{term $n=1$}) + \sum_{n=2}^\infty\frac{c_p^n}{n!}|B_1(0)| C^{\frac{1}{2}}_0 n!
K^{\frac{n-2}{2}}t^{\frac{1}{2}}.
\end{align*}
The term corresponding to $n=0$ is, clearly, finite and independent of the smoothness of $\sigma$. The term $n=1$ is bounded by the term for $n=2$ by applying elementary inequality $\alpha \leq 1+\alpha^2$ for all $\alpha \geq 0$.
Regarding the terms corresponding to $n \geq 2$, i.e.\,the convergence of the series, it only remains to note that, by Corollary \ref{cor1}, the constant $K$ can be made as small as needed assuming that the form-bound $\nu$ is sufficiently small. 

Thus, $\mathbf{E}\sup_{s \leq r \leq t}|J^x_{s,r}|^{2p} dx$ is bounded by a constant independent of the smoothness of $\sigma$, from which assertion (\textit{i}) follows upon selecting everywhere above $s=0$.

\medskip

(\textit{ii}) The Malliavin derivative $D_sX_t^x$, $0 \leq s<t \leq T$, satisfies  SDE of the same type as \eqref{nabla_sde}:
$$
D_sX_t^x=\sigma(s,X_s^x)+\int_s^t \nabla \sigma(\tau,X_\tau^x)\cdot D_s X_\tau^x dW_\tau.
$$
So, we can repeat the proof of (\textit{i}) word by word, except that now the analogue of Claim \ref{claim_gronwall} takes form
$$
\biggl[\mathbf{E}\sup_{s \leq r \leq t}|D_{s}X_r^x|^{2p}\biggr]^2 \leq C_p \|\sigma\|^{4p}_\infty \mathbf E e^{c_p\int_s^t |\nabla \sigma(r,X^x_{s,r})|^2 dr }
$$
since the initial value for $D_s X_r^x$ at $r=s$ is $\sigma(s,X_s^x)$.

\medskip

(\textit{iii}) Set
$
\hat{J}_{s,t}^x:=J_{s,t}^{X_s^x}.
$
Let $0 \leq s < s' \leq t \leq T$. We note the standard identities
$$
D_s X_t^x=\hat{J}_{s',t}^x D_s X_{s'}^x
$$
and
\begin{equation}
\label{k2}
D_{s'}X_t^x=\hat{J}_{s',t} \sigma(s',X_{s'}^x).
\end{equation}
Hence
$$
D_s X_t^x - D_{s'}X_t^x = \hat{J}_{s',t}^x \big( D_s X_{s'}^x - \sigma(s',X_{s'}^x) \big).
$$
Apply the Cauchy-Schwarz inequality to the right-hand side, 
\begin{align}
\|\hat{J}_{s',t}^x \big( D_s X_{s'}^x & - \sigma(s',X_{s'}^x) \big)\|_{L^{2p}(B_1(0),L^{p}(\Omega))} \notag \\
& \leq \|\hat{J}_{s',t}^x\|_{L^{4p}(B_1(0),L^{2p}(\Omega))}\|D_s X_{s'}^x - \sigma(s',X_{s'}^x)\|_{L^{4p}(B_1(0),L^{2p}(\Omega))}. \label{iii_1}
\end{align}

1. The norm of the first factor
$
\|\hat{J}_{s',t}^x\|_{L^{4p}(B_1(0),L^{2p}(\Omega))}
$
is bounded from above by a constant independent of the smoothness of $\sigma$, which follows, via \eqref{k2}, from the uniform ellipticity condition $\sigma \sigma^{\top} \geq \kappa>0$, and thus $\|\sigma^{-1}\|_\infty \leq \kappa^{-\frac{1}{2}}$, and the bound on $\|D_{s'}X_t^x\|_{L^{4p}(B_1(0),L^{2p}(\Omega))}$ stemming from assertion (\textit{ii}) with $p$ replaced by $2p$.

\medskip

2. Next, we estimate the second factor in \eqref{iii_1}:
\begin{align}
\|D_s X_{s'}^x & - \sigma(s',X_{s'}^x)\|_{L^{4p}(B_1(0),L^{2p}(\Omega))} \notag \\
& \leq \|D_s X_{s'}^x - \sigma(s,X_{s}^x)\|_{L^{4p}(B_1(0),L^{2p}(\Omega))} + \| \sigma(s,X_{s}^x) - \sigma(s',X_{s'}^x)\|_{L^{4p}(B_1(0),L^{2p}(\Omega))}. \label{D2}
\end{align}

(a) The first term in \eqref{D2} has  the Malliavin derivative and $\sigma$  evaluated at the same time $s$. So, this term can be estimated by applying assertion (\textit{ii}) with $p$ replaced by $2p$: 
$$
 \|D_s X_{s'}^x - \sigma(s,X_{s}^x)\|_{L^{4p}(B_1(0),L^{2p}(\Omega))} \leq K_1|s-s'|^{\frac{1}{8p}}.
$$

(b) Let us handle the second term in \eqref{D2}:
\begin{align*}
\| \sigma(s,X_{s}^x) & - \sigma(s',X_{s'}^x)\|_{L^{4p}(B_1(0),L^{2p}(\Omega))}^{4p} \\
& \leq \| \sigma(s,X_{s}^x) - \sigma(s',X_{s}^x)\|_{L^{4p}(B_1(0),L^{2p}(\Omega))}^{4p} + \| \sigma(s',X_{s}^x) - \sigma(s',X_{s'}^x)\|_{L^{4p}(B_1(0),L^{2p}(\Omega))}^{4p} \\
& \leq 2^{4p-2}\|\sigma\|_\infty^{4p-2}\big\langle \big(\mathbf{E}|\sigma(s,X_s^x)-\sigma(s',X_s^x)|\big)^2\big\rangle_{x \in B_1(0)} + \| \sigma(s',X_{s}^x) - \sigma(s',X_{s'}^x)\|_{L^{4p}(B_1(0),L^{2p}(\Omega))}^{4p} \\
& =:I_1 + I_2.
\end{align*}

\begin{itemize}

\item[b.1] We have
\begin{align*}
(\text{constant times})\; I_1 \leq C_0 \big\langle \big(\mathbf{E}|\sigma(s,X_s^x)-\sigma(s',X_s^x)|\big)^2 \rho(x)\big\rangle_{x \in \mathbb R^d} & = \|P^{0,s}|\sigma(s,\cdot)-\sigma(s',\cdot)|\|^2_{L^2_\rho(\mathbb R^d)} \\
& (\text{apply Proposition \ref{contr_prop})}) \\
& \leq C\|\sigma(s,\cdot)-\sigma(s',\cdot)\|^2_{L_\rho^2(\mathbb R^d)} \\
& (\text{use ($\mathbf{H}$) in Theorem \ref{thm1}}) \\
& \leq C_2|s-s'|^{2\gamma},
\end{align*}
where the constant $C$ does not depend on the smoothness of $\sigma$.

\item[b.2] Next, let us estimate $I_2$. Once again, we are looking for an upper bound in terms of a positive power of $|s-s'|$, with constants that do not depend on the smoothness of $\sigma$. Denote by $\sigma_{\eta}$, $\eta \downarrow 0$, the standard mollifier in the spatial variable applied to $\sigma$. We write
\begin{align}
\| \sigma(s',X_{s}^x) & - \sigma(s',X_{s'}^x)\|_{L^{4p}(B_1(0),L^{2p}(\Omega))} \notag \\
& \leq \| \sigma(s',X_{s}^x) - \sigma_\eta(s',X_{s}^x)\|_{L^{4p}(B_1(0),L^{2p}(\Omega))} \tag{$K_1$}\label{s1} \\
&+ \| \sigma_\eta(s',X_{s}^x) - \sigma_\eta(s',X_{s'}^x)\|_{L^{4p}(B_1(0),L^{2p}(\Omega))} \tag{$K_2$} \label{s2} \\
&+ \| \sigma_\eta(s',X_{s'}^x) - \sigma(s',X_{s'}^x)\|_{L^{4p}(B_1(0),L^{2p}(\Omega))} \tag{$K_3$} \label{s3}.
\end{align}
The terms \eqref{s1} and \eqref{s3} are estimated by essentially arguing as in (b.1), e.g.\,for \eqref{s1}:
\begin{align*}
\| \sigma(s',X_{s}^x) - \sigma_\eta(s',X_{s}^x)\|_{L^{4p}(B_1(0),L^{2p}(\Omega))}^{4p} & \leq 2^{4p-2}\|\sigma\|_\infty^{4p-2}\big\langle \big(\mathbf{E}|\sigma(s',X_s^x)-\sigma_\eta(s',X_s^x)|\big)^2\big\rangle_{x \in B_1(0)} \\
& = 2^{4p-2}\|\sigma\|_\infty^{4p-2}\|P^{0,s}|\sigma(s',\cdot)-\sigma_\eta(s',\cdot)|\|_{L^2(B_1(0))}^2 \\ 
& \leq C \|P^{0,s}|\sigma(s',\cdot)-\sigma_\eta(s',\cdot)|\|_{L_\rho^2(\mathbb R^d))}^2 \\
& (\text{use Proposition \ref{contr_prop}}) \\
& \leq C \|\sigma(s',\cdot)-\sigma_\eta(s',\cdot)\|_{L_\rho^2(\mathbb R^d)}^2 \\
& (\text{use Lemma \ref{approx_lem}}) \\
& \leq C_1 \eta.
\end{align*}
Therefore, 
$$
\eqref{s1} \leq C_1 \eta^{\frac{1}{2p}}.
$$

Similarly,
$$
\eqref{s3} \leq C_1 \eta^{\frac{1}{2p}}.
$$

Let us deal with \eqref{s2}. We have $$| \sigma_\eta(s',X_{s}^x) - \sigma_\eta(s',X_{s'}^x)| \leq \|\nabla \sigma_{\eta}(s',\cdot)\|_{L^\infty(\mathbb R^d)} |X_s^x-X_{s'}^x|.$$ Hence, using $\|\nabla \sigma_\eta(s',\cdot)\|_{L^\infty(\mathbb R^d)} \leq C_1\eta^{-1}\|\sigma(s')\|_{L^\infty(\mathbb R^d)} \leq C \eta^{-1}$, we obtain
\begin{align*}
\mathbf{E} |\sigma_\eta(s',X_{s}^x) - \sigma_\eta(s',X_{s'}^x)|^{2p} & \leq C^{2p} \eta^{-2p} \mathbf E|X_s^x-X_{s'}^x|^{2p} \\
& = C^{2p} \eta^{-2p} \mathbf{E}\left|\int_{s}^{s'}\sigma(r,X_r^x)dW_r\right|^{2p} \quad \text{(apply BDG inequality)}\\
& \leq C_2 \eta^{-2p}|s'-s|^{p}.
\end{align*}
We arrive at the following estimate:
$$
\eqref{s2} \leq C_3 \eta^{-1}|s'-s|^{\frac{1}{2}}.
$$
\end{itemize}
Combining the estimates on \eqref{s1}-\eqref{s3}, we obtain
$$
I_2^{\frac{1}{4p}} \leq 2C_1\eta^{\frac{1}{2p}} + C_2 \eta^{-1}|s'-s|^{\frac{1}{2}}.
$$
The latter, and the estimate on $I_1$, allow us to bound the second term in  \eqref{D2} as follows:
\begin{align*}
\| \sigma(s,X_{s}^x) - \sigma(s',X_{s'}^x)\|_{L^{4p}(B_1(0),L^{2p}(\Omega))} & \leq I_1^{\frac{1}{4p}} + I_2^{\frac{1}{4p}} \\
& \leq C\bigl( |s'-s|^{\frac{\gamma}{2p}} + \eta^{\frac{1}{2p}} + \eta^{-1}|s'-s|^{\frac{1}{2}} \bigr).
\end{align*}
The left-hand side does not depend on $\eta$, so it is for us to select. We take $\eta=|s'-s|^{\frac{1}{4}}$, thus arriving at
\begin{align*}
\| \sigma(s,X_{s}^x) - \sigma(s',X_{s'}^x)\|_{L^{4p}(B_1(0),L^{2p}(\Omega))} & \leq C(|s'-s|^{\frac{\gamma}{2p}} + |s'-s|^{\frac{1}{8p}} + |s'-s|^{\frac{1}{4}}) \\
& (\text{use the fact that we work on a finite time interval}) \\
& \leq C_1(|s'-s|^{\frac{\gamma}{2p}} + |s'-s|^{\frac{1}{8p}}).
\end{align*}

\medskip

Finally, returning to \eqref{D2}, we can estimate
$$
\|D_s X_{s'}^x - \sigma(s',X_{s'}^x)\|_{L^{4p}(B_1(0),L^{2p}(\Omega))} \leq K_1 |s'-s|^{\frac{1}{8p}} +  C_1(|s'-s|^{\frac{\gamma}{2p}} + |s'-s|^{\frac{1}{8p}}).
$$
It remains to use this in \eqref{iii_1} to end the proof of (\textit{iii}).
\end{proof}

Armed with Propositions \ref{prop_grad} and \ref{prop2}, we can now prove Theorem \ref{thm1}(\textit{i}), i.e.\,strong existence. We follow the approach of \cite{RZ}. Since $\sigma_m$ are smooth, by classical theory there exists a unique continuous random field  $X^m:\Delta_2(T) \times \mathbb R^d \times \Omega \rightarrow \mathbb R^d$ such that
\begin{equation}
\label{sde2sigma}
X_{s,t}^{x,m}=x+ \int_s^t \sigma_m(r,X^{x,m}_{s,r})dW_r, \quad 0 \leq s \leq t \leq T, \quad x \in \mathbb R^d.
\end{equation}

\begin{remark}
\label{flow_rem}
We will use the follow standard properties of the approximating SDEs which are obtained from the Picard construction and pathwise uniqueness. For every $m$, the continuous version of
the random field
$
\bigl\{X_{r,t}^{x,m} \mid 0\leq r\leq t\leq T, x\in\mathbb R^d\bigr\}
$
may be chosen so that, on an event of probability one,
\begin{equation}
\label{approx_flow}
X_{r,t}^{x,m}=X_{s,t}^{X_{r,s}^{x,m},m},
\quad
0\leq r\leq s\leq t\le T, \quad x\in\mathbb R^d.
\end{equation}
Moreover, for every $r\leq t$ and $x\in\mathbb R^d$, $X_{r,t}^{x,m}$ is measurable with respect to $\mathcal F_{r,t}^W= \sigma(W_u-W_r \mid r\leq u\leq t)$
after completion. In particular, for fixed $s\leq t$, the random
map
$$
\mathbb R^d\ni y\longmapsto X_{s,t}^{y,m}
$$
is measurable with respect to $\mathcal F_{s,t}^W$ as a
$[C_{\rm loc}(\mathbb R^d)]^d$-valued random
variable and is independent of $\mathcal F_s^W$.
\end{remark}

\textbf{Step 1.}~We make a few preliminary observations.
By the standard Burkholder-Davis-Gundy inequality and the uniform (in $m$) boundedness of $\sigma_m$ on $\mathbb R^d$, 
\begin{align}
\mathbf{E}|X^{x,m}_{s,t_2}-X^{x,m}_{s,t_1}|^r &= \mathbf{E} \biggl| \int_{t_1}^{t_2} \sigma_m (t, X_{s,t}^{x,m}) \, dW_t \biggr|^r \notag \\
&\leq C_r \mathbf{E} \biggl| \int_{t_1}^{t_2} |\sigma_m (t, X_{s,t}^{x,m})|^2 \, dt \biggr|^{\frac{r}{2}} \leq  C(t_2 -t_1)^{\frac{r}{2}}. \label{est.3}
\end{align}
In particular, for $0 \leq s \leq t \leq T$,
\begin{equation}
\label{est.4}
\sup_m\sup_{x \in \mathbb R^d}\mathbf E|X_{s,t}^{x,m}-x|^r \leq C(t-s)^{\frac{r}{2}}.
\end{equation}
Next, one has 
\begin{equation}
\label{est.5}
\sup_m\sup_{0\leq s \leq t \leq T}\sup_{y \in \mathbb R^d} \mathbf{E}\int_{B_1(y)} |\nabla_x X^{x,m}_{s,t}|^r dx<\infty.
\end{equation}
Indeed,
\begin{align*}
\mathbf{E}\int_{B_1(y)} |\nabla_x X^{x,m}_{s,t}|^r dx
& \leq \left(\int_{B_1(y)} \bigl(\mathbf{E} |\nabla_x X^{x,m}_{s,t}|^r\bigr)^2 dx \right)^{\frac{1}{2}} |B_1(y)|^{\frac{1}{2}}  = C_d \|\nabla X^{\cdot,m}_{s,t}\|^r_{L^{2r}(B_1(y),L^r(\Omega))},
\end{align*}
so it remains to apply Proposition \ref{prop2}(\textit{i}).

\medskip
Let now $r>d$. (At Steps 4 and 5 we will have to fix  $r$ even larger).

\medskip

\textbf{Step 2.}~Let us compare solutions with different initial points, but the same initial time. Applying the Sobolev embedding theorem to \eqref{est.5}, we obtain the following bound on the H\"{o}lder seminorm
$$
\sup_m\sup_{0 \leq s \leq t \leq T} \sup_{y \in \mathbb R^d} \mathbf{E}[X_{s,t}^{\cdot,m}]^r_{C^{1-\frac{d}{r}}(B_1(y))}<\infty,
$$ 
Therefore,
for all $0 \leq s \leq t \leq T$, $x,y \in \mathbb R^d$ with $|x-y| \leq 1$, 
\begin{equation*}
\mathbf{E}|X^{x,m}_{s,t}-X^{y,m}_{s,t}|^r \leq C|x-y|^{r-d}, \quad r-d>0.
\end{equation*}
For $|x-y|>1$, we estimate, using \eqref{est.4}
$$
\mathbf{E}|X^{x,m}_{s,t}-X^{y,m}_{s,t}|^r \leq C_r (|x-y|^r + \mathbf{E}|X^{x,m}_{s,t}-x|^r + \mathbf{E}|X^{y,m}_{s,t}-y|^r) \leq C|x-y|^r.
$$
Combining the last two bounds, we obtain, for all $0 \leq s \leq t \leq T$, $x,y \in \mathbb R^d$,
\begin{equation}
\label{est.6}
\mathbf{E}|X^{x,m}_{s,t}-X^{y,m}_{s,t}|^r \leq C(|x-y|^{r-d} + |x-y|^r).
\end{equation}

\textbf{Step 3.}~Next, we compare solutions with the same initial points, but different initial times. That is, assume that $0 \leq  s_1 \leq  s_2 \leq  t$. 
By the flow property of the approximating solutions (see Remark \ref{flow_rem}),
$$
X_{s_1,t}^{x,m}=X_{s_2,t}^{X_{s_1,s_2}^{x,m},m}.
$$
Moreover, the map
$y\mapsto X_{s_2,t}^{y,m}$ is independent of $\mathcal F_{s_2}^W$, while $X_{s_1,s_2}^{x,m}$ is $\mathcal F_{s_2}^W$-measurable. Therefore,
$$
\mathbf E\left|X_{s_1,t}^{x,m}-X_{s_2,t}^{x,m}\right|^r=\mathbf E\big[\mathbf E\big|X_{s_2,t}^{y,m}-X_{s_2,t}^{x,m}\big|^r \big|_{y=X_{s_1,s_2}^{x,m}}\big].
$$
Hence
\begin{align}
\label{est.7}
\mathbf{E} | X_{s_1,t}^{x,m}-X_{s_2,t}^{x,m} |^r  &\leq  C_r \biggl(\mathbf{E}  \bigl| \int_{s_1}^{s_2}\sigma_m (s,X^{x,m}_{s_1,s}) \, dW_s \bigr |^r 
+ \mathbf{E}  \bigl|    \int_{s_2}^{t} [\sigma_m (s,X^{x,m}_{s_1,s}) - \sigma_m (s,X^{x,m}_{s_2,s}) ]\, dW_s \bigr|^r\biggr)  \notag \\
& (\text{apply \eqref{est.3} to the first term}) \\
  &\leq   C_r \biggl((s_2-s_1)^{\frac{r}{2}}  
 + \mathbf{E}  \bigl|    \int_{s_2}^{t} \bigl[\sigma_m (s,X^{X^{x,m}_{s_1,s_2}}_{s_2,s}) - \sigma_m (s,X^{x,m}_{s_2,s})\bigr] \, dW_s \bigr|^r\biggr)  \notag \\ 
 & (\text{reuse the SDE in the second term}) \\
&\leq C_r\biggl( |s_2-s_1|^{\frac{r}{2}} + \mathbf{E}\bigl | X^{X^{x,m}_{s_1,s_2}}_{s_2,t} - X^{x, m}_{s_2, t} \bigr|^r\biggr)  \\
&= 	C_r\biggl(|s_2-s_1|^{\frac{r}{2}} + \mathbf{E}\bigl[ \mathbf{E}  |X^{y, m}_{s_2,t} - X^{x, m}_{s_2, t} |_{y = X^{x,m}_{s_1,s_2}}^{r} \bigr]\biggr) \notag  \\
& (\text{use \eqref{est.6} in the second term}) \\
& \leq  	C_r\biggl(|s_2-s_1|^{\frac{r}{2}} + C\mathbf{E} \bigl(|X^{x,m}_{s_1,s_2} - x|^{r-d} + |X^{x,m}_{s_1,s_2} - x|^{r}\bigr)\biggr) \quad (\text{use again \eqref{est.3}})\notag\\
& \leq C_{1,r}(s_2-s_1)^{\frac{r-d}{2}} \notag
\end{align}
(there is additional $(s_2-s_1)^{\frac{r}{2}}$ term, but, since we work on a finite time interval, it is dominated by $(s_2-s_1)^{\frac{r-d}{2}}$). 

To summarize: estimates \eqref{est.3}, \eqref{est.6}, \eqref{est.7}  yield, for all $r>d$,
\begin{equation}
\label{ineq1}
\mathbf{E}|X_{s_1,t_1}^{x,m}-X_{s_2,t_2}^{y,m}|^r \leq C_{1,r}(|t_2-t_1|^{\frac{r}{2}} + |x-y|^{r-d} + |x-y|^r + |s_2-s_1|^{\frac{r-d}{2}} )
\end{equation}
for all $(s_i,t_i) \in \Delta_2(T)$, $i=1,2$.

\medskip

\textbf{Step 4.}~Proposition \ref{prop2}(\textit{i})-(\textit{iii}) yields

\begin{claim} 
\label{claim_com}
For fixed $y \in \mathbb R^d$ and $0 \leq r<t \leq T$, the random fields $X^{x,m}_{r,t}$, $x \in B_1(y)$
 verify conditions of Lemma \ref{lem_compact} where we take  $F_m(x):=X^{x,m}_{r,t}$, $U:=B_1(y)$.
\end{claim}

For the reader's convenience, we included the proof of Claim \ref{claim_com} in Appendix \ref{claim_com_proof}.

Claim \ref{claim_com} and Lemma \ref{lem_compact} yield that, for every fixed $(s,t) \in \mathbb Q^2 \cap \Delta_2(T)$ and every integer $N \geq 1$, the sequence
$$
B_N\times\Omega\ni(x,\omega)
\mapsto X_{s,t}^{x,m}(\omega), \quad m=1,2,\dots
$$
has a convergent subsequence in $L^2(B_N(0)\times\Omega)$ (i.e.\,we cover ball $B_N(0)$ by a finite number of balls of radius $1$ and pass to a convergent subsequence finitely many times).
Further, since $\mathbb Q^2 \cap \Delta_2(T)$ is countable, a diagonal argument yields a subsequence of $\{X_{s,t}^{\cdot,m}\}_{m \geq 1}$ (without loss of generality, still denoted by $\{X_{s,t}^{\cdot,m}\}$), and random fields
$
Y_{s,t}\in L^2_{\mathrm{loc}}(\mathbb R^d\times\Omega)$, $(s,t)\in \mathbb Q^2 \cap \Delta_2(T)$
such that, for every $N$ and all $(s,t)\in \mathbb Q^2 \cap \Delta_2(T)$
$$
X_{s,t}^{\cdot,m}\rightarrow Y_{s,t}
\quad\text{in }L^2(B_N\times\Omega) \text{ as } m \rightarrow \infty.
$$
Passing to a further subsequence, we may additionally assume that
$$
\sum_{m=1}^{\infty}\|
X_{s,t}^{\cdot,m}-Y_{s,t}\|_{L^2(B_N\times\Omega)}^2<\infty, \quad (s,t) \in \mathbb Q^2 \cap \Delta_2(T),\quad N \geq 1.
$$
It follows that there exists a Borel measurable set $G \subset \mathbb R^d$ such that its complement $\mathbb R^d - G$ is of measure zero and such that  
\begin{equation}
\sum_{m=1}^{\infty}
\mathbf E \left|X_{s,t}^{x,m}-Y_{s,t}(x)\right|^2 <\infty, \quad x \in G.
\label{70b}
\end{equation}
Fix a countable subset $D \subset G$ that is dense in $\mathbb R^d$.
Then there exists $\Omega_0 \subset \Omega$, $\mathbf P(\Omega_0)=1$, such that
\begin{equation}
X_{s,t}^{x,m}\rightarrow Y_{s,t}(x) \quad \text{ on $\Omega_0$ for all } x \in D,\;(s,t)\in \mathbb Q^2 \cap \Delta_2(T). 
\label{70c}
\end{equation}
Now, define the sought limiting process, so far on a countable dense set of times and starting points:
$$
X_{s,t}^x:=Y_{s,t}(x) \quad \text{ for } (s,t)\in \mathbb Q^2 \cap \Delta_2(T),\quad s<t,\quad x\in D, \qquad \text{ and }
X_{s,s}^x:=x.
$$

Clearly, \eqref{70b} yields convergence $X_{s,t}^{x,m} \rightarrow X_{s,t}^{x}$ in $L^2(\Omega)$ for all $(s,t)\in \mathbb Q^2 \cap \Delta_2(T)$. Interpolating with \eqref{est.4}, one has 
$$
X_{s,t}^{x,m} \rightarrow X_{s,t}^x\quad \text{ in } L^r(\Omega) \quad \forall\,r \geq 2.
$$ 
Next, fix an even integer
$r>3d+4$ (this will be needed at the next step). Applying the previous convergence in \eqref{ineq1}, we obtain
\begin{equation}
\label{limit_est}
\mathbf E\left|X_{s_1,t_1}^{x}-X_{s_2,t_2}^{y}\right|^r
\leq C_{1,r}\big(|t_2-t_1|^{\frac{r}{2}}+|x-y|^{r-d}+|x-y|^r +|s_1-s_2|^{\frac{r-d}{2}}
\big)
\end{equation}
for all
$
(s_i,t_i)\in \mathbb Q^2 \cap \Delta_2(T)$, $x,y\in D.
$

\medskip

\textbf{Step 5.}~Since $\frac{r-d}{2}>d+2$ ($=$ dimension of $\Delta_2(T) \times \mathbb R^d$), the Kolmogorov-Chentsov theorem (or, rather, its version due to Kunita \cite[Theorem 1.4.1]{Ku}, see Appendix \ref{kunita_app}) allows us to extend $X_{s,t}^x$ to a continuous random field on $\Delta_2(T) \times \mathbb R^d$, which we still denote by $X_{s,t}^x$.
Let us show that, after passing to a subsequence of $\{X^m\}$ (as usual, re-denoted as $\{X^m\}$)
\begin{equation}
\label{X_sm_conv}
X^{m} \rightarrow X\quad \mathbf P\text{-a.s. locally uniformly on }\Delta_2(T)\times\mathbb R^d.
\end{equation}

The rest of Step 5 is the proof of \eqref{X_sm_conv}. 

We will need to introduce a few notations that will save us some space later.

Set $K_N:=\Delta_2(T)\times\bar{B}_N(0)$, $N \geq 1$.

Define the distance between the time-space parameters $\xi=(s,t,x)$, $\eta=(s',t',y) \in \Delta_2(T) \times \mathbb R^d$ by 
$$
d(\xi,\eta):=|s-s'|+|t-t'|+|x-y|.
$$

Using this distance, we define the modulus of continuity of the approximating random field $X^m$ on $K_N$ by
$$
\omega_{m,N}(\delta)
:=\sup_{\xi,\eta\in K_N, d(\xi,\eta)\leq\delta}
|X^m(\xi)-X^m(\eta)|,
$$
where, in the rest of Step 5, we write
$$
X^m(s,t,x):=X_{s,t}^{x,m}.
$$
We divide the rest of Step 5 into a few parts. In the first part we deal with random fields $X^m$ and in the second part with $X$.

1.~Inequality \eqref{ineq1} implies that, whenever $\xi,\eta\in K_N$ and $d(\xi,\eta)\leq 1$,
$$
\sup_m \mathbf E|X^m(\xi)-X^m(\eta)|^r
\leq
C_{1,r} d(\xi,\eta)^{\frac{r-d}{2}}.
$$
This inequality and $\sup_{\xi \in K_N}\sup_m\mathbf{E}|X^m(\xi)|^r<\infty$, which follows from \eqref{est.4}, allow us to apply the Kolmogorov–Chentsov theorem of \cite[Theorem 1.4.7]{Ku} (Appendix \ref{kunita_app}) that yields the tightness of the laws of $X^m|_{K_N}$ in $C(K_N)$.
In turn, this yields, for every $\varepsilon>0$,
\begin{equation}
\label{delta_conv}
\lim_{\delta \downarrow 0}\sup_m \mathbf{P}[\omega_{m,N}(\delta)>\varepsilon]=0
\end{equation}
for every $\varepsilon>0$ and every $N$.
Indeed, the tightness of $X^m|_{K_N}$ implies that, for any small $\alpha>0$, there is a compact set $M \subset C(K_N)$ such that the probability that a realization of $X^m|_{K_N}$ is in $M$
$$
\mathbf{P}[X^m|_{K_N} \in M] \geq 1-\alpha.
$$
Since $M$ must be equicontinuous, there exists $\delta>0$ such that $\sup_{f \in M}\sup_{\xi,\eta\in K_N, d(\xi,\eta)\leq\delta}|f(\xi)-f(\eta)| \leq \varepsilon$, and so
$$
\mathbf{P}\left[\sup_{\xi,\eta\in K_N, d(\xi,\eta)\leq\delta}|X^m(\xi)-X^m(\eta)| \leq \varepsilon\right] \geq \mathbf{P}[X^m|_{K_N} \in M] \geq 1-\alpha.
$$
It follows that $\mathbf{P}[\sup_{\xi,\eta\in K_N, d(\xi,\eta)\leq\delta}|X^m(\xi)-X^m(\eta)| > \varepsilon] \leq \alpha$.
Since $\alpha$ could be chosen arbitrarily small, we obtain \eqref{delta_conv}.

2.~Define the modulus of continuity of the limiting field $X$ by
$$
\omega_N(\delta):=\sup_{\xi,\eta\in K_N, d(\xi,\eta)\leq \delta}|X(\xi)-X(\eta)|.
$$
Since $X$ is already a.s.\,continuous on compact $K_N$, we have
\begin{equation}
\label{omega_conv}
\omega_N(\delta) \rightarrow 0 \quad \mathbf P\text{-a.s. as }\delta \downarrow 0.
\end{equation}

3.~Fix some $\delta>0$ and put for brevity $\mathbb D:=(\mathbb Q^2 \cap \Delta_2(T)) \times D$. Note that since $\mathbb D \cap K_N$ is dense in $K_N$, there exists a finite $\delta$-net $\Gamma_{N,\delta}\subset \mathbb D\cap K_N$.
So, for every $\xi \in K_N$ there exists $q(\xi) \in\Gamma_{N,\delta}$ such that
$$
d(\xi,q(\xi)) \leq \delta.
$$
Now we estimate
$$
|X^m(\xi)-X(\xi)| \leq |X^m(\xi)-X^m(q(\xi))| + |X^m(q(\xi))-X(q(\xi))| + |X(q(\xi))-X(\xi)|.
$$
Taking the maximum over $\xi\in K_N$, we further obtain
$$
\|X^m-X\|_{C(K_N)} \leq \omega_{m,N}(\delta) + \max_{q\in\Gamma_{N, \delta}} |X^m(q)-X(q)| + \omega_N(\delta).
$$
Hence
$$
\mathbf P\bigl[\|X^m-X\|_{C(K_N)}>3\varepsilon\bigr] \leq \mathbf P\bigl[\omega_{m,N}(\delta)>\varepsilon\bigr] +
\mathbf P[\max_{q\in\Gamma_{N,\delta}}|X^m(q)-X(q)|>\varepsilon]+\mathbf P\bigl[\omega_N(\delta)>\varepsilon\bigr].
$$
For a fixed $\delta>0$, the second term on the right tends to zero as $m \rightarrow \infty$ since the $\delta$-net $\Gamma_{N,\delta}$ is finite and $X^m(q)\to X(q)$ almost surely for every $q\in \mathbb D$. Hence, by \eqref{delta_conv},
$$
\limsup_{m\to\infty}\mathbf P\bigl[
\|X^m-X\|_{C(K_N)}>3\varepsilon\bigr] \leq \sup_m \mathbf P\bigl[\omega_{m,N}(\delta)>\varepsilon\bigr] + \mathbf P\bigl[\omega_N(\delta)>\varepsilon\bigr].
$$
Taking $\delta \downarrow 0$ and using \eqref{delta_conv} and \eqref{omega_conv}, we arrive at
\begin{equation}
\label{Xm_X}
X^m \rightarrow X \quad\text{in probability in }C(K_N), \quad \text{ for every } N \geq 1.
\end{equation}

4.~We finally extract a subsequence that converges a.s.\,on every $K_N$. By \eqref{Xm_X}, passing to a subsequence,
$$
\mathbf P\left[
\|X^{m_j}-X\|_{C(K_j)}>2^{-j}\right]
\leq 2^{-j}.
$$
By the Borel-Cantelli lemma, on an event of probability one,
$
\|X^{m_j}-X\|_{C(K_j)} \leq 2^{-j}
$
starting with some $j$. Hence, for every fixed $N$,
$$
\sup_{(s,t) \in \Delta_2(T), |x|\leq N} \left|X_{s,t}^{x,m_j}-X_{s,t}^x\right| \rightarrow 0 \quad\mathbf P\text{-a.s.}
$$
Thus, after passing to a subsequence, we arrive at \eqref{X_sm_conv}.

\medskip

\textbf{Step 6.} Next, we show that $X_{s,t}^x$, $x \in \mathbb R^d$, is a family of strong solutions (Definition \ref{strong_def}) to SDE
\begin{equation}
\label{sde2sigma_limit}
X_{s,t}^{x}=x+  \int_s^t \sigma(r,X^{x}_{s,r}) \, dW_r, \quad s \leq t \leq T
\end{equation}
In view of \eqref{X_sm_conv}, and since $X^x_{s,t}$ are by construction  adapted to the increment filtration of the Brownian motion, to end the proof it suffices for us to show that
\begin{equation}
\label{limit_sigma}
\int_s^t \sigma_m(\tau,X_{s,\tau}^{x,m}) \, dW_\tau \rightarrow \int_s^t \sigma(\tau,X_{s,\tau}^{x}) \, dW_\tau \quad \text{ in } L^2(\Omega,[C[s,T]]^d).
\end{equation}
Indeed, using the BDG inequality, we have
\begin{align}
\mathbf{E}\sup_{s \leq t \leq T}\biggl|\int_s^t &\sigma_m(\tau,X_{s,\tau}^{x,m})dW_\tau -\int_s^t \sigma(\tau,X_{s,\tau}^{x})dW_\tau \biggr|^2 \notag \\
&  \leq 2\mathbf{E}\sup_{s \leq t \leq T}\biggl|\int_s^t (\sigma_m-\sigma_k)(\tau,X_{s,\tau}^{x,m})dW_\tau\biggr|^2 \\
&+ 2\mathbf{E}\sup_{s \leq t \leq T}\biggl|\int_s^t \sigma_k(\tau,X_{s,\tau}^{x,m})dW_\tau - \int_s^t \sigma_k(\tau,X_{s,\tau}^{x})dW_\tau\biggr|^2 \notag \\
& + 2\mathbf{E}\sup_{s \leq t \leq T}\biggl|\int_s^t (\sigma_k-\sigma)(\tau,X_{s,\tau}^{x})dW_\tau\biggr|^2 \notag \\
& \leq  2C\mathbf{E}\int_s^T |\sigma_m-\sigma_k |^2(\tau,X_{s,\tau}^{x,m})d\tau  +  2C\mathbf{E}\int_s^T \biggl| \sigma_k(\tau,X_{s,\tau}^{x,m}) -\sigma_k(\tau,X_{s,\tau}^{x})\biggr |^2  d\tau \label{exp_bd}\\
&+ 2C\mathbf{E}\int_s^T |\sigma_k-\sigma|^2 (\tau,X_{s,\tau}^{x})d\tau. \label{exp_bd2}
\end{align}
We now apply the following consequence of Proposition \ref{prop_grad}. In what follows, without loss of generality, $x \in B_{10}(0)$.

\begin{lemma}
\label{lem_conv}
Define occupation measures: for a Borel measurable set $A \subset [s,t] \times \mathbb R^d$, set
$$
\mu_m^{s,x}(A):=\mathbf{E}\int_s^t \mathbf{1}_A(r,X_{s,r}^{x,m})dr, \quad m=1,2,\dots
$$
and
$$
\mu^{s,x}(A):=\mathbf{E}\int_s^t \mathbf{1}_A(r,X_{s,r}^{x})dr.
$$
The following are true:
\begin{itemize}
\item[{\rm (\textit{i})}] For every function $0 \leq F \in C_c([s,t] \times \mathbb R^d) \cap \mathbf{F}_\mu$,
\begin{align*}
\int F d\mu_m^{s,x} & = \mathbf{E}\int_s^t F(r,X_{s,r}^{x,m})dr \\
& \rightarrow \int_s^t \int_{\mathbb R^d} F d\mu^{s,x} = \mathbf{E}\int_s^t F(r,X_{s,r}^{x})dr \text{ as } m \rightarrow \infty
\end{align*}
and
\begin{equation}
\label{est_7}
\mathbf{E}\int_s^t F(r,X_{s,r}^{m,x})dr,\;\; \mathbf{E}\int_s^t F(r,X_{s,r}^{x})dr \leq C \|F\|_{L^2([s,t], L^2_{\rho}(\mathbb R^d))}^{\frac{2}{q}},
\end{equation}
where constant $C$ depends on $\mu$, but does not depend on the smoothness or the support of $F$.

\item[{\rm (\textit{ii})}] The limiting occupation measure $\mu^{s,x}$ is absolutely continuous with respect to the weighted Lebesgue measure $dr \times \rho(x)dx$ on $[s,t] \times \mathbb R^d$.

\smallskip

\item[{\rm (\textit{iii})}] For every function $0 \leq F \in L^2([s,t] \times L^2_{\rho}(\mathbb R^d)) \cap \mathbf{F}_\mu$,
$$
\mathbf{E}\int_s^t F(r,X_{s,r}^{x})dr \leq C \|F\|_{L^2([s,t], L^2_{\rho}(\mathbb R^d))}^{\frac{2}{q}}.
$$ 
\end{itemize}

\end{lemma}
\begin{proof}
(\textit{i}) The first estimate in \eqref{est_7} is an immediate consequence of Proposition \ref{prop_grad}, i.e.\,\eqref{v_est_unif}, that we apply (on the interval $[s,t]$ rather than $[0,T]$) with initial function $f=0$. The convergence in (\textit{i}) follows from \eqref{X_sm_conv} via the Dominated convergence theorem. This convergence transforms the first estimate into the second estimate.

(\textit{ii}) Our goal is to show that $\mu^{s,x} \ll dr \times \rho(x)dx$. Let $K \subset [s,t] \times \mathbb R^d$ be a compact set of $dr \times \rho(x)dx$-measure zero. There exists a decreasing sequence of open sets $U_k \supset K$ such that $(dr \times \rho(x)dx)(U_k) \downarrow 0$. By Urysohn's theorem, there exists a sequence of continuous functions $\zeta_k$ having $\sprt \zeta_k \subset U_k$ and identically equal to $1$ on $K$. Now, we estimate:
\begin{align*}
\mu^{s,x}(K) & \leq \int \zeta_k d\mu^{s,x} \\
& \leq C\|\zeta_k\|^{\frac{2}{q}}_{L^2([s,t],L^2_{\rho}(\mathbb R^d))} \downarrow 0 \quad \text{ as } k \rightarrow \infty.
\end{align*}
Hence, $\mu^{s,x}(K)=0$. Since $\mu^{s,x}$ is a Radon measure, by the regularity $\mu^{s,x}(A)=0$ for all $dr \times \rho(x)dx$ measure zero $A$.

(\textit{iii}) We use (\textit{ii}) to pass from the second estimate in (\textit{i}) to (\textit{iii}). Let $F_n:=\eta^{\mathbb R^{1+d}}_{\varepsilon_n} \ast \mathbf{1}_n F$, where $\eta_{\epsilon_n}$ is a mollifier on $\mathbb R^{1+d}$ with $\epsilon_n \downarrow 0$ and $\mathbf{1}_n:=\{(r,x) \in [s,t] \times \mathbb R^d \mid |x| \leq n, F(r,x) \leq n\}$. Then $F_n \in \mathbf{F}_\mu$ with the same $c_\mu$ as $F$, see Appendix \ref{Conv_b_n} where we show that, selecting $\varepsilon_n \downarrow 0$ sufficiently fast, we have
$$
\|F_n-F\|_{L^2([s,t],L^2_{\rho}(\mathbb R^d))} \rightarrow 0,
$$
hence there exists a subsequence (without loss of generality, $\{F_n\}$ itself) such that 
$$
F_n \rightarrow F \quad \text{$dr \times \rho(x)dx$-a.e.\,on $[s,t] \times \mathbb R^d$.}
$$
By (\textit{ii}), 
$$
F_n \rightarrow F  \quad \text{$\mu^{s,x}$-a.e.\,on $[s,t] \times \mathbb R^d$}.
$$
We can apply Fatou's lemma in $n \rightarrow \infty$ in the second estimate in (\textit{ii}) applied to $F_n$ to obtain (\textit{iii}).
\end{proof}

Let us now use Lemma \ref{lem_conv} to send \eqref{exp_bd}, \eqref{exp_bd2} to zero, as $m \rightarrow \infty$, by selecting $k$ appropriately. The first term in 
the first term in \eqref{exp_bd} is bounded, upon selecting $F:=|\sigma_m-\sigma_k|^2$ in Lemma \ref{lem_conv}, as follows:
$$
\mathbf{E}\int_s^T |\sigma_m-\sigma_k |^2(\tau,X_{s,\tau}^{x,m})d\tau \leq C\|\sigma_m-\sigma_k\|^2_{L^2([0,T],L^2_\rho(\mathbb R^d))}.
$$ 
The right-hand side can be made smaller than any given $\varepsilon$, assuming that $m \geq m_\varepsilon$, by fixing $k$ sufficiently large. The same holds, upon selecting $F:=|\sigma_k-\sigma |$ in Lemma \ref{lem_conv}, for \eqref{exp_bd2} -- select $k$ even larger, if needed, to make \eqref{exp_bd2} smaller than given $\varepsilon$. Finally, for thus fixed $k$, the second term in \eqref{exp_bd} goes to zero as $m \rightarrow \infty$ by the Dominated convergence theorem. \hfill \qed

\begin{remark}
At the last step we could appeal instead to Krylov's  occupation-time estimate. See Section \ref{proof_struct_sect} for the discussion.
\end{remark}

\bigskip

\section{Proof of Proposition \ref{prop_grad}}
\label{prop_grad_proof}

We carry out the proof, which is a rather direct combination of the proofs of the analogous results in \cite{Ki_Osaka} and \cite{KiS_Osaka}, taking into account the remark on $(\nabla_r a_{il})_{i=1}^{d}$ and $\nabla a$ in the beginning of Section \ref{contr_grad_sect}. Without loss of generality, we assume throughout the proof that the ellipticity constant $\kappa=1$.

1.~We first consider the case of the right-hand side $F=0$ and the constant weight $\rho = 1$. We will explain in the end what needs to be modified in the proof to include non-zero $F$ with $\rho$ vanishing polynomially at infinity.

We reverse time and carry out the proof for solutions of the corresponding initial-value problem that we write in the divergence form:
\begin{align}
\label{CP0}
\left\{
\begin{array}{l}
\partial_t u -  \frac{1}{2} \nabla \cdot a \cdot \nabla +  \frac{1}{2}  \nabla a \cdot \nabla u=0 \qquad \text{ in } ]0,T] \times \mathbb R^d, \\[2mm]
 u|_{t=s}=f,
\end{array}
\right.
\end{align}
where $0 \leq s<T$. Here
$a \mapsto \nabla a$ is the row-divergence operator defined by
$$(\nabla a)_{k}:=\sum_{i=1}^{d} ( \nabla_{i} a_{ik}), \quad k=1,\dots,d.$$

Below we will first select $s$ close to $T$, but then will remove this constraint and will be able to take $s=0$.

Set
$$w_{r} := \nabla_{r}u, \quad w:=(w_r)_{r=1}^d \equiv \nabla u,$$ 
\[
I_q := \sum_{r=1}^{d} \int_s^\tau \langle (\nabla_r w)^{2}|w|^{q-2} \rangle dt , \quad J_q := \int_s^\tau \langle \big|\nabla |w|\big|^{2}|w|^{q-2} \rangle dt,
\]
\[
I_q^a :=\sum_{r=1}^{d} \int_s^\tau \langle (\nabla_r w \cdot a \cdot \nabla_r w)|w|^{q-2} \rangle dt , \quad J_q^a := \int_s^\tau \langle (\nabla |w| \cdot a \cdot \nabla |w|) |w|^{q-2} \rangle dt.
\]
Define the commutator:
$$[F, G] :=FG - GF.$$ 
Set $A := -\nabla \cdot a \cdot \nabla $.

We will be using the test function due to Kovalenko-Sem\"{e}nov \cite{KS}:
\[
\phi_r := -\partial_{r} (w_r |w|^{q-2}).
\]

\begin{remark}
\label{test_func_rem}
The use of this test function allows for a fine control of the admissible form-bounds. There are alternative rougher arguments that involve differentiating the equation and integrating by parts in a more straightforward way, see \cite{Ki_survey}. One challenging part of the proof given below -- not handled in the elliptic setting of \cite{KS} -- is dealing with the time derivative of the solution; we deal with it arguing as in \cite{Ki_Osaka}.
\end{remark}

That is, we multiply the parabolic equation in \eqref{CP0} by $\phi_r$, integrate in space and time (from the initial time $s$ up to time $\tau$), and then sum over $r$ to get :
\begin{align*}
\frac{1}{q}\langle |w(\tau)|^{q} \rangle -\frac{1}{q}\langle |\nabla f |^{q} \rangle & +\frac{1}{2} \int_s^\tau \langle Aw, w|w|^{q-2} \rangle dt  =\\
 & - \frac{1}{2}  \sum_{r=1}^{d}  \int_s^\tau \langle [\nabla_r, A]u, w_r|w|^{q-2} \rangle dt    - \frac{1}{2}  \sum_{r=1}^{d}\int_s^\tau \langle (\nabla a)\cdot w, \phi_r  \rangle dt,
\end{align*}
where we have used $\sum_{r=1}^d \langle A \nabla_r u, w_r|w|^{q-2} \rangle dt=\int_s^\tau \langle Aw, w|w|^{q-2} \rangle dt$, with $A$ applied to vector field $w$ componentwise.
We evaluate:
\[
\int_s^\tau \langle Aw, w|w|^{q-2} \rangle dt  = I_q^a  + (q-2)J_q^a
\]
Since $a \geq I$ (i.e.\,the ellipticity condition with $\kappa=1$), and so $I_q^a \geq I_q$ and $J_q^a \geq  J_q$, we obtain inequality
\begin{align}
\label{iq2}
\frac{1}{q}\langle |w(\tau)|^{q} \rangle - \frac{1}{q}\langle |\nabla f |^{q} \rangle & +\frac{1}{2}  I_q  + \frac{(q-2)}{2} J_q \notag \\
& \leq - \frac{1}{2}  \sum_{r=1}^{d}  \int_s^\tau \langle [\nabla_r, A]u, w_r|w|^{q-2} \rangle dt - \frac{1}{2} \sum_{r=1}^{d}\int_s^\tau \langle (\nabla a)\cdot w, \phi_r  \rangle dt.
\end{align}
Thus, our goal is to estimate the two terms in the right-hand side of \eqref{iq2} in terms of $I_q$ and $J_q$. This will be done in the next two claims.

\begin{claim}
\label{Claim1}
\begin{align*}
\left|\sum_{r=1}^{d}  \int_s^\tau \langle [\nabla_r, A]u, w_r|w|^{q-2} \rangle dt  \right| & \leq M_{1}I_{q} +  N_{1}J_{q}+ C_1\sum_{r, l=1}^{d} \int_s^\tau g_{\gamma_{rl}}(t)\langle |w|^q \rangle dt, 
\end{align*}
where 
\[
M_{1} :=  \frac{1}{4\alpha_1}, \qquad N_{1}:= \alpha_1 \gamma \frac{q^{2}}{4}  + 
 (q-2)\bigg(\alpha_2 \gamma \frac{q^{2}}{4} + \frac{1}{4\alpha_2} \bigg), \quad C_1:=\big( \alpha_1+(q-2)\alpha_2\big),
\]
where $\alpha_1,\alpha_2>0$ are arbitrary, and 
$$
\gamma := \sum_{r, l= 1}^{d}\gamma_{rl}.
$$
\end{claim}

 Note that  $M_1$ and $N_1$ can be made as small as needed  by assuming that the form-bounds $\delta_a$ and $\gamma_{rl}$'s are sufficiently small (provided $\alpha_1,\alpha_2>0$ are first chosen to be sufficiently large).

\begin{proof}[Proof of Claim \ref{Claim1}]
We have
$
[\nabla_r,A]=\nabla_r \sum_{i,l=1}^d \nabla_i (a_{il} \nabla_l) - \sum_{i,l=1}^d \nabla_i (a_{il} \nabla_l\nabla_r),
$
so, using integration by parts, we evaluate:
\begin{align*}
 \sum_{r=1}^d\int_s^\tau \langle [\nabla_r, A]u, w_r|w|^{q-2} \rangle dt &= \sum_{r,i,l=1}^{d} \int_s^\tau \langle (\nabla_r a_{i l})w_l, (\nabla_i
w_r) |w|^{q-2} \rangle dt \\
&+ (q-2)\sum_{r,i,l=1}^{d} \int_s^\tau \langle (\nabla_r a_{i l})w_l,w_r|w|^{q-2}\nabla_i |w| \rangle dt  \\
&\text{(use Cauchy-Schwarz)}\\
&\leq \alpha_1 \sum_{r, l=1}^{d}\int_s^\tau \langle | \nabla_r a_{\cdot l}|^2 |w|^q \rangle dt + \frac{1}{4\alpha_1} I_q \\
& + (q-2)  \biggl[\alpha_2 \sum_{r, l=1}^{d}\int_s^\tau \langle | \nabla_r a_{\cdot l}|^2 |w|^q \rangle dt + \frac{1}{4\alpha_2} J_q \biggr]. 
\end{align*}
Since $\nabla_r a_{\cdot l} \in \mathbf{F}_{\gamma_{rl}}$, 
\begin{align*}
\sum_{r, l=1}^{d} \int_s^\tau \langle | \nabla_r a_{\cdot l}|^2 |w|^q \rangle dt & \leq \sum_{r, l=1}^{d} \gamma_{rl} \frac{q^{2}}{4} \int_s^\tau \langle (\nabla |w|)^2 |w|^{q-2} \rangle \, dt + \sum_{r, l=1}^{d} c_{\gamma_{rl}}\int_s^\tau \langle |w|^q \rangle dt  \\
& = \gamma \frac{q^{2}}{4}J_q +  \sum_{r, l=1}^{d} c_{\gamma_{rl}}\int_s^\tau \langle |w|^q \rangle dt.
\end{align*}
Thus,
\begin{align*}
\sum_{r, l=1}^{d}  \int_s^\tau \langle [\nabla_r, A]u, w_r|w|^{q-2} \rangle dt &  \leq   \frac{1}{4\alpha_1} I_q  + \biggl[ \alpha_1 \gamma \frac{q^{2}}{4}  + 
 (q-2)\big(\alpha_2 \gamma \frac{q^{2}}{4} + \frac{1}{4\alpha_2} \big) \biggr] J_q \\
& + \big( \alpha_1+(q-2)\alpha_2\big) \sum_{r, l=1}^{d} c_{\gamma_{rl}}\int_s^\tau \langle |w|^q \rangle dt,
\end{align*}
which is the claimed inequality.
\end{proof}

\begin{claim}
\label{Claim2}
\[
\biggl| \sum_{r=1}^{d}\int_s^\tau \langle (\nabla a)\cdot w, \phi_r  \rangle dt \bigg|  \leq M_2 I_q + N_2 J_q+ C_2 c_{\delta_a} \int_s^\tau \langle |w|^{q} \rangle dt,
\]
where
\[
M_2=\frac{\nu_1}{4}\frac{4\varkappa}{4\varkappa-q+2}\frac{q}{2}  , \qquad N_2=\frac{\nu_1}{4}\frac{4\varkappa}{4\varkappa-q+2}(q-2)\bigg(\varkappa+\frac{1}{2}\bigg)+\frac{1}{\nu_1}\frac{q^2}{4}\delta_a + (q-2) \biggl[ \nu_2 
+ 
\frac{\delta_a}{4\nu_2} \frac{q^2}{4}\biggr], 
\]
\[
C_2=\frac{1}{\nu_1} + \frac{q-2}{4\nu_2},
\]
where $\varkappa>\frac{q-2}{4}$ and $\nu_1$, $\nu_2>0$ are arbitrary. 
\end{claim}

Again, we can make $M_2$, $N_2$ arbitrarily small by assuming that the form-bound $\delta_a$ of the vector field $\nabla a$ is sufficiently small (provided that $\nu_1$, $\nu_2$ are first chosen to be sufficiently small).

\begin{proof}[Proof of Claim \ref{Claim2}] Write
$
\sum_{r=1}^{d}\int_s^\tau \langle (\nabla a)\cdot w, \phi_r  \rangle dt=\int_s^\tau \langle (\nabla a)\cdot w, \phi  \rangle dt,
$
where
$$
\phi:=\sum_{r=1}^d \phi_r \equiv -\nabla \cdot (w|w|^{q-2}).
$$
Now, instead of integrating by parts right away as we did in Claim \ref{Claim1}, we first evaluate the test function $\phi$:
\begin{equation}
\label{phi_rep}
\phi= (-\Delta u)|w|^{q-2} -w \cdot \nabla |w|^{q-2},
\end{equation}
so
\begin{align}
\int_s^\tau \big\langle (\nabla a)\cdot w, \phi  \big\rangle dt & = - \int_s^\tau \big\langle (\nabla a)\cdot w, \Delta u\,|w|^{q-2} \big\rangle dt -
 \int_s^\tau \big\langle (\nabla a)\cdot w, w \cdot \nabla |w|^{q-2}\big\rangle dt \notag \\[2mm]
 & =: W_1+W_2. \label{a_est}
\end{align}
The rest goes in two steps:

1.~We estimate $W_1$ using Cauchy-Schwarz and applying $\nabla a \in \mathbf{F}_{\delta_a}$:
\begin{align*}
|W_1| & \leq \frac{\nu_1}{4} \int_s^\tau \langle |w|^{q-2}|\Delta u|^2 \rangle dt + 
\frac{1}{\nu_1} \int_s^\tau \left\langle |\nabla a|^2 |w|^q \right\rangle dt  \\
& \leq \frac{\nu_1}{4} \int_s^\tau \langle |w|^{q-2}|\Delta u|^2 \rangle dt + \frac{1}{\nu_1} \left [\delta_a \frac{q^2}{4}J_q + c_{\delta_a}\int_s^\tau \langle |w|^{q} \rangle \right] 
\end{align*}
In turn, representing $|\Delta u|^2=(\nabla \cdot w)^2$
and integrating by parts twice, we obtain:
\begin{multline*}
\int_s^\tau \langle |w|^{q-2}|\Delta u|^2 \rangle dt  = -\int_s^\tau \langle  w \cdot \nabla |w|^{q-2}, \Delta u \rangle dt
+\sum_{r=1}^d \int_s^\tau \left\langle w \cdot \nabla w_r, \nabla_r |w|^{q-2}\right\rangle dt 
+ I_q\\ =:-F+H+I_q,
\end{multline*}
where we estimate
$$
|F| \leq (q-2)\left( \frac{1}{4\varkappa}  \int_s^\tau \langle |w|^{q-2}|\Delta u|^2\rangle dt + \varkappa J_q\right),
\quad |H| \leq (q-2)\left( \frac{1}{2} I_q+ \frac{1}{2} J_q\right).
$$
Hence
\begin{equation}
\label{Delta_est}
\biggl(1-\frac{q-2}{4\varkappa} \biggr)\int_s^\tau \langle |w|^{q-2}|\Delta u|^2 \rangle dt \leq I_q+(q-2)\left(\varkappa J_q + \frac{1}{2} I_q + \frac{1}{2} J_q \right), \quad \varkappa>\frac{q-2}{4},
\end{equation}
so
$$
|W_1| \leq \frac{\nu_1}{4}\frac{4\varkappa}{4\varkappa-q+2}\biggl(I_q+(q-2)\left(\varkappa J_q + \frac{1}{2} I_q + \frac{1}{2} J_q \right) \biggr) + \frac{1}{\nu_1} \left [\delta_a \frac{q^2}{4} J_q + c_{\delta_a}\int_s^\tau \langle |w|^{q} \rangle dt \right].
$$

2.~We estimate $W_2$: 
\begin{align*}
|W_2| & \leq  (q-2) \int_s^\tau \langle |w|^{\frac{q-2}{2}} |\nabla |w|| |\nabla a||w|^{\frac{q}{2}} \rangle dt \\
& \leq
(q-2) \left[ \nu_2 \int_s^\tau \langle |w|^{q-2} |\nabla |w||^2\rangle dt + 
\frac{1}{4\nu_2} \int_s^\tau \left\langle |\nabla a|^2|w|^q  \right\rangle  dt \right] \\ 
& \leq (q-2) \left[ \nu_2 J_q
+ 
\frac{\delta_a}{4\nu_2} \frac{q^2}{4}J_q +
\frac{1}{4\nu_2} c_{\delta_a}\int_s^\tau \langle |w|^{q}\rangle dt \right].
\end{align*}
It remains to apply these estimates on $|W_1|$, $|W_2|$ in \eqref{a_est} to end the proof.
\end{proof}

We now apply Claims \ref{Claim1} and \ref{Claim2} in \eqref{iq2}:
\begin{align*}
\frac{2}{q}\langle |w(\tau)|^{q} \rangle + \big(1-\sum_{i=1}^2 M_i\big)I_q & + \big(q-2-\sum_{i=1}^2 N_i\big)J_q \leq \frac{2}{q}\langle |\nabla f |^{q} \rangle + K\int_s^\tau \langle |w|^q \rangle dt, \\
& K:=C_1\sum_{r, l=1}^{d}  c_{\gamma_{rl}} + C_2 c_{\delta_a},
\end{align*}
where, in view of the remarks after the claims, $1-\sum_{i=1}^2 M_i>0$ and $q-2-\sum_{i=1}^2 N_i>0$, provided that the form-bounds $\gamma_{rl}$ and $\delta_a$ are sufficiently small.

The previous inequality holds for all $s<\tau \leq T$. In fact, since $I_q(\tau)$, $J_q(\tau)$ and the last term are monotonically increasing in $\tau$, we have
\begin{align*}
\frac{2}{q}\sup_{\tau \in [s,T]}\langle |w(\tau)|^{q} \rangle + \big(1-\sum_{i=1}^2 M_i\big)I_q(T) & + \big(q-2-\sum_{i=1}^2 N_i\big) J_q(T) \leq \frac{2}{q}\langle |\nabla f |^{q} \rangle \\
& + K \int_s^T \langle |w|^q \rangle.
\end{align*}
Hence
\begin{align*}
\frac{2}{q}\sup_{\tau \in [s,T]}\langle |w(\tau)|^{q} \rangle + \big(1-\sum_{i=1}^2 M_i\big)I_q(T) & +  \big(q-2-\sum_{i=1}^2 N_i\big)J_q(T) \leq \frac{2}{q}\langle |\nabla f |^{q} \rangle \\
& + K(T-s)\sup_{\tau \in [s,T]}\langle |w(\tau)|^{q} \rangle.
\end{align*}
Assuming that the interval $[s,T]$ is sufficiently small, we obtain
\begin{align*}
C_0\sup_{\tau \in [s,T]}\langle |w(\tau)|^{q} \rangle + \big(1-\sum_{i=1}^2 M_i\big) I_q(T) + \big(q-2-\sum_{i=1}^2 N_i\big)J_q(T) &\leq \frac{2}{q}\langle |\nabla f |^{q} \rangle,
\end{align*}
where $$C_0:=\frac{2}{q}-K(T-s)>0.$$

Further, applying elementary inequality $I_q \geq  J_q$, we arrive at
\begin{align*}
C_0\sup_{\tau \in [s,T]}\langle |w(\tau)|^{q} \rangle + C J_q \leq   \frac{2}{q}\langle |\nabla f |^{q} \rangle.
\end{align*}
for $C>0$,
which yields the sought inequality with $K_1=\frac{C}{C_0}$ and $K_2=\frac{2}{qC_0}$ -- but so far only on a small interval $[s,T]$. We can, however, extend the previous inequality to the interval of arbitrary length using the semigroup property at the expense of increasing the constants $K_1$, $K_2$.
This ends the proof in the case $F=0$ and $\rho =1$.

\medskip

3.~Let us now handle non-zero right-hand side $F \neq 0$. Since $\rho \equiv 1$ for now, we assume that $F \in L^2([0,T],L^2(\mathbb{R}^d))$.  In this case, after multiplying the parabolic equation by the test function $\phi_r$, summing in $1 \leq r \leq d$ and integrating, we obtain the following term 
$$
\sum_{r=1}^d \int_s^\tau \langle |F|,\phi_r \rangle=:S.
$$
We need to estimate it from above in terms of $I_q$ and $J_q$. To this end, we repeat the argument from \cite[Proof of Theorem 2.2]{KM_JDE}, namely, using \eqref{phi_rep}, we write
\begin{align*}
S =\int_{s}^{\tau} \langle |w|^{q-2}\Delta v,|F| \rangle & + \int_{s}^{\tau} \langle w \cdot \nabla |w|^{q-2},|F| \rangle \\
&=: Z_{1}+Z_{2}.
\end{align*}
Applying Cauchy-Schwarz, we have
\begin{align*}
Z_1 & \leq  \frac{\alpha_{1}}{4}  \int_{s}^{\tau} \langle |w|^{q-2}|\Delta v |^{2}\rangle dt +\frac{1}{\alpha_{1}} \int_{s}^{\tau} \big\langle |w|^{q-2}|F|^2 \big\rangle dt \qquad (\alpha_1>0) \\
& \text{(we are using \eqref{Delta_est})}\\
&\leq \frac{\varkappa \alpha_{1}}{4\varkappa -q+2}\big [I_{q}+(q-2)(\varkappa J_{q}+\frac{1}{2}I_{q}+\frac{1}{2}J_{q})\big ]+\frac{1}{\alpha_{1}} L_q
\end{align*}
where $$L_q:=\int_{s}^{\tau} \big\langle |w|^{q-2}|F|^2 \big \rangle dt.$$
Next,
\begin{align*}
Z_2 &\leq (q-2)\int_{s}^{\tau}  \big \langle |w|^{\frac{q}{2}-1}\nabla |w|, |w|^{(q-2)/2}|F| \big\rangle \quad (\alpha_2>0)\\
&\leq (q-2)\bigg [ \alpha_{2}J_{q}+\frac{1}{4\alpha_{2}} L_q \bigg ].
\end{align*}
It remains to estimate $L_q$:
\begin{align*}
L_{q} &= \int_{s}^{\tau} \langle |F|^{2-\frac{4}{q}}|w|^{q-2}, |F|^{\frac{4}{q}}\rangle \\
& (\text{apply Young's inequality}) \\
&\leq \frac{q-2}{q}\varepsilon^{\frac{q}{q-2}}\int_{s}^{\tau} \langle|F|^{2} |w|^{q} \rangle  +\frac{2}{q}\varepsilon^{-\frac{q}{2}}\int_{s}^{\tau} \langle |F|^2 \rangle  \\
& \text{(we are using $F \in \mathbf{F}_\mu$)}\\
&\leq \frac{q-2}{q}\varepsilon^{\frac{q}{q-2}} \bigg [\mu \frac{q^{2}}{4}J_{q}+\mu \int_{s}^{\tau}\langle |w|^{q}\rangle\bigg ] +\frac{2}{q}\varepsilon^{-\frac{q}{2}}\int_{s}^{\tau} \big\langle |F|^2 \big\rangle.
\end{align*}
This completes the estimate of the right-hand side term $S=\sum_{r=1}^d \int_s^\tau \langle |F|,\phi_r \rangle$.

\medskip

4.~Next, in the weighted setting, i.e.\,when the weight $\rho$ is non-constant, after we multiply the equation by the test function $\phi_r$, we integrate with respect to the measure $\rho(x) dx$. All occurrences of $\nabla \rho$ resulting from the integration by parts are removed using \eqref{two_est}. In the previous estimate we will have $\int_s^\tau \big\langle |F|^2 \rho \big\rangle$, as claimed,  instead of $\int_s^\tau\big\langle |F|^2 \big\rangle$. For details, if needed, see \cite{KM_JDE}, or refer to the end of the proof of Proposition \ref{prop1} below.

\bigskip

\section{Proof of Theorem \ref{thm1}(\textit{ii}): Feller propagator}

The assumptions of Theorem \ref{thm1}(\textit{ii}) imply that the symmetric matrix fields $a=\sigma\sigma^{\top}$, $a_n=\sigma_n \sigma_n^{\top}$ satisfy
$$
a, a_n \in [L^\infty(\mathbb R^{1+d})]^{d \times d}, \quad a,a_n \geq \kappa I \text{ for some fixed $\kappa>0$},
$$
$$
a_n \rightarrow a \quad  \text{ in $[L_{\loc}^2(\mathbb R,L^2_\rho(\mathbb R^d)]^{d \times d}$},
$$
$$
\nabla a_n \rightarrow \nabla a \quad  \text{ in $[L_{\loc}^2(\mathbb R,L^2_\rho(\mathbb R^d)]^{d}$}
$$
and
$$(\nabla_r a_{n,il})_{i=1}^{d} \in  \mathbf{F}_{\gamma_{rl}} \quad \text{ for all }r,l=1,\dots,d, \qquad \nabla a_n \in  \mathbf{F}_{\delta_{a}}$$ 
with the same constants $c_{\delta_a}$, $c_\nu$. The form-bounds $\delta_a$ and $\gamma_{rl}$ are bounded from above by $C\nu$ for appropriate $C>0$, where $\nu$ is the form-bound of $|\nabla \sigma|$ -- this follows from the product rule and $\|\sigma_n\|_\infty \leq \|\sigma\|_\infty<\infty$. So, it suffices for us to show that if $\delta_a$ and $\gamma_{rl}$ are sufficiently small, then assertion (\textit{ii}) of Theorem \ref{thm1} holds.

Without loss of generality, the ellipticity constant $\kappa=1$. We can reverse the direction of time and carry out the proof for solutions of the corresponding initial-value problem written in the divergence form:
\begin{align}
\label{CP4}
\left\{
\begin{array}{l}
\partial_t u - \frac{1}{2} \nabla \cdot a_n \cdot \nabla + \frac{1}{2} \nabla a_n\cdot \nabla u_n=0 \qquad \text{ in } ]s,\infty[ \times \mathbb R^d, \\
 u_n|_{t=s}=f.
\end{array}
\right.
\end{align}
It suffices to construct the forward Feller propagator for $f \in C_c^\infty$ and then use the density argument and the $L^\infty$ contractivity.

We subtract the approximating equations in \eqref{CP4}:
\begin{equation}
\label{diffeq}
(\partial_t + \frac{1}{2} A_n +\frac{1}{2}  \nabla a_n \cdot \nabla )g =  \frac{1}{2} \nabla \cdot (a_n -a_m)\cdot \nabla u_m -  \frac{1}{2} (\nabla a_n - \nabla a_m ) \cdot \nabla u_m
\end{equation} 
where, recall, $A_n=-\nabla \cdot a_n \cdot \nabla$, and we set $g:=u_n-u_m$.

\begin{lemma}[Iteration inequality]
\label{Iter_lemma}
Fix some $0 < \alpha < 1$ and 
$\varsigma$, $\lambda>1$ such that $$\varsigma<\frac{d}{d-2+2\alpha},$$ and $$\frac{1/(1-\alpha)}{\lambda}=\frac{d/(d-2+2\alpha)}{\varsigma}.$$
Then there exist $\theta>0$, $k>1$ and $m_0$ such that for all $m,n \geq m_0$,
for all $r \geq r_0>\frac{2}{2-\sqrt{\delta_a}}$, we have
\begin{multline}
\label{iterineq}
\|u_m-u_n\|_{L^{\frac{r}{1-\alpha}}([s,s+\theta],L^{\frac{rd}{d-2+2\alpha}})}\\ \leq \left(C_0 \|\nabla u_m\|^2_{L^{2\lambda'}([s,s+\theta],L^{2\varsigma'})}\right)^{\frac{1}{r}}(r^{2k})^{\frac{1}{r}}\|u_m-u_n\|_{L^{(r-2)\lambda}([s,s+\theta],L^{(r-2)\varsigma})}^{1-\frac{2}{r}},
\end{multline}
for a constant $C_0$ that does not depend on $m$ or $s$. Here, as usual, $\frac{1}{\varsigma}+\frac{1}{\varsigma'}=1$, $\frac{1}{\lambda}+\frac{1}{\lambda'}=1$. 
\end{lemma}

The proof repeats the proof of \cite[Lemma 2]{Ki_Osaka}, see also \cite{KiS_note}, and consists of applying Cauchy-Schwarz inequality and using the form-boundedness of $\nabla a_n$. The only new term that we get, after multiplying \eqref{diffeq} by $g=u_m-u_n$, is 
$$
\int_s^\tau \langle \nabla \cdot (a_n -a_m)\cdot \nabla u_m, \eta |\eta |^{1-\frac{2}{r}} \rangle
$$
where $\eta :=  g|g|^{\frac{r-2}{2}}$. It is estimated, after integrating by parts and applying Cauchy-Schwarz, using the uniform boundedness of $a_n$ (in $(t,x)$ and $n$).

\medskip

Now, one can iterate the  inequality in Lemma \ref{Iter_lemma}; these are Moser's iterations. 
The crucial point here is that we can appeal to the gradient bound of Proposition \ref{prop_grad}, after applying the Sobolev embedding theorem and using the interpolation inequality, to establish that in \eqref{iterineq}
$$
\sup_m \|\nabla u_m\|^2_{L^{2\lambda'}([s,s+\theta],L^{2\varsigma'})}<\infty.
$$
Set $$D_T=\{(s,t) \mid 0  \leq s \leq t \leq T\}, \quad D_{T,\,\theta}:=D_{T} \cap \{(s,t) \mid 0 \leq t-s \leq \theta\}, \quad \theta<T.$$
Let $U_n=\{U_n^{t,s}\}_{0 \leq s \leq t \leq T}$ denote the (forward) Feller propagator corresponding to the initial-value problem \eqref{CP4}.

\begin{lemma}
\label{Iter2_lemma}
For any $ r_0>\frac{2}{2-\sqrt{\delta_a}}$ there exist $\theta>0$, constants $B<\infty$ and $$\gamma:=
\bigl( 1-\frac{\varsigma d }{d+2}\bigr)\biggl(1-\frac{\varsigma d }{d+2}+\frac{2\varsigma}{r_0} \biggr)^{-1}
>0$$ {\rm(}$1<\varsigma<\frac{d+2}{d}${\rm)} independent of $m,n$ such that
\begin{equation}
\label{iterineq0}
\|U_mf-U_nf\|_{L^\infty(D_{T,\,\theta} \times \mathbb R^d)}  \leq B \sup_{0 \leq s \leq T-\theta}\|U_mf-U_nf\|^\gamma_{L^{r_0}([s,s+\theta],L^{r_0})} \quad \text{ for all } n,m.
\end{equation}
\end{lemma}

See \cite[proof of Lemma 3]{Ki_Osaka}, see also \cite{KiS_note}.

\medskip
 
That $\{U_mf\}$ indeed converges in $L^{r_0}([s,s+\theta], L^{r_0})$ is proved in the same way as in \cite[Lemma 4]{Ki_Osaka}. 
That is, it suffices to prove that $\{U_mf\}$ is a Cauchy sequence in $L^{2}([s,s+\theta] \times \mathbb R^d)$ and then interpolate with the $L^\infty$ contractivity bound for $u_n$, $u_m$.
In turn, the proof of convergence in $L^{2}([s,s+\theta],L^2)$
consists of two steps:

1) Subtracting again the equations for $u_m$, $u_n$ and applying Cauchy-Schwarz and the uniform in $n$ form-boundedness of $\nabla a_n$, but this time working in the weighted space $L^2_\rho$. The only new term that we get is 
$$
\int_s^\tau \langle \nabla \cdot (a_n -a_m)\cdot \nabla u_m, \rho (u_m-u_n)\rangle.
$$
We need to show that it tends to zero as $n$, $m \rightarrow \infty$. This is done using the integration by parts where the property \eqref{two_est} plays a crucial role,  invoking the gradient bound of Proposition \ref{prop_grad} and the convergence $a_n-a_m \rightarrow 0$ in $L^2_{\loc}(\mathbb R,L^2_\rho(\mathbb R^d))$.

2) ``cutting tails'' of $u_n$, i.e.\,working with weight $1-\rho$ and using the compact support assumption on the initial function $f$ (so $u_n$ must be small, uniformly in $n$, when considered away from the support of $f$).

\medskip

This yields the existence of the limit
$$
\exists\; s\mbox{-}C_\infty\mbox{-}\lim_{m \rightarrow \infty}U_m^{t,s}=:U^{t,s}f \quad \text{(uniformly in $(s,t) \in D_{T,\theta}$)}, 
$$ 
which can then be extended to $D_T$ by showing the consistency of the limits.
The sought backward Feller propagator is now obtained as $P^{r,t}:=U^{T-r,T-t}$. \hfill \qed

 \medskip

 \section{Proof of Theorem \ref{thm2}} 
\label{thm2_proof_sect}

Following the discussion after Theorem \ref{thm2}, we only need to comment on the proof of the strong existence part (\textit{i}).

In the next estimate, assume that the functions $f_i \in L^2_{\loc}(\mathbb R^{d+1})$ ($i \geq 1$) are form-bounded
\begin{equation}
\label{fbb_fi}
\|f_i(t,\cdot)\varphi\|_2^2   \leq \nu \|\nabla \varphi\|_2^2 +  g_{\nu}(t) \|\varphi \|_2^2
\end{equation}
for some $\nu>0$, with $0 \leq g_\nu \in L^{1+\varepsilon}(\mathbb R)$.

By the assumptions of the theorem,
\begin{equation}
\label{b_fbb}
b \in \mathbf{F}_{\delta,g_\delta} \quad \text{ with $0 \leq g_\delta \in L^{1+\varepsilon}(\mathbb R)$.}
\end{equation}

Everywhere below, drift $b$ and functions $f_i$ are additionally assumed to be smooth. However, the constants in the estimates will not depend on smoothness or boundedness of $b$ and $f_i$. 

By classical theory, there exists a unique strong solution $X_t^x$ to
$$
X_t^x=x+\int_0^t b(\tau,X_\tau^x)d\tau + W_t.
$$
Let $0 \leq T_0 \leq T_1 \leq T$ where $T$ is fixed.

\begin{proposition3prime}
Assume that $\theta>0$ in the definition of $\rho$ is fixed sufficiently
small. There exist positive constants $C_0$ and $K$, independent of the
smoothness and boundedness of $b$ and $f_i$, such that, for every $n\geq 1$,
every $0\leq T_0<T_1\leq T$, and every $1\leq \alpha_i\leq d$,
\begin{equation}
 \left\|
 \mathbf E\int_{\Delta_n(T_0,T_1)}
       \prod_{i=1}^n
       \nabla_{\alpha_i}f_i(t_i,X_{t_i}^{\,\cdot})
       \,dt_1\cdots dt_n
 \right\|_{L^2_\rho(\mathbb R^d)}^2
 \leq
 C_0K^n(T_1-T_0)^{\frac{\varepsilon}{1+\varepsilon}}.
\end{equation}
Moreover, $K$ can be made arbitrarily small by requiring the form-bounds
$\delta$ and $\nu$ in \eqref{b_fbb} and \eqref{fbb_fi} to be sufficiently
small. 
\end{proposition3prime}

\begin{proof}
1.~We first assume 
\begin{equation}
\label{A_b}
 \sprt f_i\subset\mathbb R\times B_R(0), \quad  i\geq1,
 \tag{$\mathrm{A}_b$}
\end{equation}
for some $R>0$ independent of $i$, 
and work with $\rho\equiv1$. Fix $n$,
put $u_{n+1}=1$, and define consecutively
\begin{equation*}
 h_k=(\nabla_{\alpha_k}f_k)u_{k+1},
 \qquad k=1,\ldots,n,
\end{equation*}
where $u_k$ solves the terminal-value problem
\begin{equation}
 \partial_tu_k+\frac{1}{2}\Delta u_k+b\cdot\nabla u_k+h_k=0,
 \qquad
 u_k(T_1,\cdot)=0
 \quad\text{on }[T_0,T_1].
\label{eq_k1}
\end{equation}
As in \cite[Proof of Lemma 4.2]{RZ},
\begin{equation}
\mathbf E_{\mathcal F^W_{T_0}} \int_{\Delta_n(T_0,T_1)}\prod_{i=1}^n
\nabla_{\alpha_i}f_i(t_i,X_{t_i}^x)
\,dt_1\cdots dt_n=u_1(T_0,X_{T_0}^x).
\end{equation}

Set
\begin{equation*}
v(t,x):=\mathbf Eu_1(T_0,X^x_{T_0}), \quad 0\leq t\leq T_0,
\end{equation*}
and define
\begin{equation*}
\tilde{u}(t,\cdot):=
\begin{cases}
v(t,\cdot), &0\leq t\leq T_0,\\
u_1(t,\cdot), &T_0\leq t\leq T_1,
\end{cases}
\quad
\tilde{h}(t,\cdot):=
\begin{cases}
0, &0\leq t<T_0,\\
h_1(t,\cdot), &T_0\leq t\leq T_1.
\end{cases}
\end{equation*}
Then
\begin{equation*}
\partial_t\tilde{u}+\frac{1}{2}\Delta \tilde{u}+b\cdot\nabla \tilde{u} +\tilde{h}=0, \quad \tilde{u}(T_1,\cdot)=0,
\end{equation*}
and
\begin{equation*}
\tilde{u}(0,x)=\mathbf E\int_{\Delta_n(T_0,T_1)} \prod_{i=1}^n \nabla_{\alpha_i}f_i(t_i,X_{t_i}^x)\,dt_1\cdots dt_n.
\end{equation*}

Put, as in the proof of Proposition \ref{prop1},
\begin{equation}
\label{FG}
F(t):=\lambda\int_0^t\bigl[g_\delta(s)+g_\nu(s)\bigr]\,ds, \quad G(t):=\lambda\int_{T_0}^t\bigl[g_\delta(s)+g_\nu(s)\bigr]\,ds.
\end{equation}
In particular,
\begin{equation*}
 e^{F(t)}=e^{F(T_0)}e^{G(t)},
 \qquad T_0\leq t\leq T_1.
\end{equation*}

Fix $M,\beta>0$. Multiplying the equation for $\tilde{u}$ by
$e^F\tilde{u}$ and integrating over $[0,T_1]\times\mathbb R^d$, we
obtain
\begin{align*}
 &\frac{1}{2}\langle \tilde{u}^2(0)\rangle
 +\frac{1}{2}\int_0^{T_1}e^F\langle|\nabla\tilde{u}|^2\rangle\,dt
 +\frac{\lambda}{2}\int_0^{T_1}
       (g_\delta+g_\nu)e^F\langle\tilde{u}^2\rangle\,dt
 \\
 &\qquad =
 \int_0^{T_1}e^F
       \langle b\cdot\nabla\tilde{u},\tilde{u}\rangle\,dt
 +
 \int_{T_0}^{T_1}e^F
       \langle(\nabla_{\alpha_1}f_1)u_2,u_1\rangle\,dt.
\end{align*}
The drift term satisfies
\begin{equation*}
 \int_0^{T_1}e^F
       \langle b\cdot\nabla\tilde{u},\tilde{u}\rangle\,dt
 \leq
 \sqrt{\delta}\int_0^{T_1}
       e^F\langle|\nabla\tilde{u}|^2\rangle\,dt
 +
 \frac{1}{2\sqrt{\delta}}\int_0^{T_1}
       g_\delta e^F\langle\tilde{u}^2\rangle\,dt.
\end{equation*}
Integrating by parts in the last term of the preceding energy identity
and applying the Cauchy-Schwarz inequality twice, we get
\begin{align*}
&\int_{T_0}^{T_1}e^F\langle(\nabla_{\alpha_1}f_1)u_2,u_1\rangle\,dt \\
&\quad = -\int_{T_0}^{T_1}e^F\langle f_1(\nabla_{\alpha_1}u_2),u_1\rangle\,dt-\int_{T_0}^{T_1}e^F\langle f_1u_2,\nabla_{\alpha_1}u_1\rangle\,dt \\
&\quad \leq M\int_{T_0}^{T_1}e^F\langle f_1^2,u_1^2\rangle\,dt +\frac{1}{4M}\int_{T_0}^{T_1} e^F\langle|\nabla u_2|^ 2\rangle\,dt \\
&\qquad\quad +\beta\int_{T_0}^{T_1}e^F\langle f_1^2,u_2^2\rangle\,dt +\frac{1}{4\beta}\int_{T_0}^{T_1} e^F\langle|\nabla u_1|^2\rangle\,dt \\
&\quad \leq \left(M\nu+\frac{1}{4\beta}\right)\int_{T_0}^{T_1}e^F\langle|\nabla u_1|^2\rangle\,dt \\
&\qquad\quad +\left(\beta\nu+\frac{1}{4M}\right) \int_{T_0}^{T_1}e^F\langle|\nabla u_2|^2\rangle\,dt \\
&\qquad\quad +M\int_{T_0}^{T_1}g_\nu e^F\langle u_1^2\rangle\,dt +\beta\int_{T_0}^{T_1}g_\nu e^F\langle u_2^2\rangle\,dt.
\end{align*}

Set
\begin{equation*}
 C_1:=\frac12-\sqrt{\delta}-M\nu-\frac{1}{4\beta},
 \qquad
 C_2:=\beta\nu+\frac{1}{4M}.
\end{equation*}
We first choose $M$ and $\beta$ sufficiently large, then choose a
constant $c>0$ sufficiently large, and finally require $\delta$ and
$\nu$ to be sufficiently small, so that $C_1>0$ and $K:=
\frac{C_2}{C_1} \vee \frac{\beta}{cC_1}<1$.
We can furthermore choose the aforementioned constants to make $K$ as small as needed. Fix $\lambda$ sufficiently large that
\begin{equation*}
 \frac{\lambda}{2}-\frac{1}{2\sqrt{\delta}}\geq0,
 \qquad
 \frac{\lambda}{2}-M\geq cC_1.
\end{equation*}
Then
\begin{equation}
 \langle\widetilde u^2(0)\rangle
 \leq
 CK\left(
       \int_{T_0}^{T_1}e^G\langle|\nabla u_2|^2\rangle\,dt
       +
       c\int_{T_0}^{T_1}g_\nu e^G\langle u_2^2\rangle\,dt
 \right),
\label{u2_est}
\end{equation}
where the factor $e^{F(T_0)}$ has been absorbed into $C$.

We next apply the same calculation to $u_k$, $2\leq k\leq n-1$.
Multiplying \eqref{eq_k1} by $e^Gu_k$ and integrating over
$[T_0,T_1]\times\mathbb R^d$ gives
\begin{align*}
 &C_1\int_{T_0}^{T_1}e^G\langle|\nabla u_k|^2\rangle\,dt
 +
 \left(\frac{\lambda}{2}-M\right)
 \int_{T_0}^{T_1}g_\nu e^G\langle u_k^2\rangle\,dt
 \\
 &\qquad
 +
 \left(\frac{\lambda}{2}-\frac{1}{2\sqrt{\delta}}\right)
 \int_{T_0}^{T_1}g_\delta e^G\langle u_k^2\rangle\,dt
 \\
 &\quad \leq
 C_2\int_{T_0}^{T_1}e^G
       \langle|\nabla u_{k+1}|^2\rangle\,dt
 +
 \beta\int_{T_0}^{T_1}g_\nu e^G
       \langle u_{k+1}^2\rangle\,dt.
\end{align*}
Consequently,
\begin{equation}
 \int_{T_0}^{T_1}e^G
       \left(
          \langle|\nabla u_k|^2\rangle
          +c\,g_\nu\langle u_k^2\rangle
       \right)\,dt
 \leq
 K\int_{T_0}^{T_1}e^G
       \left(
          \langle|\nabla u_{k+1}|^2\rangle
          +c\,g_\nu\langle u_{k+1}^2\rangle
       \right)\,dt.
\label{uk1_est}
\end{equation}
For $n\geq2$, iteration gives
\begin{equation*}
 \int_{T_0}^{T_1}e^G
       \left(
          \langle|\nabla u_2|^2\rangle
          +c\,g_\nu\langle u_2^2\rangle
       \right)\,dt
 \leq
 K^{n-2}
 \int_{T_0}^{T_1}e^G
       \left(
          \langle|\nabla u_n|^2\rangle
          +c\,g_\nu\langle u_n^2\rangle
       \right)\,dt.
\end{equation*}
It remains to estimate the last integral. Since $u_{n+1}=1$, we have
$h_n=\nabla_{\alpha_n}f_n$. The energy inequality for $u_n$ and
the integration by parts give
\begin{align*}
 &\int_{T_0}^{T_1}e^G
       \left(
          \langle|\nabla u_n|^2\rangle
          +c\,g_\nu\langle u_n^2\rangle
       \right)\,dt
 \\
 &\qquad\leq
 C\left|
 \int_{T_0}^{T_1}e^G
       \langle\nabla_{\alpha_n}f_n,u_n\rangle\,dt
 \right|
 \\
 &\qquad=
 C\left|
 \int_{T_0}^{T_1}e^G
       \langle f_n,\nabla_{\alpha_n}u_n\rangle\,dt
 \right|
 \\
 &\quad\leq
 \frac{1}{2}
 \int_{T_0}^{T_1}e^G
       \left(
          \langle|\nabla u_n|^2\rangle
          +c\,g_\nu\langle u_n^2\rangle
       \right)\,dt
 +
 C\int_{T_0}^{T_1}e^G\langle f_n^2\rangle\,dt.
\end{align*}
Hence,
\begin{equation}
 \int_{T_0}^{T_1}e^G
       \left(
          \langle|\nabla u_n|^2\rangle
          +c\,g_\nu\langle u_n^2\rangle
       \right)\,dt
 \leq
 C\int_{T_0}^{T_1}e^G\langle f_n^2\rangle\,dt.
\label{un_est}
\end{equation}
Choose $\varphi\in C_c^\infty(\mathbb R^d)$ such that
$\varphi\geq1$ on $B_R(0)$. By \eqref{A_b} and \eqref{fbb_fi},
\[
 \langle f_n^2(t)\rangle
 \leq
 \langle f_n^2(t),\varphi^2\rangle
 \leq
 \nu\langle|\nabla\varphi|^2\rangle
 +g_\nu(t)\langle\varphi^2\rangle.
\]
Since $G$ is bounded on $[T_0,T_1]$, H\"older's inequality yields
\begin{align}
 \int_{T_0}^{T_1}e^G\langle f_n^2\rangle\,dt \leq
 C\left[
       T_1-T_0+\int_{T_0}^{T_1}g_\nu(t)\,dt
   \right] \leq
 C(T_1-T_0)^{\frac{\varepsilon}{1+\varepsilon}},
\label{int_last}
\end{align}
where we used that $T$ is fixed.

Combining \eqref{u2_est}-\eqref{int_last}, we obtain, for $n\geq2$,
\begin{equation*}
 \left\|
 \mathbf E\int_{\Delta_n(T_0,T_1)}
       \prod_{i=1}^n
       \nabla_{\alpha_i}f_i(t_i,X_{t_i}^{\,\cdot})
       \,dt_1\cdots dt_n
 \right\|_2^2
 \leq
 CK^{n-1}(T_1-T_0)^{\frac{\varepsilon}{1+\varepsilon}}.
\end{equation*}
After changing the constant $C$, this has the asserted form
$C_0K^n(T_1-T_0)^{\frac{\varepsilon}{1+\varepsilon}}$.

When $n=1$, the same conclusion follows directly by applying the
energy inequality to $\tilde{u}$ and using
\begin{equation*}
 \int_{T_0}^{T_1}e^F
       \langle\nabla_{\alpha_1}f_1,\tilde{u}\rangle\,dt
 =
 -\int_{T_0}^{T_1}e^F
       \langle f_1,\nabla_{\alpha_1}\tilde{u}\rangle\,dt,
\end{equation*}
followed by the preceding cutoff estimate for $f_1$. Taking $C_0$ large, if necessary,
gives the sought estimate for $n=1$.

\medskip

2.~Let us remove the assumption \eqref{A_b}. We repeat the preceding
calculation in $L^2_\rho(\mathbb R^d)$. Applying the form-boundedness
condition on $f_i$ to ``test function'' $\sqrt{\rho} u$ gives, for every fixed $\eta>0$,
\begin{equation*}
 \langle f_i^2,u^2\rangle_\rho
 \leq
 (1+\eta)\nu\langle|\nabla u|^2\rangle_\rho
 +
 \bigl(g_\nu+C_\eta\nu\theta\bigr)
 \langle u^2\rangle_\rho,
\end{equation*}
and the analogous estimate holds for $b$ (with $\delta$ and $g_\delta$). Here we have used \eqref{two_est}.
The integration by parts in the diffusion term produces additional
terms bounded by
\begin{equation*}
 \eta\langle|\nabla u|^2\rangle_\rho
 +C_\eta\theta\langle u^2\rangle_\rho.
\end{equation*}
In addition, integration by parts in
$\langle(\nabla_{\alpha_k}f_k)u_{k+1},u_k\rangle_\rho$ produces
\begin{equation*}
-\int_{\mathbb R^d}f_ku_{k+1}u_k\,\nabla_{\alpha_k}\rho\,dx.
\end{equation*}
This term is estimated by
\begin{align*}
\left|\int_{\mathbb R^d}f_ku_{k+1}u_k\,\nabla_{\alpha_k}\rho\,dx\right|
&\leq\eta\langle f_k^2,u_{k+1}^2\rangle_\rho+C_\eta
\left\langle\frac{|\nabla\rho|^2}{\rho^2},u_k^2\right\rangle_\rho\\
&\leq \eta\langle f_k^2,u_{k+1}^2\rangle_\rho + C_\eta\theta\langle u_k^2\rangle_\rho.
\end{align*}
Thus, after first fixing $\eta$ and $\theta$ sufficiently small, all
the additional terms are absorbed by the preceding argument, with
$g_\delta+g_\nu$ in \eqref{FG} replaced by
$g_\delta+g_\nu+C\theta$.

Finally, using \eqref{fbb_fi} with $\sqrt{\rho}$ as a test function,
\begin{equation*}
\langle f_n^2(t)\rangle_\rho \leq \nu\langle|\nabla\sqrt{\rho}|^2\rangle +g_\nu(t)\langle\rho\rangle \leq C\nu\theta+Cg_\nu(t).
\end{equation*}
It follows that
\begin{equation*}
 \int_{T_0}^{T_1}e^G
       \langle f_n^2(t)\rangle_\rho\,dt
 \leq
 C(T_1-T_0)^{\frac{\varepsilon}{1+\varepsilon}}.
\end{equation*}
The weighted versions of \eqref{u2_est}-\eqref{int_last} now give the
asserted $L^2_\rho(\mathbb R^d)$ estimate.
\end{proof}

The next proposition, which is a consequence of Proposition $\ref{prop1}'$, verifies compactness of the Wiener-Sobolev space following \cite{RZ}; the reader can compare it with Proposition \ref{prop2} in the proof of Theorem \ref{thm1}.

\begin{proposition4prime}
For every $r \geq 1$, there exist constants $K_1$, $K_2$ (independent of smoothness or boundedness of $b$) such that

\medskip

{\rm (\textit{i})} $\sup_{y \in \mathbb R^d}\|\nabla X_t^x - I\|_{L^{2r}(B_1(y),L^r(\Omega))} \leq K_1 t^{\frac{1}{4r}\frac{\varepsilon}{1+\varepsilon}}$ for all $0 \leq t \leq T$;

\medskip

{\rm (\textit{ii})} $\sup_{y \in \mathbb R^d}\|D_sX_t^x - I\|_{L^{2r}(B_1(y),L^r(\Omega))} \leq K_1(t-s)^{\frac{1}{4r}\frac{\varepsilon}{1+\varepsilon}}$ for a.e.\,$s \in [0,T]$ and $0 \leq s \leq t \leq T$;

\medskip

{\rm (\textit{iii})} $\sup_{y \in \mathbb R^d}\|D_s X_t^x - D_{s'}X_t^x\|_{L^{2r}(B_1(y),L^r(\Omega))} \leq K_2|s-s'|^{\frac{1}{8r}\frac{\varepsilon}{1+\varepsilon}}$ for a.e.\,$s,s' \in [0,T]$ and $0 \leq s,s' \leq t \leq T$.
\end{proposition4prime}

In detail, since $b$ is bounded and smooth, one has
\begin{equation}
\label{sde_b_deriv}
\nabla X_t^x-I=\int_0^t \nabla b(s,X_s^x)\cdot \nabla X_s^x ds.
\end{equation}
We iterate this identity, 
$$
\nabla X_t^x-I=\sum_{n=1}^\infty \int_{\Delta_n(t)} \prod_{i=1}^n \nabla b(t_i,X_{t_i}^x)dt_1\dots dt_n,
$$
so
\begin{equation}
\label{sum}
\|\nabla X_t^x-I\|_{L^{2r}(\mathbb R^d,L^r(\Omega))} \leq \sum_{n=1}^\infty \left\|\int_{\Delta_n(t)} \prod_{i=1}^n \nabla b(t_i,X_{t_i}^x) dt_1\dots dt_n \right\|_{L^{2r}(\mathbb R^d,L^r(\Omega))}.
\end{equation}
It is at this step that we apply Proposition $\ref{prop1}'$ with $f_i$ taken to be the derivatives of the components of $b$. The Malliavin derivatives satisfy an SDE of the same type as \eqref{sde_b_deriv} and are controlled in the same way using Proposition $\ref{prop1}'$. The remainder of the proof is identical to \cite[Proof of Theorem 1]{KiM_strong} and follows closely the proof of Theorem \ref{thm1}.

\bigskip

\appendix

\section{Example of $\{b_n\}$}

\label{Conv_b_n}

In Theorem \ref{thm2}, we consider a bounded smooth approximation $\{b_m\}$ of form-bounded $b \in \mathbf{F}_{\delta, g_\delta}$ that satisfies

1) $\{b_m\}$ are form-bounded with the same $\delta$ and $\sup_m \|g_{\delta,m}\|_{L^{1+\varepsilon}(\mathbb R)}<\infty$;

2)
 \begin{equation}
\label{global_cnv}
 b_m\rightarrow b \quad\text{in } [L_{\loc}^2(\mathbb R,L^2_{\rho}(\mathbb R^d))]^d
 \quad\text{as }m\to\infty.
 \end{equation}

\medskip

One straightforward way to construct $\{b_m\}$ is as follows. 

Let $\{\gamma_\varepsilon\}$ denote the standard mollifier on $\mathbb R^{1+d}$, i.e.\,a family of positive smooth functions in $C_c^\infty(\mathbb R^{1+d})$ defined by 
$$
\gamma_{\varepsilon}(s,y):=\frac{1}{\varepsilon^{1+d}}\gamma(s/\varepsilon,y/\varepsilon), \quad (s,y) \in \mathbb R^{1+d}
$$ 
with $\int_{\mathbb R^{1+d}}\gamma=1$.

\begin{lemma}
\label{lem_b}
The vector fields
$$
b_m:=\gamma_{\varepsilon_m} \star b,
$$
where $\star$ is the convolution on $\mathbb R^{1+d}$ and $\varepsilon_m \downarrow 0$,
satisfy 1), 2).
\end{lemma}

\begin{remark}
It is easy to see that $\{b_m\}$ defined slightly differently,
$$
b_m:=\gamma_{\varepsilon_m} \star (\mathbf{1}_m b),
$$
where $\mathbf{1}_m$ is the indicator function of the set $\{(t,x) \in \mathbb R^{1+d} \mid |t| \leq m, |x| \leq m, |b(t,x)| \leq m\}$,
are form-bounded with form-bound slightly larger -- but still independent of $m$ -- than $\delta$ (not a problem since all our conditions on all form-bounds are strict inequalities), provided that we choose $\varepsilon_m \downarrow 0$ sufficiently rapidly. Indeed, we can write
$$
b_m:=\mathbf{1}_m b + \big(\gamma_{\varepsilon_m} \star \mathbf{1}_m b- \mathbf{1}_m b\big).
$$
The first term in the RHS, i.e.\,vector field $\mathbf{1}_m b$, is obviously form-bounded with the same $\delta$ and even with the same function $g_\delta$. The second term is a difference between a bounded vector field with compact support and its mollification; it is not difficult to modify the first example in Section \ref{sing_drift_sect} dealing with the Lebesgue class drifts to show that, provided that $\varepsilon_m\downarrow 0$ sufficiently rapidly, the form-bound of $\gamma_{\varepsilon_m} \star \mathbf{1}_m b- \mathbf{1}_m b$ goes to zero as $m \rightarrow \infty$. Finally, we can use the fact that the sum of two form-bounded vector fields is a form-bounded vector field.
\end{remark}

\begin{proof}[Proof of Lemma \ref{lem_b}]
The boundedness and smoothness of $b_m$ follow from the standard properties of mollifiers. 

1. Let us prove 1). For every $t \in \mathbb R$ and $\varphi\in C_c^\infty(\mathbb R^d)$,
\begin{align*}
\|b_m(t,\cdot)\varphi\|_2^2 &\leq \int_{\mathbb R}\int_{\mathbb R^d} \gamma_{\varepsilon_m}(s,y) \int_{\mathbb R^d}
|b(t-s,x-y)|^2|\varphi(x)|^2 dxdyds.
\end{align*}
Hence
$$
\|b_m(t,\cdot)\varphi\|_2^2 \leq \int_{\mathbb R}\int_{\mathbb R^d} \gamma_{\varepsilon_m}(s,y) \int_{\mathbb R^d}|b(t-s,z)|^2|\varphi(z+y)|^2dzdyds. 
$$  
Applying the form-boundedness of $b(t-s, \cdot)$ to $\varphi (\cdot+y)$ and using the translation-invariance of the $L^2$ norm, we obtain
\begin{equation}
\|b_m(t,\cdot)\varphi\|_2^2 \leq \delta\|\nabla\varphi\|_2^2 + g_{\delta,m}(t) \|\varphi\|_2^2, 
\end{equation} 
where 
$
g_{\delta,m}$ is the convolution of $t \mapsto \int_{\mathbb{R}^d} \gamma_{\varepsilon_m}(\cdot,y)dy$ and $t \mapsto g_\delta(t)$. Since the former is itself a mollifier (on $\mathbb R$), and the mollification preserves the $L^{1+\varepsilon}$ norm, we have
$$
\|g_{\delta,m}\|_{L^{1+\varepsilon}(\mathbb R)} \leq  \|g_\delta\|_{L^{1+\varepsilon}(\mathbb R)},
$$
i.e.\,we have verified 1).

\medskip

 2. To verify 2), we first note that, by the standard properties of mollifiers,
\begin{equation}
\label{conv_l2_loc}
b_m\rightarrow b \quad\text{in } [L_{\loc}^2(\mathbb R^{1+d})]^d.
\end{equation}
Our goal now is to upgrade this convergence to \eqref{global_cnv}.

To this end, for each $R>0$, define functions
 $$
 \eta_R(y):=\xi_R(|y|), \qquad \xi_R(r):=
 \begin{cases}
 0, & 0\leq r<R,\\
 r-R, & R\leq r\leq R+1,\\
 1, & r>R+1.
 \end{cases}
 $$
 Then
 $$
 \mathbf{1}_{\mathbb{R}^d\setminus B_{R+1}} \leq \eta_R^2 \qquad\text{and}\qquad
 |\nabla\eta_R| \leq \mathbf{1}_{B_{R+1}\setminus B_R}.
 $$
Next, it is easily seen that $b_m-b$ is form-bounded with parameters $4\delta$ and $2(g_{\delta,m}+g_\delta)$, where, by the argument in the previous step, $\sup_m\|2(g_{\delta,m}+g_\delta)\|_{L^1(\mathbb R)}<\infty$ (note that we have $L^1$ here, not $L^{1+\varepsilon}$, although this is insignificant at this step).
Therefore, using the form-boundedness of $b_m-b$ with test function $\eta_R\sqrt{\rho}$, we have 
 \begin{align*}
\int_0^T \int_{\mathbb{R}^d\setminus B_{R+1}} |b_m(t,y)-b(t,y)|^2\rho(y)\,dydt
\leq
 4\delta T \bigl\|\nabla(\eta_R\sqrt{\rho})\bigr\|_2^2 + 2\int_0^T (g_{\delta,m}(t)+g_\delta(t)) 
 dt \bigl\|\eta_R\sqrt{\rho}\bigr\|_2^2.
 \end{align*}
Since $\rho$ is in $L^1(\mathbb R^d)$, vanishes at infinity and satisfies $\eqref{two_est}$, we have $\bigl\|\eta_R\sqrt{\rho}\bigr\|_2
 \rightarrow 0$ and $\bigl\|\nabla(\eta_R\sqrt{\rho})\bigr\|_2 \rightarrow 0 $ as $R\to\infty$. Therefore, 
 \begin{equation}
 \lim_{R\to\infty}\sup_m \int_0^T\int_{\mathbb{R}^d\setminus B_{R+1}} |b_m(t,y)-b(t,y)|^2\rho(y)\,dy\,dt =0.
 \end{equation}
 On the other hand, for each $R>0$, we have by \eqref{conv_l2_loc}  
 \begin{align*}
 &\int_0^T\int_{B_{R+1}} |b_m(t,y)-b(t,y)|^2\rho(y)\,dy\,dt\\
 &\leq \|\rho\|_{L^\infty(B_{R+1})} \int_0^T\int_{B_{R+1}} |b_m(t,y)-b(t,y)|^2\,dy\,dt \rightarrow 0
 \end{align*}
 as $m\to\infty$. This, and the previous convergence, yield \eqref{global_cnv}.    
 \end{proof}

 \bigskip

\section{Example of $\{\sigma_n\}$}

\label{approx_app}

Let $\sigma$ be as in Theorem \ref{thm1}, i.e.\,for every $t \in \mathbb R$, 

1) $\sigma(t,\cdot) \in [L^\infty(\mathbb R^d)]^{d \times d}$,

2) $\sigma(t,\cdot) \sigma(t,\cdot)^{\top} \geq \kappa I$ a.e.\,on $\mathbb R^d$,

3) the form-boundedness condition \eqref{nabla_sigma} on the derivatives of $\sigma$ holds, i.e.\,for a.e.\,$t \in \mathbb R$
\begin{equation*}
\||\nabla \sigma(t,\cdot)|\varphi\|_2^2   \leq \nu \|\nabla \varphi\|_2^2+ c_\nu\|\varphi\|_2^2 \quad \forall\,\varphi \in C_c^\infty(\mathbb R^d),
\end{equation*}

4) ($\mathbf{H}$) holds for $\sigma$, i.e.\,$\sigma$ is locally uniformly H\"{o}lder continuous as a $[L^2_{\rho}(\mathbb R^d))]^{d \times d}$-valued function.

\medskip

Let $\{\eta_\varepsilon\}$ denote the standard (spatial) mollifier on $\mathbb R^d$, i.e.\,for a fixed $0 \leq \eta \in C_c^\infty(B_1)$,  $\int_{\mathbb R^d} \eta(x)dx=1$, put $\eta_\varepsilon(x):=\varepsilon^{-d}\eta(x/\varepsilon)$.

\begin{lemma} 
\label{approx_lem}

Assuming that the form-bound $\nu$ is sufficiently small, the sequence $\sigma^\varepsilon(t,\cdot):=\eta_\varepsilon \ast \sigma(t,\cdot)$, satisfies, for all $\varepsilon$ sufficiently small, uniformly in a.e.\,$t \in \mathbb R$, 

$1'$) $\|\sigma^\varepsilon(t,\cdot)\|_{L^\infty(\mathbb R^{d})} \leq \|\sigma(t,\cdot)\|_{L^\infty(\mathbb R^{d})} $

$2'$) $
\sigma^\varepsilon(t,\cdot) (\sigma^\varepsilon(t,\cdot))^{\top} \geq \frac{\kappa}{4}I;
$

$3'$) \eqref{nabla_sigma} holds for $\sigma^\varepsilon$ with the same form-bound $\nu$ and constant $c_\nu$, i.e.\,independent of $\varepsilon$;

$4'$)  ($\mathbf{H}$) holds for $\sigma^\varepsilon$ with the same H\"{o}lder continuity exponent and constant.

$5')$ $$\|\sigma(t,\cdot)-\sigma^\varepsilon(t,\cdot)\|_{L_\rho^2(\mathbb R^d)} \leq C_1\varepsilon, \quad t \in \mathbb R.$$

$6'$) Set $a=\sigma\sigma^{\top}$, $a^\varepsilon=\sigma^\varepsilon(\sigma^\varepsilon)^{\top}$. Then 
$$
\nabla_k a^\varepsilon \rightarrow \nabla_k a \quad \text{ in } [L_{\loc}^2(\mathbb R,L^2_\rho(\mathbb R^d))]^{d \times d} \text{ as $\varepsilon \downarrow 0$}
$$
for all $1 \leq k \leq d$. That is, we have componentwise convergence of the derivatives of $a^\varepsilon$, which in particular implies the convergence of the row-divergence operators $\nabla a^\varepsilon \rightarrow \nabla a$ in $[L_{\loc}^2(\mathbb R,L^2_\rho(\mathbb R^d))]^{d}$.
\end{lemma}

\begin{proof}
$1'$) is straightforward. Let us prove $3'$). We have $\nabla \sigma^\varepsilon(t,\cdot)=\eta_\varepsilon \ast \nabla \sigma(t,\cdot)$. Therefore, by Cauchy-Schwarz,
$$
|\nabla \sigma^\varepsilon(t,\cdot)|^2 \leq \eta_\varepsilon \ast |\nabla \sigma(t,\cdot)|^2.
$$
Now, we verify form-boundedness of $|\nabla \sigma^\varepsilon(t,\cdot)|$ uniformly in $t$:
\begin{align*}
\langle |\nabla \sigma^\varepsilon(t,\cdot)|^2,\varphi^2\rangle & \leq \langle \eta_\varepsilon \ast |\nabla \sigma(t,\cdot)|,\varphi^2\rangle \\
& = \langle |\nabla \sigma(t,\cdot)|^2,\eta_\varepsilon \ast \varphi^2 \rangle \\
& (\text{use form-boundedness \eqref{nabla_sigma} of $|\nabla \sigma(t,\cdot)|$}) \\
& \leq \nu \langle |\nabla \sqrt{\eta_\varepsilon \ast \varphi^2}|^2 \rangle + c_\nu \langle \eta_\varepsilon \ast \varphi^2 \rangle \\ 
& (\text{Cauchy-Schwarz yields $|\nabla \sqrt{\eta_\varepsilon \ast \varphi^2}|^2 \leq \eta_\varepsilon \ast |\nabla \varphi|^2$}) \\
& \leq \nu \langle |\nabla \varphi|^2\rangle + c_\nu \langle \varphi^2 \rangle,
\end{align*}
as claimed (see e.g.\,\cite[Sect.\,4.4]{KiS_theory} for details, if needed).

\medskip

Proof of $2'$). We will use $3'$), i.e.\,uniform in $\varepsilon$ form-boundedness of $|\nabla \sigma^\varepsilon|$. Let $t \in \mathbb R$.
Our goal is to show that the minimal eigenvalue $\lambda_{\min}$ of $\sigma^\varepsilon(t,x)(\sigma^\varepsilon(t,x))^{\top}$ is greater or equal to $\frac{\kappa}{4}$. 
To this end, we first note  that there exists $y \in B_\varepsilon(x)$ such that
\begin{equation}
\label{sigma_sigma}
|\sigma^\varepsilon(t,x)-\sigma(t,y)|_{{\rm o}} \leq C \sqrt{\nu+c_\nu \varepsilon^2}
\end{equation}
for a universal constant $C$,
where $|A|_{{\rm o}}:=\max_{v \in \mathbb R^d, |v|=1} |Av|$ (o stands for ``operator norm''); the operator norm and the Euclidean norms of $A$ are equivalent, but it will be convenient to work with the operator norm. We prove \eqref{sigma_sigma} below. With \eqref{sigma_sigma} at hand, we can estimate
$$
\sqrt{\lambda_{\min}(\sigma^\varepsilon(t,x)(\sigma^\varepsilon(t,x))^{\top})} \geq \sqrt{\lambda_{\min}(\sigma(t,y)\sigma(t,y)^{\top}}
) - |\sigma^\varepsilon(t,x)-\sigma(t,y)|_{{\rm o}},$$
which is, taking into account the variational characterization of eigenvalues of $\sigma\sigma^{\top}$, simply the inequality
$$
\min_{|v|=1}|\sigma^\varepsilon(t,x)v| \geq \min_{|v|=1}|\sigma(t,y)v| - \max_{|v|=1}|(\sigma^\varepsilon(t,x)-\sigma(t,y))v|.
$$
Therefore,
\begin{align*}
\sqrt{\lambda_{\min}(\sigma^\varepsilon(t,x)(\sigma^\varepsilon(t,x))^{\top})} & \geq \sqrt{\kappa} - C \sqrt{\nu+c_\nu \varepsilon^2} \\
& \geq \frac{\sqrt{\kappa}}{2}
\end{align*}
provided $\nu$ is sufficiently small, starting with some small $\varepsilon$, as claimed.

It remains to prove \eqref{sigma_sigma}. To this end, we compare $\sigma^\varepsilon(t,x)$ and $\sigma(t,y)$ to a common average: $\sigma_{B_\varepsilon(x)}(t):=\frac{1}{|B_\varepsilon(x)|}\int_{B_\varepsilon(x)}\sigma(t,z)dz$. 

Step 1: We have 
\begin{align*}
|\sigma_\varepsilon(t,x)-\sigma_{B_\varepsilon}(t)|_{{\rm o}}^2 & = \biggl|\int_{B_\varepsilon(x)} \eta_\varepsilon(x-z)(\sigma(t,z)-\sigma_{B_\varepsilon}(t))dz \biggr|_{{\rm o}}^2 \\
& \leq \int_{B_\varepsilon(x)} \eta_\varepsilon(x-z)|\sigma(t,z)-\sigma_{B_\varepsilon}(t)|_{{\rm o}}^2 dz \\
& (\text{use $\|\eta_\varepsilon\|_\infty \leq c\varepsilon^{-d}$}) \\
& \leq c\varepsilon^{-d}\int_{B_\varepsilon(x)} |\sigma(t,z)-\sigma_{B_\varepsilon}(t)|^2 dz \\
& (\text{apply the Poincar\'{e} inequality}) \\
& \leq C\varepsilon^{2-d} \int_{B_\varepsilon(x)} |\nabla \sigma(t,z)|^2dz = C\varepsilon^{2-d} \int_{\mathbb R^d} \mathbf{1}_{B_\varepsilon(x)}|\nabla \sigma(t,z)|^2dz\\
& (\text{apply form-boundedness of $|\nabla \sigma(t,\cdot)|$ after replacing $\mathbf{1}_{B_\varepsilon(x)}$ } \\
& \text{by non-negative smooth function  $\varphi_\varepsilon=1$ on $B_\varepsilon(x)$, $\varphi_\varepsilon=0$ on $\mathbb R^d - B_{2\varepsilon}(x)$}) \\
& \leq C_1\varepsilon^{2-d} \bigl(\nu \varepsilon^{d-2} + c_\nu \varepsilon^d \bigr) = C_1\big( \nu  + c_\nu \varepsilon^2\big).
\end{align*}

Step 2: Let us now compare $\sigma(t,y)$, for some choice of $y \in B_\varepsilon(x)$, and  $\sigma_{B_\varepsilon}(t)$. By the estimates above,
$$
\varepsilon^{-d}\int_{B_\varepsilon(x)} |\sigma(t,z)-\sigma_{B_\varepsilon}(t)|_{{\rm o}}^2 dz \leq C\big( \nu  + c_\nu \varepsilon^2\big).
$$
Therefore, 
\begin{align*}
|\{z \in B_\varepsilon(x) \mid |\sigma(t,z)-\sigma_{B_\varepsilon}(t)|_{{\rm o}}^2>2C\big( \nu  + c_\nu \varepsilon^2\big)\}| & \leq \frac{1}{2C\big( \nu  + c_\nu \varepsilon^2\big)}\int_{B_\varepsilon(x)}|\sigma(t,z)-\sigma_{B_\varepsilon}(t)|_{{\rm o}}^2dz \\
& \leq  
\frac{1}{2C\big( \nu  + c_\nu \varepsilon^2\big)} \varepsilon^d C\big( \nu  + c_\nu \varepsilon^2\big) = \frac{\varepsilon^d}{2}.
\end{align*}
Since the set of $x$ where the ellipticity condition $\sigma(t,x)\sigma(t,x)^{\top} \geq \kappa I$ is violated, if non-empty, has measure zero, by the previous estimate there exist (plenty of) $y \in B_\varepsilon(x)$ such that 
$$
|\sigma(t,y)-\sigma_{B_\varepsilon}(t)|_{{\rm o}}^2 \leq 2C\big( \nu  + c_\nu \varepsilon^2\big).
$$

Combining the estimates of Step 1 and Step 2, we arrive at \eqref{sigma_sigma}, up to re-denoting $C$.

\medskip

Proof of $5'$). 
Fix $t \in \mathbb R$. We have
\begin{align}
 |\sigma(t,x)-\sigma^\varepsilon(t,x)|^2
 &=\left|\int_{B_1}\eta(z)\bigl(\sigma(t,x)-\sigma(t,x-\varepsilon z)\bigr)\,dz\right|^2 \notag\\
 &\leq \int_{B_1}\eta(z) |\sigma(t,x)-\sigma(t,x-\varepsilon z)|^2\,dz,
 \label{mf}
\end{align}
where, in turn, 
 $\sigma(t,x)-\sigma(t,x-\varepsilon z)=\varepsilon\int_0^1
 (\nabla \sigma) (t, x- s \varepsilon z)\,z ds$.
Then
\begin{equation}
\label{mf2}
|\sigma(t,x)-\sigma(t,x-\varepsilon z)|^2 \leq \varepsilon^2|z|^2\int_0^1|\nabla \sigma(t,x-s\varepsilon z)|^2 ds,
\end{equation}
so, multiplying \eqref{mf} by $\rho$, integrating over $\mathbb R^d$ and using \eqref{mf2}, we obtain
\begin{align*}
 \|\sigma(t,\cdot)-\sigma^\varepsilon(t,\cdot)\|_{L_\rho^2(\mathbb R^d)}^2 \leq \varepsilon^2 \int_{B_1}\eta(z)|z|^2 \int_0^1\int_{\mathbb R^d}\rho(x)|\nabla \sigma(t,x-s\varepsilon z)|^2 dx ds dz.
\end{align*}
Making the change of variables $y=x-s \varepsilon z$ for fixed $z$, and using $|z|\leq 1$ and $0\leq s\varepsilon  \leq 1$, so that $|s \varepsilon z|\leq 1$, we have $\rho(y+s\varepsilon z)\leq C_d\rho(y)$. Hence we arrive at
\begin{align*}
\|\sigma(t,\cdot)-\sigma^\varepsilon(t,\cdot)\|_{L_\rho^2(\mathbb R^d)}^2 & \leq C_d\varepsilon^2\left(\int_{B_1}|z|^2\eta(z)\,dz\right)\int_{\mathbb R^d}\rho(y)|\nabla \sigma(t,y)|^2 dy \\
&\leq C\varepsilon^2\|\nabla \sigma(t,\cdot) \|_{L_\rho^2(\mathbb R^d)}^2.
\end{align*}
It remains to apply the form-boundedness \eqref{nabla_sigma} of $|\nabla\sigma(t,\cdot)|$ and use the fact that both $\rho$ and $|\nabla \rho|$ are square integrable, to arrive at the claimed inequality $\|\sigma(t,\cdot)-\sigma^\varepsilon(t,\cdot)\|_{L_\rho^2(\mathbb R^d)} \leq C_1\varepsilon$.

\medskip

Proof of $4')$. Note first that, given a function (or matrix field) $h$ on $\mathbb R^d$, we have
\begin{align*}
\|\eta_\varepsilon \ast h\|_{L^2_\rho(\mathbb R^d)}^2  & \leq C\int_{\mathbb R^d} \eta_\varepsilon(z)\int_{\mathbb R^d} \rho(x) |h(x-z)|^2dxdz \\
& = C\int_{\mathbb R^d} \eta_\varepsilon(z)\int_{\mathbb R^d} \rho(y+z) |h(y)|^2dydz \\
& (\text{use that if $|z| \leq 1$, then $\rho(y+z) \leq C_d\rho(y)$, $y \in \mathbb R^d$}) \\
& \leq C \|h\|_{L^2_\rho(\mathbb R^d)}^2.
\end{align*}
Applying this inequality to $h(\cdot)=\sigma(t,\cdot)-\sigma(s,\cdot)$, we see that the locally uniform H\"{o}lder continuity of $\sigma$ in time as a  $L^2_\rho$-valued function is preserved by the mollification in the spatial variables.

\medskip

Proof of $6'$). Fix $1 \leq k \leq d$. We need to establish convergence of $\nabla_k a^\varepsilon$ to $\nabla_k a$ in $L^2([t_0,t_1],L^2_{\rho}(\mathbb R^d))$. We assume that $t_0$, $t_1$ are fixed from now on. We have
\begin{equation*}
\nabla_k a=(\nabla_k\sigma)\sigma^\top + \sigma(\nabla_k\sigma)^\top,
\end{equation*}
Similarly,
\begin{equation*}
\nabla_k a^\varepsilon=(\nabla_k\sigma^\varepsilon)(\sigma^\varepsilon)^\top+\sigma^\varepsilon(\nabla_k\sigma^\varepsilon)^\top.
\end{equation*}
Subtracting these two identities from each other, we obtain
\begin{align*}
\nabla_k a^\varepsilon-\nabla_k a= & (\nabla_k\sigma^\varepsilon-\nabla_k\sigma)(\sigma^\varepsilon)^\top +(\nabla_k\sigma)(\sigma^\varepsilon-\sigma)^\top \\
&+(\sigma^\varepsilon-\sigma)(\nabla_k\sigma)^\top + \sigma^\varepsilon(\nabla_k\sigma^\varepsilon-\nabla_k\sigma)^\top.
\end{align*}
Therefore, using $\|\sigma^\varepsilon\|_\infty \leq \|\sigma\|_\infty$, we obtain
\begin{align*}
\|\nabla_k a^\varepsilon-\nabla_k a\|_{L^2([t_0,t_1],L^2_\rho)} \leq 2\|\sigma\|_\infty
\|\nabla_k\sigma^\varepsilon-\nabla_k\sigma\|_{L^2([t_0,t_1], L^2_\rho)} + 2\bigl\||\nabla_k\sigma|\,|\sigma^\varepsilon-\sigma|\bigr\|_{L^2([t_0,t_1],L^2_\rho)}.
\end{align*}
We can show that the first term in the RHS goes to zero as $\varepsilon \downarrow 0$ by adapting the proof of 2) in Lemma \ref{lem_b}. (Namely, if a sequence of uniformly in $\varepsilon$ form-bounded vector fields converges locally in $L^2$ -- which, in our case, in a standard property of mollifiers -- then we can upgrade this convergence to the global convergence with respect to weight $\rho$.) The second term goes to zero by $\sigma^\varepsilon(t,\cdot) \rightarrow \sigma(t,\cdot)$ a.e.\,on $\mathbb R^d$ (this is a standard property of mollifiers) for a.e.\,$t \in \mathbb R$, upon applying the Dominated convergence theorem.
\end{proof}

\bigskip

\section{Proof that $J$ is invertible (in the proof of the stochastic Gronwall claim \ref{claim_gronwall})}
\label{J_app}

Step 1.~Let us make a preliminary observation regarding evaluating $f(J_t)$, where $f$ is a $C^2$ function of $d^2$ variables. The matrix identity $dJ_t=\sum_{l=1}^d A_t^l J_t dW_t^l$ yields scalar identities
\begin{equation}
\label{J_ij}
dJ^{ij}_t=\sum_{l=1}^d (A_t^l J_t)^{ij} dW_t^l
\end{equation}
The It\^{o}'s formula therefore yields
$$
df(J_t)= \sum_{i,j=1}^d \partial_{ij} f(J_t)dJ^{ij}_t + \frac{1}{2}\sum_{i,j,k,r=1}^d \partial_{ij} \partial_{kr} f(J_t)d[J^{ij},J^{kr}]_t,
$$
where the square brackets denote the cross-variation. Now, invoking identity \eqref{J_ij} and its immediate consequence $d[J^{ij},J^{kr}]_t=\sum_{l=1}^d (A_t^l J_t)^{ij}(A_t^l J_t)^{kr}dt$, we obtain
$$
df(J_t)=\sum_{l=1}^d \sum_{i,j=1}^d \partial_{ij} f(J_t) (A_t^l J_t)^{ij} dW_t^l + \frac{1}{2} \sum_{l=1}^d \sum_{i,j,k,r=1}^d \partial_{ij} \partial_{kr} f(J_t)(A_t^l J_t)^{ij}(A_t^l J_t)^{kr}dt.
$$
It is convenient to rewrite the right-hand side in compact form in terms of the directional derivatives of $f$ ($\nabla$ is the gradient on $\mathbb R^{d \times d}$, so the direction is a $d \times d$ matrix):
\begin{equation}
\label{prelim}
df(J_t)=\sum_{l=1}^d \nabla_{A_t^l J_t} f(J_t) dW_t^l + \frac{1}{2} \sum_{l=1}^d \nabla^2_{A_t^lJ_t, A_t^lJ_t} f (J_t)dt.
\end{equation}

\medskip

Step 2.~Now, armed with identity \eqref{prelim}, we take $f(X):=\det X$. Our goal is to obtain, using \eqref{prelim}, the SDE satisfied by $\det(J_t)$. The use of \eqref{prelim}  thus requires evaluating directional derivatives of the determinant. These are given by well-known formulas obtained using Taylor polynomial expansion of $\det X$ as a function of $X \in \mathbb R^{d \times d}$: 
$$
\nabla_{B X} \det(X)=\det (X) {\rm tr} B, \qquad \nabla^2_{BX,BX} \det (X)=\det X \bigl(({\rm tr\,}B)^2 - {\rm tr\,}B^2 \bigr).
$$
Importantly, there is no $X$ under the trace sign: we arrive at the following \textit{linear} SDE for $D_t:=\det J_t$:
$$
dD_t=D_t \sum_{l=1}^d {\rm tr}(A_t^l)dW_t^l + D_t \frac{1}{2}\sum_{l=1}^d \bigl(({\rm tr\,}(A_t^l))^2 - {\rm tr\,}(A_t^l)^2 \bigr)dt, \quad D_s=1.
$$
We can solve it:
$$
\det J_t = \exp \left( \sum_{l=1}^{d}\int_s^t \mathrm{tr}(A_r^l) dW_r^l -\frac{1}{2}\sum_{l=1}^{d}\int_s^t \mathrm{tr} \bigl((A_r^l)^2\bigr) dr \right).
$$
Hence $\det J_t$ never vanishes.

\bigskip

\section{Proof of Claim \ref{claim_com}}
\label{claim_com_proof}

\begin{proof}[Proof of Claim \ref{claim_com}]
Fix $y\in\mathbb R^d$ and $0\leq r<t\leq T$, and set
$U:=B_1(y)$, $F_m(x):=X_{r,t}^{x,m}$. The following argument works componentwise, so Lemma \ref{lem_compact} applies equally to these
vector-valued random fields. Recall that
\[
X_{r,t}^{x,m}
=
x+\int_r^t\sigma_m(\tau,X_{r,\tau}^{x,m})\,dW_\tau,
\qquad
\sup_m\|\sigma_m\|_\infty\leq M.
\]
We apply Proposition \ref{prop2} with $p=2$ to $X_{r,\cdot}^{x,m}$ -- crucially, 
the constants $K_1$, $K_2$ in this proposition are independent not only of $m$ and the centre $y$, but also of the initial time $r$.

\subsubsection*{Verification of \eqref{a.1}}
By It\^{o}'s isometry,
\[
\mathbf E|F_m(x)|^2
\leq
2|x|^2+
2\mathbf E\int_r^t
|\sigma_m(\tau,X_{r,\tau}^{x,m})|^2\,d\tau
\leq 2|x|^2+2M^2T.
\]
Since $y$ is fixed,
\begin{equation}
\label{claim_F_bd}
\sup_m\mathbf E\int_U|F_m(x)|^2\,dx
\leq C_{y,T}.
\end{equation}
Proposition \ref{prop2}(\textit{i}) and the Cauchy-Schwarz inequality yield
\[
\mathbf E\int_U
|\nabla_x X_{r,t}^{x,m}-I|^2\,dx
\leq
|U|^{1/2}
\|\nabla_xX_{r,t}^{\cdot,m}-I\|_
 {L^4(U,L^2(\Omega))}^2
\leq C|t-r|^{1/4}.
\]
Hence
\begin{equation}
\label{claim_grad_bd}
\sup_m\mathbf E\int_U|\nabla_xF_m(x)|^2\,dx
\leq C_T.
\end{equation}
Now, \eqref{claim_F_bd} and
\eqref{claim_grad_bd} yield \eqref{a.1}.

\subsubsection*{Verification of \eqref{a.2}}
Let us write in what follows
 $$G_m(s,x):=D_sF_m(x)\;\;\big(=D_sX_{r,t}^{x,m}\big).$$
Since $X_{r,t}^{x,m}$ depends only on the increments of the Brownian motion on
$[r,t]$,
\begin{equation}
\label{claim_zero}
G_m(s,x)=0
\qquad\text{for a.e. }s\notin[r,t].
\end{equation}
For $s\in[r,t]$, Proposition \ref{prop2}(\textit{ii}) gives
\[
\|G_m(s,\cdot)-\sigma_m(s,X_{r,s}^{\cdot,m})\|_{L^4(U,L^2(\Omega))}\leq C|t-s|^{1/8}.
\]
Since
$
\|\sigma_m(s,X_{r,s}^{\cdot,m})\|_
 {L^4(U,L^2(\Omega))}
\leq M|U|^{1/4},
$
we thus obtain
\begin{equation}
\label{G_bd}
\sup_m \sup_{s\in[r,t]}
\|G_m(s,\cdot)\|_{L^4(U,L^2(\Omega))}
\leq C_T.
\end{equation}
Hence, \eqref{claim_zero} and Cauchy-Schwarz, yield
$$
\sup_m
\mathbf E\int_U\int_0^T|D_sF_m(x)|^2 dsdx=
\sup_m\int_r^t\mathbf E\int_U|G_m(s,x)|^2 dx ds
\leq C_T
$$
$\Rightarrow$ \eqref{a.2}.

\subsubsection*{Verification of \eqref{a.3}}

Case 1: Let $s,s'\in[r,t]$. Then, by Proposition \ref{prop2}(\textit{iii}),
$$
\|G_m(s,\cdot)-G_m(s',\cdot)\|_
 {L^4(U,L^2(\Omega))}
\leq
K_2\left(
|s-s'|^{1/16}+|s-s'|^{\gamma/4}
\right).
$$
Put
$\alpha:=\frac{1}{16} \wedge \frac{\gamma}{4}$
and select some
$0<\beta<\alpha$.
Since we work on a finite time interval, the previous estimate yields
$
\|G_m(s,\cdot)-G_m(s',\cdot)\|_
 {L^4(U,L^2(\Omega))}
\leq C_T|s-s'|^\alpha.
$
Thus, using the Cauchy-Schwarz inequality, we have
\[
\mathbf E\int_U
|G_m(s,x)-G_m(s',x)|^2\,dx
\leq C_T|s-s'|^{2\alpha}.
\]
It follows that
\begin{align*}
&\mathbf E\int_U\int_r^t\int_r^t
\frac{|G_m(s,x)-G_m(s',x)|^2}
{|s-s'|^{1+2\beta}}
 ds ds' dx \notag\\
&\leq
C_T\int_r^t\int_r^t
|s-s'|^{-1+2(\alpha-\beta)}
ds ds'
<\infty 
\end{align*}
since $\beta<\alpha$.

Case 2: Let $s\in [r,t]$ and $s' \in [0,T] \setminus [r,t]$. Then $G_m(s',\cdot)=0$. Therefore, using \eqref{G_bd}, we can estimate
\begin{align*}
\mathbf E\int_U\int_{s \in [r,t], s' \in [0,T] \setminus [r,t]} & \frac{|G_m(s,x)-G_m(s',x)|^2}
{|s-s'|^{1+2\beta}}
 ds ds' dx \\
& \leq C_T \int_{r}^t \biggl( \int_0^r \frac{ds'}{(s-s')^{1+2\beta}} +\int_t^T \frac{ds'}{(s'-s)^{1+2\beta}} \biggr)ds \\
 & (\text{discard two negative terms coming from the endpoints $s'=0$ and $s'=T$}) \\
 & \leq C'_{T,\beta}\int_r^t
\left[
(s-r)^{-2\beta}+(t-s)^{-2\beta}
\right]ds\leq
C_{T,\beta}(t-r)^{1-2\beta},
\end{align*}
where in the last integration we have used $\beta<\frac{1}{2}$.

Case 3: The case $s' \in [r,t]$ and $s \in [0,T] \setminus [r,t]$ gives the same contribution in the right-hand side of \eqref{a.3} as the Case 2.

Case 4: If both $s,s' \not \in [r,t]$, then  $G_m(s',\cdot)=G_m(s,\cdot)=0$, so the contribution to the right-hand side of \eqref{a.3} is zero.

Cases 1-4 yield \eqref{a.3}. This ends the proof of Claim \ref{claim_com}.
\end{proof}

\bigskip

\section{Kolmogorov–Chentsov theorems}
\label{kunita_app}

For the reader's convenience, we state below Kunita's Kolmogorov–Chentsov-type theorems \cite{Ku} that we have used in the proof of Theorem \ref{thm1}(\textit{i}).

\begin{theorem}[{\cite[Theorem 1.4.1]{Ku}}]
Let $X(x)$, $x\in D$, be a random field with values in a Banach space $B$, where $D$ is a domain in $\mathbb R^d$.
Assume that there exist positive constants
$\gamma$, $C$, and $\alpha_i$, $i=1,\ldots,d$ with $ \sum_{i=1}^{d}\alpha_i^{-1}<1$ satisfying 
\begin{equation*}
\mathbf{E}\!\left[ \left\|X(x)-X(y)\right\|^\gamma \right] \leq C\left(   \sum_{i=1}^{d}  |x^i-y^i|^{\alpha_i}\right),
\quad x,y\in D.
\end{equation*}
where, $\|\cdot\|$ stands for the norm of the Banach space and $ x=(x^1,\ldots,x^d)$, $y=(y^1,\ldots,y^d)$.Then $X$ has a continuous modification.
\end{theorem}

\begin{theorem}[{\cite[Theorem 1.4.7]{Ku}}]
Let $ \{X_n(x) \mid x \in I = [0,1]^d \}$ be a sequence of continuous random fields with values in a real separable
semi-reflexive Fr\'echet space $S$ with seminorms $\|\cdot\|_N$, $N=1,2,\ldots$. Assume that, for each $N$, there exist
positive constants $\gamma$, $C$, and $\alpha_1,\ldots,\alpha_d$, with $ \sum_{i=1}^{d}\alpha_i^{-1}<1$,
such that
\begin{equation*}
\mathbf{E}\left[ \left\|X_n(x)-X_n(y)\right\|_N^\gamma \right] \leq C\left( \sum_{i=1}^{d}|x^i-y^i|^{\alpha_i} \right),
\quad x,y\in I,
\end{equation*}
and
\begin{equation*}
\mathbf{E} \left[ \left\|X_n(x)\right\|_N^\gamma \right] \leq C, \quad x\in I,
\end{equation*}
for every $n$. Then $\{X_n\}$ is tight with respect to the semi-weak topology of $C(I,S)$, i.e.\,the topology of the uniform convergence on $I$ with respect to the weak topology of $S$.
\end{theorem}

\bigskip

\section*{AI usage disclosure}

The authors made use of OpenAI ChatGPT 5.6 for mathematical discussions related to the proof of Proposition \ref{prop2}. We also used ChatGPT for editorial assistance. The text of the article -- including all em dashes -- is written in its entirety by the authors.

\medskip

\end{document}